\documentclass[11pt, reqno]{article}
\usepackage{graphicx,psfrag,amssymb,amsmath,amsthm}
\usepackage[bookmarks,bookmarksnumbered]{hyperref}

\usepackage{color, xcolor}
\usepackage{psfrag, comment}
\usepackage{graphicx}
\usepackage{pdfpages}
\usepackage{epstopdf}

\AppendGraphicsExtensions{.pdf}
\usepackage{amscd, amsfonts, tikz,hyperref}
   \usepackage{eucal}
   \usepackage{graphicx,psfrag,amssymb,amsmath,amsthm, comment, enumerate}
\usepackage[T1]{fontenc}
\usepackage[utf8]{inputenc}
\usepackage{authblk}
\usepackage{yhmath}
\usepackage{pdfsync}
\usepackage{color}
\usepackage{enumitem}

\usepackage{pgf,tikz}
\usepackage{mathrsfs}
\usetikzlibrary{arrows}

\newtheorem{theorem}{Theorem}
\newcounter{countp}

\newtheorem*{SqRtheorem*}{Squash Rigidity Theorem}
\newtheorem*{StRtheorem*}{Stadium Rigidity Theorem}
\newtheorem{mtheorem}{Theorem}[section]
\newtheorem{lemma}[mtheorem]{Lemma}

\newtheorem{proposition}[mtheorem]{Proposition}

\newtheorem{remark}[mtheorem]{Remark}

\def\beq{\begin{equation}}
\def\eeq{\end{equation}}
\def\beqn{\begin{equation*}}
\def\eeqn{\end{equation*}}

\def\diam{{\rm diam}}

\def\bi{\mathbf{i}}

\def\bn{\mathbf{n}}

\def\cA{\mathcal{A}}
\def\cB{\mathcal{B}}
\def\cC{\mathcal{C}}
\def\cD{\mathcal{D}}
\def\cE{\mathcal{E}}

\def\cI{\mathcal{I}}
\def\cJ{\mathcal{J}}
\def\cK{\mathcal{K}}
\def\cL{\mathcal{L}}
\def\cM{\mathcal{M}}
\def\cN{\mathcal{N}}
\def\cO{\mathcal{O}}

\def\cQ{\mathcal{Q}}
\def\cR{\mathcal{R}}
\def\cS{\mathcal{S}}
\def\cT{\mathcal{T}}

\def\IN{{\mathbb N}}

\def\IQ{{\mathbb Q}}
\def\IR{{\mathbb R}}

\def\IZ{{\mathbb Z}}

\def\hOmega{\hat{\Omega}}

\newcommand{\eps}{\varepsilon}

\def\p{\partial}

\def \ML{\CMcal{ML}}

\def \CcV{\CMcal{V}}

\def \CcC{\CMcal{C}}

\def \hOmega{\widehat{\Omega}}

\def \hgamma{\widehat{\gamma}}
\def \hs{\hat{s}}

\def \ogamma{\overline{\gamma}}
\def \os{\overline{s}}
\def \ot{\overline{t}}
\def \ovarphi{\overline{\varphi}}
\def \hA{\widehat{A}}
\def \hB{\widehat{B}}
\def \hC{\widehat{C}}
\def \hx{\widehat{x}}
\def \hy{\widehat{y}}

\def \hl{\hat{l}}
\def \hcQ{\widehat{\cQ}}
\def \htt{\hat{t}}

\def \hC{\widehat{C}}

\numberwithin{equation}{section}

\begin{document}

\title{Length Spectrum Rigidity for Piecewise Analytic Bunimovich Billiards}

\author[1]{Jianyu Chen\thanks{jychen@suda.edu.cn}}
\author[2]{Vadim Kaloshin\thanks{vadim.kaloshin@gmail.com}}
\author[3]{Hong-Kun Zhang\thanks{hongkun@math.umass.edu}}
\affil[1]{\footnotesize School of Mathematical Sciences
\& Center for Dynamical Systems and Differential Equations, Soochow University, China}
\affil[2]{\footnotesize IST Austria}
\affil[3]{\footnotesize Department of Mathematics \& Statistics,  University of Massachusetts Amherst, USA}

\date{}

\maketitle

\begin{abstract}
In the paper, we establish Squash Rigidity Theorem - 
the dynamical spectral rigidity for
piecewise analytic Bunimovich squash-type stadia 
whose convex arcs are homothetic.
We also establish Stadium Rigidity Theorem - 
the dynamical spectral rigidity for
piecewise analytic Bunimovich stadia whose flat boundaries are 
a priori fixed. 
 
In addition, for smooth Bunimovich  squash-type stadia
we compute the Lyapunov exponents along the maximal period
two orbit, as well as the value of the Peierls' Barrier function from
the maximal marked length spectrum associated
to the rotation number $\frac{2n}{4n+1}$.
\end{abstract}
\tableofcontents
\section{Introduction}

\subsection{Background and Notations}

{
A natural question is to understand what information on
the geometry of the billiard table is encoded by {\it the length spectrum},
i.e., the set of lengths of periodic orbits.
Motivated by the famous question of M. Kac \cite{Kac66}:
``Can one hear the shape of a drum?'', which is formally called \emph{Laplace inverse spectral problem}, 
we propose the following question from the perspective of billiard dynamics:
is the knowledge of length spectrum sufficient
to reconstruct the shape of the billiard table and hence the whole dynamics?
We refer to this problem as the \emph{dynamical inverse spectral problem}.

Both inverse spectral problems turn out to be extremely challenging,
and only little progress have been achieved for some classes of convex billiards.
On the one hand, the celebrated work 
\cite{Zel00, Zel04, Zel09} by Zelditch shows that the Laplace inverse spectral problem has a positive answer 
in the case on a generic class of analytic $\mathbb{Z}_2$-symmetric planar strictly convex domains.
Hezari-Zelditch \cite{HeZel10} have obtained a higher dimensional analogue of this result. 
Very recently, Hezari-Zelditch~\cite{HeZel19} showed that 
nearly circular ellipses are spectrally determined among all smooth domains,
without assuming  any symmetry, convexity, or closeness to the ellipse, on the class of domains.
On the other hand, Colin de Verdi\`ere \cite{Colin84} had shown that the marked length spectrum 
determines completely the geometry of convex analytic billiards which have the symmetries of an ellipse. 
In the non-analytic situation, 
De Simoi, Kaloshin and Wei \cite{dSKW17} have proven the length spectral rigidity, i.e., in the class of sufficiently smooth 
$\mathbb{Z}_2$-symmetric strictly convex table sufficiently close to a circle all deformations preserving 
the length spectrum are isometries. However, there are a number of counter-examples to the inverse 
spectral problem (see. e.g. \cite{GWW92,Sun85}), while the billiard domains in these examples are neither smooth 
nor strictly convex. To this end, great interest has been raised to see if the inverse spectral  problem 
holds for a certain family of non-smooth non-convex billiards.

In this paper, we shall consider a  class of Bunimovich billiard tables, which are not smooth at several boundary points.
Moreover, these tables are  not strictly convex. We would like to stress the dynamics on the Bunimovich billiards 
is significantly different from the elliptic dynamics on strictly convex billiards, that is, the billiard ball motion in 
Bunimovich tables exhibits hyperbolic behavior, accompanied with  strong chaotic behavior and also have singularities. 
Here, we are able to obtain the spectral rigidity results for the first class of  hyperbolic billiards with singularities. It is 
worth mentioning some recent results in \cite{BdKL18, dSKL19}, where marked length spectral rigidity is shown for 
some open sets of hyperbolic billiards, whose dynamics are uniformly hyperbolic and can be coded by subshifts of 
finite type. Nevertheless, the Bunimovich billiards are non-uniformly hyperbolic and cannot be conjugate to symbolic 
systems on finite alphabet set. It was recently shown in \cite{CWZ19} the induced systems of Bunimovich
billiards are conjugate to a positive recurrent countable Markov shift with respect to the SRB measure. Hence the methods in
\cite{BdKL18, dSKL19} cannot be directly applied for our setting. 

In what follows, we describe the class of Bunimovich billiards and their dynamics.}
More precisely,
we investigate two classes of billiards tables satisfying the following assumptions: \\

\noindent \textbf{Assumption I}
\begin{itemize} 
\item[($\mathrm{I_s}$)] {\it A Bunimovich stadium} $\Omega$ is a domain
whose boundary $\p\Omega$ is made of two $C^3$ strictly
convex arcs $\Gamma_1$ and $\Gamma_2$, as well as two
flat parallel boundaries $\Gamma_3$ and $\Gamma_4$,
which are two opposite sides of a rectangle
(see Fig.~\ref{fig:stadia}, left).

\item[($\mathrm{I_{ss}}$)] {\it A Bunimovich squash-type stadium} $\Omega$ is
a domain whose boundary $\p\Omega$ is made of two $C^3$
strictly convex arcs $\Gamma_1$ and $\Gamma_2$, as well as
two flat boundaries $\Gamma_3$ and $\Gamma_4$, which may
not be parallel (see Fig.~\ref{fig:stadia}, middle).
\end{itemize}

\noindent \textbf{Assumption II}

In both cases ($\mathrm{I_s}$) and ($\mathrm{I_{ss}}$), we require that $\p\Omega$ is $C^1$ but not
$C^2$ smooth at each gluing point $\Gamma_i\cap \Gamma_j$,
where $i=1, 2$ and $j=3, 4$. \\

\noindent \textbf{Assumption III}

\begin{itemize} 
\item[($\mathrm{III_s}$)]
A Bunimovich stadium $\Omega$
satisfies {\it the defocusing mechanism},\footnote{
In fact, for all the results in this paper, we only need (1) the uniqueness of maximal period two orbit
and the shadowing orbits; (2) these
orbits are hyperbolic. The defocusing mechanism is just a sufficient condition, which is quite strong
but somehow easy to check using elementary geometry.
} i.e.,
for any $P_1\in \Gamma_1$ and $P_2\in \Gamma_2$,
\beq\label{def defocusing}
\left|\overline{P_1P_2}\right|>\max\left\{\left|\overline{P_1Q_1}\right|, \ \left|\overline{P_2Q_2}\right| \right\},
\eeq
where $Q_i$ is the other intersection point between the line passing through
$ P_1, P_2$ and the osculating circle of $\Gamma_i$ at $P_i$, $i=1, 2$ (see Fig.~\ref{fig:stadia}, left).

\item[($\mathrm{III_{ss}}$)]
For a Bunimovich squash-type stadiam $\Omega$,
let $\widetilde\Omega$ be the double cover table by
attaching a symmetric copy to $\Omega$ along $\Gamma_3$ or $\Gamma_4$,
and let $\widetilde\Gamma_1$ and $\widetilde\Gamma_2$ be the
two new arcs of $\widetilde\Omega$.
A slightly stronger condition is required for a Bunimovich squash-type stadium $\Omega$:
it satisfies {\it the doubly defocusing mechanism}, that is,
\eqref{def defocusing} holds for any $P_1\in \Gamma_1\cup \widetilde{\Gamma}_1$
and $P_2\in \Gamma_2\cup \widetilde\Gamma_2$ (see Fig.~\ref{fig:stadia}, middle and right).

\end{itemize}

\begin{figure}[h]
\begin{center}
\includegraphics[width=\textwidth]{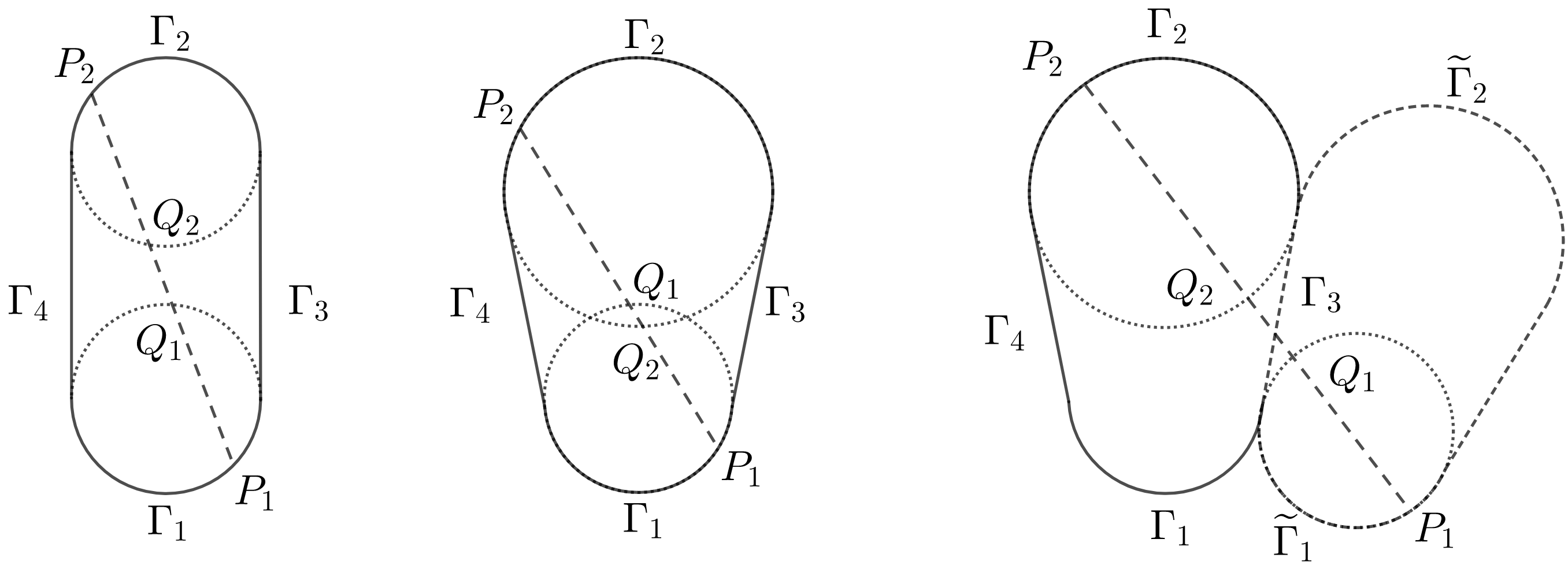}
\end{center}
\caption{Bunimovich (squash-type) stadia and the (doubly) defocusing}
\label{fig:stadia}
\end{figure}

Note that the class of the Bunimovich (squash-type) stadia is
a generalization of the standard Bunimovich (squash) stadia, which is
formed by circular arcs $\Gamma_1$ and $\Gamma_2$.
We remark that the mechanism of (doubly) defocusing is robust
for Bunimovich (squash-type) stadia
under $C^3$ perturbations of $\Gamma_1$ and $\Gamma_2$.

It will be shown in Section~\ref{sec: period two}
that under Assumptions  ($\mathrm{I}$)-($\mathrm{III}$), 
any Bunimovich (squash-type) stadium $\Omega$ possesses a 
unique maximal period two orbit $\gamma^*=\overline{AB}$,
where $A\in \Gamma_1$ and $B\in \Gamma_2$.
The following condition is further assumed for the Bunimovich squash-type stadia.
\\

\noindent \textbf{Assumption IV}

Let $\chi$ be a positive constant.
A Bunimovich squash-type stadium $\Omega$ is said to be \emph{$\chi$-homethetic} 
near the period two orbit $\gamma^*=\overline{AB}$
if there exists an orientation preserving 
linear transformation $\cS: \IR^2\to \IR^2$ 
such that 
\begin{itemize}
\item[(1)] $\cS$ is a homothety with ratio $\chi$, i.e.,
\beqn
\left|\cS(v_1) - \cS(v_2)\right|=\chi |v_1-v_2| \ \ \text{for any}\ v_1, v_2\in \IR^2.
\eeqn
\item[(2)]
$\cS(A)=B$ and $\cS\left(\widehat{\Gamma}_1\right)= \widehat{\Gamma}_2$,
where  
$\widehat{\Gamma}_1$ is a sub-curve of $\Gamma_1$ containing $A$
and 
$\widehat{\Gamma}_2$ is a sub-curve of $\Gamma_2$ containing $B$.
\end{itemize}

We stress that Assumption (IV) will only be imposed for the Bunimovich squash-type stadia.
We shall call the constant $\chi$ the homothety ratio. 
Note that if $\chi=1$, then the two arcs $\Gamma_1$ and $\Gamma_2$ are locally isometric
near $\gamma^*=\overline{AB}$. 

\bigskip

To describe the billiard dynamics on the table $\Omega$,
we assume the billiard ball moves at a unit speed,
and the boundary $\partial \Omega$ is oriented in
the counter-clockwise direction. 
Set $R=\left| \partial \Omega \right|$. 
The phase space is a cylinder given by
\beq\label{phase space}
M:=\left\{(r,\varphi)\,:\,  r\in [0, R]/\{0\sim R\},
\ \varphi\in[-\pi/2,\pi/2] \right\},
\eeq
where $r$ is an arclength parameter of $\p\Omega$ and $\varphi$
is the angle formed by the collision vector and the inward normal
vector of the boundary. We denote by $\tau(z, z_1)$ the length of
the free path of a billiard trajectory connecting $z=(r,\varphi)$
and $z_1=(r_1, \varphi_1)$ in $M$, and we also denote by
$$
F:M\to M, \qquad F:(r,\varphi)\to (r_1,\varphi_1)
$$ the associated billiard map.
\medskip

Given any $q$-periodic billiard orbit
$\gamma=\overline{z_1 z_2 \dots z_q} $, i.e.,
$Fz_i=z_{i+1}$ for $1\le i\le q-1$ and $Fz_q=z_1$,
we set $z_k=z_i$ if $k\equiv i \pmod q$.
The total length for the periodic orbit $\gamma$ is given by
\beq\label{generating}
L(\gamma):=\sum_{i=1}^q \tau (z_i, z_{i+1}).
\eeq
The winding number $p$ of a $q$-periodic orbit $\gamma$
measures how many times the orbit $\gamma$ goes around
$\p \Omega$ along the counter-clockwise direction until it
comes back to the starting point. The rotation number of
a $q$-periodic orbit $\gamma$ is given by $\rho(\gamma):=p/q$,
where $p\ge 1$ is the winding number\ of $\gamma$.
Due to time reversibility, we study only  $p/q\in \IQ\cap (0,1/2]$.
We denote the set of periodic orbits
of rotation number $p/q$, by $\Gamma_{p/q}$.
\medskip

We introduce {\it  the length spectrum} of a billiard table $\Omega$ as the set
of lengths of all periodic orbits, counted with multiplicity:
\beqn
\cL(\Omega):=\IN\cdot \{L(\gamma)\,|\, \gamma \text{ is a periodic billiard orbit }\}\
\cup \ \IN\cdot \{|\partial \Omega|\}.
\eeqn
One difficulty working with the length spectrum $\cL(\Omega)$
is that its (length) elements have no labels, e.g. rotation numbers of
the associate periodic orbits. One possibility is to consider the so-called
{\it maximal marked length spectrum} as in \cite{GuiKaz80} (see also
\cite{Sib04} and \cite{PS92}),
by associating to each length the corresponding rotation number.
More precisely, we consider a map
$
\ML_{\Omega}^{\max}: \mathbb{Q}\cap(0,1/2]\to\mathbb{R}^+
$
such that for any $p/q\in\mathbb{Q}\cap(0, 1/2]$ in lowest terms,
\beq \label{eq:MLS}
\ML_{\Omega}^{\max}(p/q)=\max\{L(\gamma)|\, \, \gamma\in\Gamma_{p/q}\}.
\eeq
We say that a periodic orbit $\gamma$ is {\it maximal} of rotation number $p/q$
if $\rho(\gamma)=p/q$ and $L(\gamma)=\ML^{\max}_{\Omega}(p/q)$.

\subsection{Motivation and the Main Results}

A natural question is

\begin{center}
{\it If two Bunimovich (squash-type) stadia have the same (marked)
length spectra, are these two tables isometric?}
\end{center}

In the case of geodesic flows on hyperbolic surfaces
(Riemannian surfaces of negative curvature) the affirmative
answer was obtained independently by Otal \cite{Otal90} and
Croke \cite{Cr90}. 
Later on,
Croke-Sharafutdinov~\cite{CrSh98} proved the Laplace spectral rigidity of compact negatively curved manifolds, and later Guillarmou-Lefeuvre~\cite{GuLe19} proved a local version of the marked length spectral rigidity for Anosov manifolds. 
It is well-known that geodesic
flows on hyperbolic surfaces is uniformly hyperbolic and, as
the result, has {\it strong chaotic properties,} e.g. the number
of periodic orbits of period up to $T$ growth exponentially
with $T$. Bunimovich squash-type stadia also represent
billiards with {\it strongly chaotic properties} and is an analog
of geodesic flows on hyperbolic surfaces. In this paper we obtain
the geometric information about the underlying stadium from its (Marked)
Length Spectrum. In particular, our results only depend on the length spectrum of periodic orbits near a period two orbit, see below for details.

\subsubsection{Dynamical spectral rigidity for piecewise
analytic domains}

Our first main result concerns the dynamical length spectrum
rigidity for the Bunimovich (squash-type) stadia. \medskip

Let $\cM$ be a space of domains and
$\left\{\Omega_\mu\right\}_{|\mu|\le 1}$ be a $C^1$ one-parameter
family in $\cM$. The family $\left\{\Omega_\mu\right\}_{|\mu|\le 1}$
is called {\it dynamically isospectral} if the length spectra are
identical for each $\mu$, i.e., $\cL(\Omega_{0})=\cL(\Omega_{\mu})$
for any $\mu\in [-1, 1]$. A domain $\Omega\in \cM$ is
{\it dynamically spectrally rigid} in $\cM$ if for any $C^1$
one-parameter family $\left\{\Omega_\mu\right\}_{|\mu|\le 1}$ in
$\cM$ with $\Omega_0=\Omega$
we have
\beqn
\left\{\Omega_\mu\right\}_{|\mu|\le 1}\ \text{ is \ dynamically\ isospectral }
\Longrightarrow \ \Omega_\mu\cong \Omega \ \text{for any}\ \mu \in [-1, 1].
\eeqn
Here $\Omega_\mu\cong \Omega$ means that $\left\{\Omega_\mu\right\}_{|\mu|\le 1}$
is an isometric family, i.e., there exists a family $\left\{\cT_\mu\right\}_{|\mu|\le 1}$ of planar isometries
such that $\Omega_\mu= \cT_\mu\Omega$. 
We also say that  $\left\{\Omega_\mu\right\}_{|\mu|\le 1}$ is a constant family 
if $\Omega_\mu=  \Omega$ for all $\mu \in [-1, 1]$. \\

\medskip

Let $\cM^\omega_{ss}$ (resp. $\cM^m_{ss}$ with $m\ge 3$) be the space of Bunimovich
squash-type stadia $\Omega$ such that $\Omega$ satisfies Assumptions $\mathrm{( I_{ss})(II)(III_{ss})}$
and the convex arcs $\Gamma_1$ and $\Gamma_2$ are analytic (resp. $C^m$ smooth) curves.
Moreover, given $\chi>0$, 
we let $\cM^\omega_{ss}(\chi)$ (resp. $\cM^m_{ss}(\chi)$ with $m\ge 3$) 
be the subspace of $\cM^\omega_{ss}$ (resp. $\cM^m_{ss}$ with $m\ge 3$)
in which the Bunimovich squash-type stadia satisfy Assumption (IV) with homothety ratio $\chi$.

In the analytic space $\cM^\omega_{ss}$ 
it is somewhat unconventional to have all four parts
$\Gamma_j,\ j=1,2,3,4$, analytic and at the same time have
Assumption (II) which requires that at the gluing points the boundary is $C^1$,
but not $C^2$. This condition will be used in the proof of 
Lemma~\ref{lem: main n vanish 00}, which is crucial to the proof of the dynamical spectral rigidity
for the Bunimovich squash-type stadia. 

\medskip

Our first main result is the following.

\begin{SqRtheorem*}
For any $\chi>0$, 
a Bunimovich squash-type stadium $\Omega\in \cM^\omega_{ss}(\chi)$
is  dynamically spectrally rigid in $\cM^\omega_{ss}(\chi)$.
\end{SqRtheorem*}

Important progress have been recently made for spectral rigidity of convex billiard tables.
Our result is similar to \cite{dSKW17}, in which De Simoi, Kaloshin and Wei established
the dynamical spectral rigidity for a class of finitely smooth strictly convex $\IZ_2$
symmetric domains sufficiently close to the circle. On the other hand, the Laplacian spectral rigidity
was proved by Hezari and Zelditch \cite{HeZel12}  for a one-parameter $C^\infty$
domains that preserve the $\IZ_2\times \IZ_2$ symmetry group of ellipse.
For the same class of domains in \cite{dSKW17}, Hezari \cite{Hezari17} showed the
Laplacian spectral rigidity under the Robin boundary condition. 
A three  disk model with $\IZ_2 \times \IZ_2$-symmetry and analytic boundary was considered
in \cite{dSKL19}, which showed that Marked Length Spectrum uniquely determined the boundary.
This result is analogous to well-known results of Zeldich \cite{Zel04, Zel09}.

Note that for any Bunimovich squash-type stadium $\Omega\in \cM_{ss}^m$ with $m\ge 3$,
there is a unique maximal period two orbit $\gamma^*=\overline{AB}$ (see Section~\ref{sec: period two}).
In the proof of Squash Rigidity Theorem, we actually show the flatness of the deformation function $\bn$ (see \eqref{def deformation function} for the definition) at
the period two orbit $\gamma^*$, which holds  
for the normalized family of Bunimovich squash-type stadia
not only in $\cM^\omega_{ss}(\chi)$ but also in $\cM^\infty_{ss}(\chi)$
(see Proposition~\ref{lem van der}).

\begin{theorem}\label{thm: main1'}
For any $\Omega\in \cM^\infty_{ss}(\chi)$ with some $\chi>0$ 
and for any $C^1$ one-parameter normalized family of dynamically isospectral domains
$\left\{\Omega_\mu\right\}_{|\mu|\le 1}$ in
$\cM^\infty_{ss}(\chi)$ with $\Omega_0=\Omega$,  we have
\beqn
\bn^{(d)}(A)=\bn^{(d)}(B)=0, \ \ \text{for any} \ d\ge 0.
\eeqn
\end{theorem}

We remark that our Theorem~\ref{thm: main1'} implies that the deformation 
function $\bn\equiv 0$ in the analytic case. Observe that in Corollary 1 of
\cite{HeZel12}\footnote{The 
first derivative $\left.\frac{d}{d\epsilon}\right|_{\epsilon=0} \rho_\epsilon(x)$ in \cite{HeZel12} is
in fact the deformation function $\bn$ in our setting, under suitable parametrization. 
}, a flat dependence on $\mu$ is proven. In the analytic setting 
it leads to absence of non-trivial real analytic deformation. We note that 
a similar problem of analysing periodic orbits approximating a period two orbit 
was studied in \cite{St03}. However, the proof is different here, because of 
the existence of singularity for billiards.

We provide the brief ideas and main steps on how to obtain our 
Theorem~\ref{thm: main1'}, which asserts the Taylor coefficients of 
the deformation function $\bn$ are all vanishing at the period two orbit $\gamma^*$. 
\begin{itemize}
\item In Section \ref{sec: period two} 
we analyse the dynamical properties of the period two orbit $\gamma^*$, 
and then in Section \ref{sec:palindromic}, we construct a sequence of palindromic
periodic orbits $\gamma_n$ such that the limiting semi-orbit $\gamma_\infty$ is homoclinic to the period two orbit  $\gamma^*$. Using the special properties of 
the palindromic orbits, we obtain quantitative estimates for the coordinates 
of $\gamma_n$  near $\gamma^*$ when $k\approx n/2$ 
(see Lemma~\ref{lem: finer est} and Remark~\ref{rem: finer est ga n}), 
as well as estimates for the shadowing of $\gamma_n$ along the homoclinic 
semi-orbit $\gamma_\infty$ (see Lemma~\ref{lem: finer est'}) . 
\item In Section \ref{sec:LIO} we study the Linearized Isospectral Functionals 
related to the special orbits. Following the work of \cite{dSKW17}, we obtain 
the special function $G(r, \varphi)=\bn(r)\cos\varphi$ must have vanishing periodic 
data for a dynamical isospectral family.  In particular, the sum of $G$ over 
the special orbits $\gamma^*$ and $\gamma_n$ must be zero. 
\item In Section \ref{sec: proof of thm1} we study the sum of $G$ over 
the palindromic orbit $\gamma_n$ by separate it into two sums: one is a global sum $S_n^{global}(\ell)$ (with minus sign) away from the period two orbit $\gamma^*$, 
the other is the local sum $S_n^{local}(\ell)$ near $\gamma^*$. As the total sum 
is vanishing, we must have $S_n(\ell)=S_n^{global}(\ell)=S_n^{local}(\ell)$. In other 
words, we define a sum $S_n(\ell)$ over $\gamma_n$ which has two representations  
(see Subsection~\ref{sec: sum G}).
\begin{itemize}
\item For the global sum: by Lemma~\ref{lem: finer est'}, when $k\le \ell$, 
the points $x_{2\ell + j}(k)$ on the palindromic orbit $\gamma_{2\ell + j}(k)$ lie on 
an asymptotic line through $\gamma_{\infty}(k)$, which have  asymptotic geometric 
spacing of order $\lambda^{-j}$. Due to this observation, we can perform 
the Lagrange interpolation method, either the classical or the weighted one,
to make a linear cancellation for the global sum representation, i.e., we show that 
a linear combination of $\{S_{2\ell+j}\}_{0\le j\le m}$ must vanish up to a higher 
order (see Lemma~\ref{lem: cancel}).
\item For the local sum: we can write down the Taylor expansions for the local 
representation of $S_n(\ell)$. Suppose that the lowest non-vanishing degree of $\bn$ 
is $d$. Using the linear cancellation that we have obtained in Lemma~\ref{lem: cancel}, 
we further obtain a linear equation with two variables $\bn^{(d)}(A)$ and $\bn^{(d)}(B)$, 
which further implies that $\bn^{(d)}(A)=\bn^{(d)}(B)=0$ by Assumption (IV).
\end{itemize}
\end{itemize}

\medskip

Note that Squash Rigidity Theorem 
is then a direct consequence
of Theorem~\ref{thm: main1'}, due to the analyticity of boundary
of Bunimovich squash-type stadia in $\cM^\omega_{ss}$.

\bigskip

We also consider special Bunimovich stadia whose flat boundaries  
are a priori fixed. To be precise, let $\Gamma_3$ and $\Gamma_4$ be 
the opposite sides of a fixed rectangle.
We then denote by $\cM^\omega_{s, b}$ (resp. $\cM^m_{s, b}$ with $m\ge 3$)
the space of Bunimovich stadia $\Omega$ such that 
$\Omega$ satisfies Assumptions $\mathrm{(I_{s})(II)}$,
the convex arcs $\Gamma_1$ and $\Gamma_2$ are analytic curves (resp. $C^m$ curves),
and the flat boundaries are exactly $\Gamma_3$ and $\Gamma_4$. 
We remark that Assumption $\mathrm{(III_s)}$ - the defocusing mechanism 
and Assumption (IV) - the homothetic condition
are not required for the Bunimovich stadia in 
$\cM^\omega_{s, b}$ (resp. $\cM^m_{s, b}$ with $m\ge 3$).
Using the unfolding trick, we provide a simpler but quite different proof for the dynamical spectral rigidity
for  the class $\cM^\omega_{s, b}$. Namely,
our second main result is stated as follows.

\medskip
 
\begin{StRtheorem*}
A Bunimovich stadium $\Omega\in \cM^\omega_{s, b}$
is dynamically spectrally rigid in $\cM^\omega_{s, b}$.
\end{StRtheorem*}

The core in the proof of Stadium Rigidity Theorem
is again flatness of the deformation function, but at the four
gluing points, i.e. 
$
P_{ij}:=\Gamma_i\cap \Gamma_j, \ \ i=1, 2, \ j=3, 4.
$
 (see Proposition~\ref{prop: vanishing bn stad} for the precise statements).

\begin{theorem}\label{thm: main2'}
For any $\Omega\in \cM^\infty_{s, b}$ and any $C^1$
one-parameter family of dynamically isospectral domains
$\left\{\Omega_\mu\right\}_{|\mu|\le 1}$ in
$\cM^\infty_{s, b}$ with $\Omega_0=\Omega$, we have
\beqn
\bn^{(d)}(P)=0, \ \ \text{for any} \ d\ge 0  \
\text{ and any gluing point } P \ \text{of} \ \Omega.
\eeqn
Here $\bn^{(d)}(P)$ denotes the one-sided $d$-th order derivative defined
on the part of convex arc $\Gamma_1$ or $\Gamma_2$.
\end{theorem}

In the proof of Theorem~\ref{thm: main2'}, we shall construct some special orbits 
of induced period two and four (see Lemma~\ref{lem: per 2} and Lemma~\ref{lem: per 4}), 
which are sufficient to compute the Taylor expansion of
the deformation function $\bn$ near the gluing points. 
The arguments there do not require the hyperbolicity at all, which is why
we can drop the defocusing mechanism.

\subsubsection{Marked length spectrum} 

Recall that $\gamma^*=\overline{AB}$ is the maximal period two orbit,
which bounces between  $\Gamma_1$ and $\Gamma_2$.
It is clear that the rotation number of $\gamma^*$ is $\frac12$.
Denote the free path $\tau^*=\tau(A, B):=\left| \overline{AB} \right|$, and note that $\tau^*=\diam(\Omega)$.

Our third main result demonstrates the marked length spectrum
provides information about the Lyapunov exponents along the
maximal period two orbit $\gamma^*$.

\begin{theorem}\label{thm: main3}
Let $\Omega$ be a Bunimovich squash-type stadium in $\cM^m_{ss}$ for $m\ge 3$.
Then with notations \eqref{eq:MLS} the following limits exist:
\beq
\begin{split}
-B_{\frac12} &:=\lim_{n\to\infty}
\left[ \ML^{\max}_{\Omega} \left( \frac{2n}{2n+1} \right) - (2n+1) \tau^*  \right], \label{def B}
\\
-\log \lambda &:=\, \lim_{n\to\infty} \frac{1}{n}
\log \left|\ML^{\max}_{\Omega} \left( \frac{2n}{2n+1} \right)  -
(2n+1) \tau^* + B_{\frac12}\right|,  \nonumber
\\
D_{\frac12}&:= \lim_{n\to \infty} \lambda^{n}
\left|\ML^{\max}_{\Omega} \left( \frac{2n}{2n+1} \right) - (2n+1) \tau^* + B_{\frac12}\right|. \nonumber
\end{split}
\eeq
\end{theorem}

The above theorem for Bunimovich squash-type stadia is similar to
results for strictly convex billiards in \cite{HKS18}. A somewhat similar
computation has been done for dispersing billiards in \cite{BdKL18}.
In the Aubry-Mather theory, $B_{\frac12}$ is usually referred to
the Peierls' Barrier function evaluated on a certain homoclinic
orbit of $\gamma^*$, and $\lambda$ is the eigenvalue of
the linearization of the billiard map along $\gamma^*$.
In the proof of Theorem \ref{thm: main3}, we show that the quantity $B_{\frac12}$
is finite, and the convergence of \eqref{def B} is exponentially fast.

\bigskip

{\bf Plan of the paper:} \
The rest of the paper is organized as follows.
In Section \ref{sec:billiard-maps} we present
auxiliary facts about the billiard map and properties of
the billiard dynamics near the unique period two orbit.
In Section \ref{sec: period two} we study the billiard dynamics
in a neighborhood of the maximal period two orbit $\gamma^*$.
In Section \ref{sec:palindromic} we analyze the palindromic
periodic orbits approximating the period two orbit
$\gamma^*$.
In Section \ref{sec:LIO} we define linearized
isospectral functionals related to some special periodic orbits, 
whose properties are closely related to dynamical spectral rigidity.
In Section \ref{sec: proof of thm1} utilizing properties of
approximating palindromic periodic orbits we prove
Squash Rigidity Theorem. In Section \ref{sec: proof of thm2}
using a different approach we prove Stadium Rigidity Theorem.
Finally, in Section~\ref{sec8} we analyze the 
periodic orbits with rotation number $\pm\frac{n}{2n+1}$
and obtain shadowing estimates similar to Section \ref{sec:palindromic}.
Then
in Section \ref{sec: proof of thm3} we prove
Theorem \ref{thm: main3} about Lyapunov exponents of
the maximal period two orbit. 

\bigskip

{\bf Acknowledgement} VK acknowledges a partial support by the NSF grant DMS-1402164.
Discussions with Martin Leguil and Jacopo De Simoi were very useful. The authors would like to thank a referee for pointing out a mistake in a preliminary version of the paper. It led to a more restrictive (homothetic) rigidity for squash stadia.  
JC was partially supported by NSFC grant 12001392 and NSF of Jiangsu BK20200850.
H.-K. Zhang is partially supported by Simons Foundation Collaboration Grants
for Mathematicians (Grant No.706383).

\bigskip

\section{The Billiard Dynamics}
\label{sec:billiard-maps}

\subsection{The Billiard Map and Its Differential}\label{sec: collision}

Let $\Omega$ be a Bunimovich squash-type stadium.
We recall that the phase space $M$ has the form
$$
M=\left\{x=(r,\varphi)\,:\, 0\leq r \leq \left|\p \Omega\right|,\ \
\varphi \in \left[-\pi/2, \pi/2 \right] \right\}.
$$
and the billiard map $F: M\to M$ sends $z=(r, \varphi)$ to $z_1=(r_1, \varphi_1)$.
The free path between $z$ and $z_1$ is given by $\tau=\tau(z, z_1)=\left|\overline{PP_1}\right|$,
where $P$ and $P_1$ be the collision points at $\p\Omega$ corresponding to $z$ and $z_1$ respectively.
The derivative $D_z F$ is given by
\beq\label{der collision}
\begin{pmatrix}
dr_1 \\ d\varphi_1
\end{pmatrix}
=
\frac{-1}{\cos\varphi_1}
\begin{pmatrix}
\tau \cK + \cos \varphi  & \tau \\
\tau \cK \cK_1 + \cK \cos\varphi_1 +\cK_1 \cos\varphi &  \tau \cK_1 + \cos\varphi_1
\end{pmatrix}
\begin{pmatrix}
dr \\ d\varphi
\end{pmatrix},
\eeq
where
$\cK=\cK(z)$ and $\cK_1=\cK(z_1)$ are the signed curvature of $\p \Omega$ at $P$ and $P_1$ respectively
(see Section 2.11 in~\cite{CM06}).
In particular, $\cK(z)$ is negative if
$P$ belongs to the convex arcs $\Gamma_1\cup \Gamma_2$.

\subsection{Wave Fronts and Unstable Curves}

We recall some basic notions and formulae in \cite{CM06}.
Given a tangent vector $dz=(dr, d\varphi)\in T_z M$,
we denote by $\|dz\|=\sqrt{ dr^2 + d\varphi^2 }$ the Euclidean norm
and by $\|dz\|_p=\cos \varphi |dr|$ the $p$-norm.
The tangent vector $dz$ corresponds to a tangent line with slope $\CcV=d\varphi/dr$ in $T_z M$,
as well as a pre-collisional wave front with slope $\cB^-$
and post-collisional wave front with slope $\cB^+$ in the phase space of the billiard flow.
The relation between these slopes are given by
\beqn
\CcV = \cB^-\cos\varphi + \cK = \cB^+\cos\varphi - \cK.
\eeqn

For any $z=(r, \varphi)\in M$ lying on the convex arcs $\Gamma_1\cup \Gamma_2$,
the $DF$-invariant unstable cones is given by
\beqn
\CcC^u(z)
=\left\{\cK\le \CcV\le 0\right\}
=\left\{ 0\le \cB^-\le 1/d \right\}
=\left\{ -2/d \le \cB^+\le -1/d \right\}.
\eeqn
where $\cK=\cK(z)<0$ and $d=-\cos\varphi/\cK$.
By the defocusing mechanism and the compactness of $\Gamma_1\cup \Gamma_2$,
there exists $\rho_0>1$ such that
\begin{itemize}
\item[(i)] $\tau=\tau(z, z_1)\ge 2\rho_0 d$ if $z$ lies on $\Gamma_1$
and $z_1=F(z)$ lies on $\Gamma_2$, or vice versa;
\item[(ii)] if further the doubly defocusing mechanism holds, and
$z$ lies on $\Gamma_1$ (resp. on $\Gamma_2$), $z'=F(z)$ lies on $\Gamma_3$
but $z_1=F^{2}(z)$ lies on $\Gamma_2$ (resp. on $\Gamma_1$),
then
$$
\tau=\widetilde\tau(z, \widetilde{z}_1)\ge 2\rho_0 d,
$$
where $\widetilde\tau(\cdot, \cdot)$ is the free path for the double cover table
$\widetilde\Omega$ (see Fig.~\ref{fig:stadia}, right),
and $\widetilde{z}_1$ is symmetric to $z_1$ with respect to $\Gamma_3$.
\end{itemize}
In either case, if $dz\in \CcC^u(z)$, then $dz_1 \in \CcC^u(z_1)$ and
\beq\label{def Lambda}
\frac{\|dz_1\|_p}{\|dz\|_p} =\frac{\left| \cB^+ \right|}{\left| \cB^-_1 \right|}
= \left| 1+\tau \cB^+ \right|=-1-\tau \cB^+ \ge 2\rho_0-1=:\Lambda>1.
\eeq
Given a smooth curve $W$ in $M$, we denote its Euclidean length and $p$-length by
\beqn
|W|=\int_W \|dz\|, \ \text{and}\ \ |W|_p=\int_W \|dz\|_p.
\eeqn
If it is an unstable curve, i.e., $d\varphi/dr\in\CcC^u(z)$ for any $z=(r, \varphi)\in W$, then
$F(W)$ (resp. $F^2(W)$) is an unstable curve in the above case (i) (resp. case (ii)),
as long as it does not hit the gluing points. Moreover, we have
\beq\label{wave contraction}
|F(W)|_p\ge \Lambda |W|_p \ \text{in case (i), \ or} \
|F^2(W)|_p\ge \Lambda |W|_p \ \text{in case (ii)}.
\eeq

\subsection{Variation of a Free Path}

In the above notations the billiard map $F$ sends $z=(r, \varphi)$
to $z_1=(r_1, \varphi_1)$. Note that if $r$ and $r_1$ are given,
then $\varphi$ and $\varphi_1$ are uniquely determined.
As the free path $\tau=\tau(z, z_1)$ is only determined by $r$
and $r_1$, we also write $\tau=\tau(r, r_1)$.
Elementary geometry shows that
\beq\label{tau dr}
\frac{\partial \tau}{\partial r}=-\sin \varphi\quad \ \  \text{and} \quad \ \ \frac{\partial \tau}{\partial r_1}=\sin \varphi_1.
\eeq
The following lemma provides a variational formula of a free path.

\begin{lemma}\label{lem: var path}
The variation, from $\tau(r, r_1)$ to $\tau(r+\Delta r, r_1+ \Delta r_1)$,
has the form:
{\allowdisplaybreaks
\begin{eqnarray}\label{var path}
\Delta \tau &=& \tau(r+\Delta r, r_1+ \Delta r_1) -\tau(r, r_1) \nonumber \\
&=& -\sin \varphi \ \Delta r + \sin \varphi_1 \ \Delta r_1  \nonumber \\
& & +\frac12\left[ \alpha(z) \Delta r^2 +
 \beta(z, z_1)\Delta r \Delta r_1
+  \alpha(z_1) \Delta r_1^2 \right]  \nonumber \\
& &  + \cO\left(\left(\Delta r^2+\Delta r_1^2\right)^{\frac32}\right),
\end{eqnarray}
\beqn
\text{with} \ \ 
\alpha(z)=\cK \cos\varphi + \frac{\cos^2\varphi}{\tau},
\alpha(z_1)=\cK_1 \cos\varphi_1 + \frac{\cos^2\varphi_1}{\tau}
\ \ \text{and} \ \
\beta(z, z_1)=\frac{2\cos \varphi \cos\varphi_1}{\tau},
\eeqn
}where $\cK$ and  $\cK_1$ are the signed curvatures of
$\p \Omega$ at $r$ and $r_1$ respectively.
\end{lemma}

\begin{proof}
Taking $dr_1=0$ in \eqref{der collision}, we obtain that
\beqn
\frac{\p \varphi}{\p r}=-\cK-\frac{\cos \varphi}{\tau}, \ \
\text{and} \ \
\frac{\p \varphi_1}{\p r} = \frac{\cos\varphi}{\tau}.
\eeqn
By time-reversibility, i.e., $(r_1, -\varphi_1)\mapsto (r, -\varphi)$, we also have
\beqn
\frac{\p \varphi_1}{\p r_1}=\cK_1 + \frac{\cos \varphi_1}{\tau}, \ \ \text{and} \ \
\frac{\p \varphi}{\p r_1}= - \frac{\cos\varphi_1}{\tau}.
\eeqn
By \eqref{tau dr}, we further obtain
{\allowdisplaybreaks
\begin{eqnarray}\label{tau dr 2}
\frac{\p^2 \tau}{\p r^2} & =& -\cos \varphi \frac{\p \varphi}{\p r} = \cK \cos\varphi + \frac{\cos^2\varphi}{\tau}, \nonumber\\
\frac{\p^2 \tau}{\p r \p r_1}  &=& -\cos \varphi \frac{\p \varphi}{\p r_1} = \frac{\cos \varphi \cos\varphi_1}{\tau}, \\
\frac{\p^2 \tau}{\p r_1^2} & =& \cos \varphi_1 \frac{\p \varphi_1}{\p r_1} = \cK_1 \cos\varphi_1 + \frac{\cos^2\varphi_1}{\tau}. \nonumber
\end{eqnarray}
}Therefore, \eqref{var path} follows from \eqref{tau dr} and \eqref{tau dr 2},
and the Taylor expansion of $\tau(x, x_1)$ up to the second order.
\end{proof}

\bigskip

\section{Analysis of the Period Two Orbit $\gamma^*$}\label{sec: period two}

\subsection{Existence and Uniqueness of the Period Two Orbit $\gamma^*$}
\label{sec: uniqueness}

Let $\Omega$ be a Bunimovich squash-type stadium  in $\cM^m_{ss}$ for $m\ge 3$.
The existence and uniqueness of the period two orbit,
which bounces between $\Gamma_1$ and $\Gamma_2$,
is due to the following lemma.

\begin{lemma}\label{lem: gamma uniq}
There exists a unique pair of points $(A, B)\in \Gamma_1\times \Gamma_2$
such that $\overline{AB}$ is perpendicular to both $\Gamma_1$ and $\Gamma_2$.
\end{lemma}

\begin{proof} Consider the free path
$\tau=\tau(r, r_1)$ for $(r, r_1)\in \Gamma_1\times \Gamma_2$,
and let $P$ and $P_1$ be the collision points
corresponding to $r$ and $r_1$ respectively.
Let $\varphi$ (resp.  $\varphi_1$) be the
angle formed by the vector $\overline{r r_1}$
and the inward (resp. outward) inner normal vector at $r$ (resp. at $r_1$).
By \eqref{tau dr},
$\overline{PP_1}$ is perpendicular to both $\Gamma_1$ and $\Gamma_2$
if and only if $(r, r_1)$ is a critical point of $\tau$.
Moreover,  by \eqref{tau dr 2}, the Hessian matrix of $\tau$ is given by
$$
\begin{pmatrix}
\cK \cos\varphi + \dfrac{\cos^2\varphi}{\tau} & \dfrac{\cos \varphi \cos\varphi_1}{\tau} \\
\dfrac{\cos \varphi \cos\varphi_1}{\tau} & \cK_1 \cos\varphi_1 + \dfrac{\cos^2\varphi_1}{\tau}
\end{pmatrix}.
$$
By the defocusing mechanism \eqref{def defocusing}, we have
$\cK\tau<-2\cos\varphi$ and $\cK_1\tau<-2\cos\varphi_1$, and thus the Hessian matrix of $\tau$
is negative definite since
\beqn
\cK \cos\varphi + \dfrac{\cos^2\varphi}{\tau} < -\dfrac{\cos^2\varphi}{\tau}<0,
\eeqn
and the determinant of Hessian is
\beqn
\left(\cK \cos\varphi + \dfrac{\cos^2\varphi}{\tau} \right)
\left(\cK_1 \cos\varphi_1 + \dfrac{\cos^2\varphi_1}{\tau}\right)
-\dfrac{\cos^2 \varphi \cos^2\varphi_1}{\tau^2}>0.
\eeqn
In other words, $\tau$ is a strictly concave function on $\Gamma_1\times \Gamma_2$,
and thus there can be at most one critical point for $\tau$.

On the compact domain $\Gamma_1\times \Gamma_2$,
$\tau$ has a global maximum point, say $(A, B)$. We
claim that $(A, B)$ must be an interior point of $\Gamma_1\times \Gamma_2$.
Otherwise, let $(r, r_1)$ be the arclength representation of $(A, B)$, and assume
that $r\in \p\Gamma_1$. Since $\Gamma_1$ is $C^1$ tangent to
flat boundaries at the gluing point,
when $r+\Delta r\in \Gamma_1$ for small $\Delta r$, we must have
$\varphi$ and $\Delta r$ are of opposite signs.
By \eqref{tau dr},
\beqn
\tau(r+\Delta r, r_1)=-\sin\varphi \Delta r + \cO(|\Delta r|^2) >0,
\eeqn
which implies that $(A, B)$ is not even a local maximum - Contradiction.
Therefore, the global maximum point $(A, B)$ is an interior point and thus
the only critical point of $\tau$ on $\Gamma_1\times \Gamma_2$.
\end{proof}

From the proof, we actually get $\tau^*=\tau(A, B)=|\overline{AB}|=\diam(\Omega)$.
In the rest of this section, we consider the maximal period two billiard orbit $\gamma^*=\overline{AB}$
that collides alternatively at $A\in \Gamma_1$ and $B\in\Gamma_2$.

\subsection{Hyperbolicity of the Maximal Period Two Orbit $\gamma^*$}\label{sec: hyperbolic}

Since $\gamma^*=\overline{AB}$ is perpendicular to
both $\Gamma_1$ and $\Gamma_2$,
we denote $x=(r_1, 0)$ and $y=(r_2, 0)$ the collision vectors at
$A$ and $B$ respectively,
for some $0<r_1<r_2<\left|\p \Omega\right|$.
We shall also use
the notation $\gamma^*=\overline{xy}$ for $\gamma^*=\overline{AB}$,
and use the notation $\tau^*=\tau(x, y)$ for $\tau^*=\tau(A, B)$.
\medskip

For convenience, we may choose $s$ as an arclength parameter on
$\Gamma_1$ oriented in the counter-clockwise direction, with $s=0$
corresponding to the position of $A$.
Similarly, we denote $t$ as a counter-clockwise arclength
parameter on $\Gamma_2$ , with $t=0$ corresponding to
the position of $B$.

Using the coordinate $(s, \varphi)$ on $\Gamma_1$ and $(t,\varphi)$
on $\Gamma_2$, the differential of the billiard map $F$ along
$\gamma^*=\overline{xy}$ can be represented by
the following matrices:
{\allowdisplaybreaks
\beq\label{DF 12}
\begin{split}
D_{x} F &=
-\begin{pmatrix}
1-\tau^* \cK_A   & \tau^* \\
\\
\tau^*\,\cK_A\cK_B-\cK_A - \cK_B  & 1-\tau^*\,\cK_B
\end{pmatrix}
=:
-\begin{pmatrix}
a_1   & \tau^* \\
b  &   a_2
\end{pmatrix}, \
\\ \ \
D_{y} F &=
-\begin{pmatrix}
a_2   & \tau^* \\
b  &   a_1
\end{pmatrix},
\end{split}
\eeq}where $\cK_A$ and $\cK_B$ are the absolute curvature of
$\p\Omega$ at $A$ and $B$ respectively.
By the defocusing mechanism \eqref{def defocusing}, we have
$\tau^*>\max\{ 2/\cK_A, \ 2/\cK_B\}$, which means that
$a_1<-1$ and $a_2<-1$. Note that $a_1a_2-b\tau^*=1$. Hence
\beq\label{DF2 1}
\begin{split}
D_{x} F^2 &=
\begin{pmatrix}
2a_1a_2-1  & 2a_2\tau^* \\
2a_1b  & 2a_1a_2-1
\end{pmatrix}, \,
\\
D_{y} F^2 &=
\begin{pmatrix}
2a_2a_1-1  & 2a_1\tau^* \\
2a_2b  & 2a_1a_2-1
\end{pmatrix}
\end{split}
\eeq
are hyperbolic matrices
since they have determinant one and the same trace
\beq\label{lambda relation}
\lambda+\lambda^{-1}=2(2a_1a_2-1) > 2,
\eeq
where  $\lambda$  denotes the leading eigenvalues of
$D_{x} F^2$ (which is the same for $D_{y} F^2$).
Therefore, $\gamma^*$ is a hyperbolic orbit.

The variation of the free path near $\gamma^*=\overline{xy}$
can be simplified as follows: if collision points move from
$(s, t)=(0, 0)$ to $(s, t)=(\Delta s, \Delta t)$,
then \eqref{var path} reads
\beq\label{var path 0}
\Delta \tau = \frac{1}{2\tau^*}\left[ a_1 \Delta s^2 + 2\Delta s \Delta t +
a_2 \Delta t^2 \right]
 + \cO\left(\left(\Delta s^2+\Delta t^2\right)^{\frac32}\right),
\eeq
where $a_1$ and $a_2$ are given by \eqref{DF 12}.

\subsection{The Linearization near $\gamma^*$}

We denote by $\theta_1^s$ (resp. $\theta_1^u$)
the angle formed by the unit stable (resp. unstable) vector
$V^{s}_{x}$ (resp. $V^u_{x}$)
of $D_{x} F^2$ with the positive $r$-axis, then
$$
V^s_{x}=(\cos\theta_1^s, \sin\theta_1^s)\quad \text{ and }
\quad V^u_{x}=(\cos\theta_1^u, \sin\theta_1^u).
$$
Using  \eqref{DF2 1} and the eigenvector equations:
\beqn
\begin{split}
(2a_1a_2-1-\lambda^{-1}) \cos\theta_1^s +
2a_2\tau^* \sin\theta_1^s &= 0, \\
(2a_1a_2-1-\lambda) \cos\theta_1^u +
2a_2\tau^* \sin\theta_1^u &=0,
\end{split}
\eeqn
we obtain
\beq
\label{theta1}
\tan\theta_1^s=\frac{\lambda^{-1}-\lambda}{4 a_2\tau^*}
=-\tan\theta_1^u.
\eeq
We then simply denote $\theta_1=\theta_1^s$, and, thus,
$\theta_1^u=-\theta_1$. Also, we rewrite
\beqn \label{eq:angle-theta}
V^s_{x}=(\cos\theta_1,\ \sin\theta_1)\ \text{ and }\
V^u_{x}=(\cos\theta_1,\, -\sin\theta_1).
\eeqn
Similarly, we denote the unit stable vector
$V^{s}_{y}=(\cos\theta_2,\sin\theta_2)$ and the unit
unstable vector $V_{y}^u=(\cos\theta_2, -\sin\theta_2)$ for
the matrix $D_{y} F^2$, where the angle $\theta_2$
satisfies that
\beq\label{theta2}
\tan\theta_2=\frac{\lambda^{-1}-\lambda}{4 a_1\tau^*}.
\eeq
In addition, using the fact that
$$
D_{x}F (V^u_{x}, V^s_{x})=(\lambda_{1, u} V^u_{y}, \lambda_{1, s} V^s_{y}) \ \ \ \text{ and }\
\ \
D_{y}F (V^u_{y}, V^s_{y})=(\lambda_{2, u} V^u_{x}, \lambda_{2, s} V^s_{x}),
$$
we obtain
\beq
\label{lambdazw}
\begin{split}
\lambda_{1, u}=-\frac{\cos\theta_1}{\cos\theta_2}\cdot \frac{1+\lambda }{2a_2}
\ \ \text{ and } \ \
\lambda_{1, s}=-\frac{\cos\theta_1}{\cos\theta_2}\cdot \frac{1+\lambda^{-1}}{2a_2}, 
\\
\lambda_{2, u}=-\frac{\cos\theta_2}{\cos\theta_1}\cdot \frac{1+\lambda }{2a_1}
\ \ \text{ and } \ \
\lambda_{2, s}=-\frac{\cos\theta_2}{\cos\theta_1}\cdot \frac{1+\lambda^{-1} }{2a_1}.
\end{split}
\eeq
It is easy to verify that 
$\lambda_{i, u}>1>\lambda_{i, s}$ for $i=1, 2$. Also,
$\lambda_{1, u}\lambda_{2, u}=\lambda$ and 
$\lambda_{1, s}\lambda_{2, s}=\lambda^{-1}$. 
Note that usually $\lambda_{i, u}\lambda_{i, s}\ne 1$, $i=1, 2$, 
unless the parallelogram formed by $(V^u_{x}, V^s_{x})$ and the one formed by $(V^u_{y}, V^s_{y})$ 
have the same area.


\medskip

To study the billiard map near $\gamma^*=\overline{xy}$,
we first recall a well known result about the linearization
near a saddle in dimension two (see e.g. \cite{Stowe86, ZZ14}).

\begin{lemma}\label{lem: linearization}
For any $\eps>0$, there are $C^{1, \frac12}$ diffeomorphisms
$\Psi_1: U_1\to \Psi_1(U_1)\subset W_1$
and $\Psi_2: U_2\to \Psi_2(U_2)\subset W_2$,
where $U_1$, $W_1$ are neighborhoods of  $x$ and
$U_2$, $W_2$ are neighborhoods of  $y$,
such that
\beqn
\Psi_2^{-1}\circ F\circ \Psi_1= D_{x} F, \ \ \text{and}\ \
\Psi_1^{-1}\circ F\circ \Psi_2= D_{y} F.
\eeqn
Moreover, $\Psi_1(x)=x$, $\Psi_2(y)=y$,
$\left\|\Psi_1^{\pm 1}- \mathrm{Id} \right\|_{C^1} \le \eps$,
$\left\|\Psi_2^{\pm 1}- \mathrm{Id} \right\|_{C^1} \le \eps$, and
\beqn
\Psi_1^{\pm 1}(x_1)-\Psi_1^{\pm 1}(x_2)=x_1-x_2 + O\left( \left|x_1-x_2\right| \cdot \max\left\{|x_1|^{0.5}, |x_2|^{0.5}\right\} \right), \,\, x_1, x_2\in U_1
\eeqn
\beqn
\Psi_2^{\pm 1}(y_1)-\Psi_2^{\pm 1}(y_2)=y_1-y_2 + O\left( \left|y_1-y_2\right| \cdot \max\left\{|y_1|^{0.5}, |y_2|^{0.5}\right\} \right),\,\,\, y_1, y_2\in U_2.
\eeqn
\end{lemma}
\medskip

For $i=1, 2$, we further choose the following invertible matrices
\beq\label{def Theta}
\Theta_i =
\begin{pmatrix}
\ \ \cos\theta_i & \cos\theta_i \\
-\sin\theta_i & \sin\theta_i
\end{pmatrix},
\eeq
and introduce a local coordinate system inside $U_1\cup U_2$ such that
\beq \label{coordinate-change}
\begin{pmatrix} s \\ \varphi \end{pmatrix}=
\Psi_1 \circ \Theta_1 \begin{pmatrix} \xi \\ \eta \end{pmatrix},
\ \text{and}
\begin{pmatrix} t \\ \psi \end{pmatrix}=
\Psi_2 \circ \Theta_2\begin{pmatrix} \zeta \\ \iota \end{pmatrix}.
\eeq
By Lemma~\ref{lem: linearization}, if $z=(\xi, \eta)\in U_1$ and $F(z)=(\zeta, \iota)\in U_2$,
then
\beq\label{linearization 2}
\zeta=\lambda_{1, u} \xi, \ \ \  \text{and} \ \ \  \iota=\lambda_{1, s} \eta.
\eeq
Similarly, if $w=(\zeta, \iota)\in U_2$ and 
$F(w)=(\xi, \eta)\in U_1$, then
\beq\label{linearization 2'}
\xi=\lambda_{2, u} \zeta, \ \ \  \text{and} \ \ \  \eta=\lambda_{2, s} \iota.
\eeq
Also, if there are $0\le m\le m'$ such that
$F^{2k}(z)=(\xi_k, \eta_k)\in U_1$ and $F^{2k}(w)=(\zeta_k, \iota_k)\in U_2$
for any $k\in [m, m']$,  then
\beq\label{linearization 1}
\begin{aligned}
\xi_{k}=\lambda^{k-m} \xi_{m}, \quad  & \quad \eta_{k}=\lambda^{-k+m} \eta_m;
\\ \
\zeta_{k}=\lambda^{k-m} \zeta_m, \quad & \quad \iota_{k}=\lambda^{-k+m} \iota_m.
\end{aligned}
\eeq

\bigskip

\section{Analysis of Palindromic Periodic Orbits $\gamma_n$}
\label{sec:palindromic}

\subsection{The Palindromic Periodic Orbits $\gamma_{n}$}
\label{sec: palin}

Let $\Omega$ be a Bunimovich squash-type stadium  in $\cM^m_{ss}$ for $m\ge 3$.
We study palindromic periodic orbits, namely, periodic orbits such that
the associated trajectory hits the billiard table perpendicularly at two `turning' points. 
For any integer $n\geq 1$, we consider the   palindromic
periodic orbits $\gamma_n$ associated with the following symbolic codes:
\beq\label{def palin}
(323\underbrace{12 \cdots 121}_{2n+1}).
\eeq
The period of $\gamma_{n}$ is equal to $2n+4$. Furthermore, $\gamma_n$ is
palindromic as we track the motion of a billiard ball
along this orbit (see Fig. \ref{fig:palin}):
\begin{itemize}
\item The `initial' stage:  from an initial position on $\Gamma_3$, a billiard ball
hits perpendicularly at $\Gamma_2$ and then returns to the initial position;
\item The `successive collision' stage:
the billiard ball moves from $\Gamma_3$ towards $\Gamma_1$, and
collides successively between $\Gamma_1$ and $\Gamma_2$ for $(2n+1)$ times,
and then gets back to the initial position on $\Gamma_3$.
Note that it hits $\Gamma_i$ perpendicularly at the $(n+1)$-th collision,
where $i=1$ is $n$ is even, and $i=2$ if $n$ is odd.
\end{itemize}

\begin{figure}[h]
\begin{center}
\includegraphics[width=.9\textwidth]{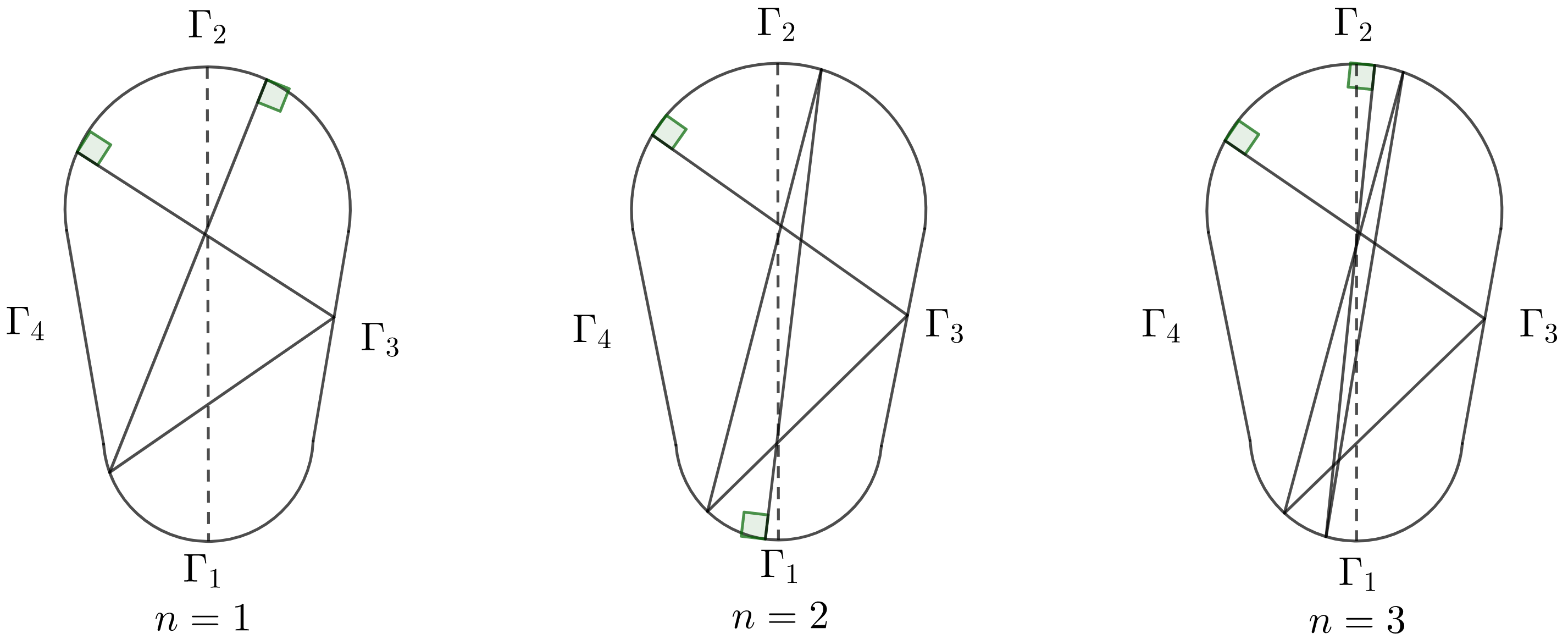}
\end{center}
\caption{The palindromic orbit $\gamma_{n}$ with $n=1, 2, 3$}
\label{fig:palin}
\end{figure}

We first need to show  the existence and uniqueness of such orbits.

\begin{lemma}\label{lem: gamma_n uniq}
For any $n\ge 1$, there exists a unique palindromic
periodic orbit $\gamma_n$ associated with the symbolic sequence \eqref{def palin}.
\end{lemma}

\begin{proof}
Let $\widetilde\Omega$ be the double cover table of $\Omega$  (see Fig.~\ref{fig:stadia}, right),
and let $\widetilde\Gamma_i$ be the new boundaries of $\widetilde\Omega$, which corresponds
to the code $\widetilde i$ for $i=1, 2, 4$.
Then $(r_0^- \overline r_0 r_0^+ r_1 \dots r_{2n+1})$ forms a periodic orbit in $\Omega$
associated with \eqref{def palin}, if and only if
$(r_0 r_1 \dots r_{2n+1})$ forms a periodic orbit in $\widetilde\Omega$
associated with
\beq\label{def palin double}
(\widetilde 2 \underbrace{12 \cdots 121}_{2n+1}),
\eeq
where $r_0$ is the symmetric point of $\overline r_0$ with respect to $\Gamma_3$.
To this end, we recall that $\widetilde\tau(\cdot, \cdot)$ is the free path in $\widetilde\Omega$,
and consider the length function
\beqn
L(r_0, r_1, \dots,  r_{2n+1})
=\sum_{k=0}^{2n+1} \widetilde\tau(r_k, r_{k+1}),
\eeqn
for $(r_0, r_1, \dots,  r_{2n+1})\in
\widetilde\Gamma_2 \times \left(\Gamma_1 \times \Gamma_2 \right)^{n} \times \Gamma_1$,
where we set $r_{2n+2}=r_0$.
Let $\varphi(r_k, r_{k+1})$ (resp.  $\varphi_1(r_k, r_{k+1})$) be the
angle formed by the vector $\overline{r_k r_{k+1}}$
and the inward (resp. outward) inner normal vector at $r_k$ (resp. at $r_{k+1}$).
By \eqref{tau dr},
\beqn
\frac{\p L}{\p r_k}=\frac{\p \widetilde\tau}{\p r_k}(r_{k-1}, r_k) + \frac{\p \widetilde\tau}{\p r_k}(r_{k}, r_{k+1})
=\sin\varphi_1(r_{k-1}, r_k) -\sin\varphi(r_k, r_{k+1}).
\eeqn
It follows that $(r_0, r_1, \dots, r_{2n+1})$ forms a periodic orbit, if and only if
it is a critical point of $L$. Similar to the proof of Lemma~\ref{lem: gamma uniq}, the Hessian matrix
of $L$ is negative definite due to the doubly focusing mechanism.
Thus, $L$ is a strictly concave function on the compact domain
$\widetilde\Gamma_2 \times \left(\Gamma_1 \times \Gamma_2 \right)^{n} \times \Gamma_1$,
and there can be at most one critical point for $L$.

On the other hand, $L$ has a global maximum point $\widetilde\gamma_n=(r_0, r_1, \dots, r_{2n+1})$
which must be an interior point. Therefore, $\widetilde\gamma_n$ is the only critical point of $L$,
which forms a periodic orbit in $\widetilde\Omega$ associated with \eqref{def palin double}.
Noticing that
\beqn
L(r_0, r_1, r_2, \dots, r_{n+1}, \dots, r_{2n}, r_{2n+1})=
L(r_0, r_{2n+1}, r_{2n}, \dots, r_{n+1}, \dots, r_{2}, r_{1}),
\eeqn
we further get $r_{n+2-k}=r_k$ for $1\le k\le n$, that is, the orbit $\widetilde\gamma_n$ is palindromic.
Finally, we obtain the unique periodic orbit $\gamma_n:=(r_0^- \overline r_0 r_0^+ r_1 \dots r_{2n+1})$ in $\Omega$
associated with \eqref{def palin}, where $\overline r_0$ is the symmetric point of $r_0$ with respect to $\Gamma_3$,
and $r_0^\pm$ is obtained as the intersection between $\overline{r_0 r_1}$ and  $\Gamma_3$.
\end{proof}

\subsection{The Homoclinic Semi-orbit $\gamma_\infty$}\label{sec: homoclinic}

We denote the collision points of $\gamma_{n}$ by
\begin{align*}
y_{n}(0^-) \mapsto x_n(0) \mapsto y_n(0) \mapsto
x_{n}(1) \mapsto y_{n}(1)  \mapsto \dots \mapsto x_{n} (n) 
\mapsto y_{n}(n)  \mapsto x_{n} (n+1),
\end{align*}
where 
\begin{itemize}
\item
at the initial stages corresponding to the codes $(323)$, 
we denote by $x_n(0)=(s_n(0), 0)$
the collision point on $\Gamma_2$, and denote by
$y_n(0)=(t_n(0), \psi_n(0))$ and $y_{n}(0^-)=(t_n(0), -\psi_{n}(0))$
the two collision points on $\Gamma_3$.\footnote{
For convenience, we extend the $(s, \varphi)$- and $(t, \psi)$-coordinates
on the full boundary $\p\Omega$. Also,
even though $x_n(0)$ lies on $\Gamma_2$,  
we still use $(s, \varphi)$-coordinate to maintain the bouncing ordering. }
\item 
at the stage of $2n+1$ successive collisions between $\Gamma_1$ and $\Gamma_2$,
we denote
\begin{align*}
\text{on}\ \Gamma_1: \ x_{n}(k) & =(s_{n}(k), \varphi_{n}(k)), \ k=1, 2, \dots, n, n+1; \\
\text{on}\ \Gamma_2: \ y_{n}(k) & =(t_{n}(k), \psi_{n}(k)), \ k=1, 2, \dots, n.
\end{align*}
\end{itemize}

By time reversibility, we have
\beq\label{time-rev s}
\begin{split}
s_{n}(n+2-k)=s_{n}(k), \ & \ \varphi_{n}(n+2-k)=-\varphi_{n}(k), \ k=1, 2, \dots, n+1,   \\
t_{n}(n+1-k)=t_{n}(k), \ & \ \psi_{n}(n+1-k)=-\psi_{n}(k),  \ k=1, 2, \dots, n.
\end{split}
\eeq
In particular, $\varphi_{n}(\frac{n+2}{2})=0$ if $n$ is even, and
$\psi_{n}(\frac{n+1}{2})=0$ is $n$ is odd.

\medskip

We first provide some rough estimates.

\begin{lemma}\label{lem: rough esta}
There exists $C>0$ such that for any $n\ge 1$ and $m\ge 0$,
\beqn
\begin{split}
\|x_{n+m}(k)-x_{n}(k)\| &\le  C\Lambda^{2k-2n},
\ \ \ \ \ k=0, 1, \dots, n+1,
\\
\|y_{n+m}(k)-y_{n}(k)\| &\le  C\Lambda^{2k-2n},
\ \ \  \ \ k=0^-, 0,  1, \dots, n.
\end{split}
\eeqn
where the constant $\Lambda>1$ is given by \eqref{def Lambda}, and
$\|\cdot\|$ denotes the Euclidean norm in the phase space $M$.
\end{lemma}

\begin{proof}
Recall that the billiard ball hits $\Gamma_2$ perpendicularly in the initial stage of $\gamma_n$
and $\gamma_{n+m}$, that is, $x_n(0)=(s_n(0), 0)$ and $x_{n+m}(0)=(s_{n+m}(0), 0)$.
Let $W$ be  the wave front between $s=s_n(0)$ and $s=s_{n+m}(0)$, associated with zero angles.
Then $W$ is an unstable curve, so is its forward iterate $F^k(W)$ for any $k\ge 0$,
unless $F^k(W)$ is cut by singularities, i.e., $F^k(W)$ hits the gluing points.

Following the trajectories of $\gamma_{n}$ and $\gamma_{n+m}$, it is not hard to
see that $F^k(W)$ stays away from the singularities for $0\le k\le 2n+1$.
Furthermore, the transition from $F^k(W)$ to $F^{k+1}(W)$ is between $\Gamma_1$ and $\Gamma_2$,
for $2\le k\le 2n+1$, then by \eqref{wave contraction},
we have the following estimates for the $p$-length of $F^k(W)$:
$$
|F^{k+i}(W)|_p \ge \Lambda^i |F^k(W)|_p, \ \ 2\le k\le k+i\le 2n+2.
$$
At Step $k=0$, $W$ goes from $\Gamma_2$ and hits the flat wall $\Gamma_3$,
and then at Step $k=1$, it collides on $\Gamma_1$. Thus,
$|F^2(W)|_p\ge \Lambda |W|_p$.

Note that there is $\varphi_0\in (0, \pi/2)$ such that $|\varphi|\le \varphi_0$ for
any point $z=(r, \varphi)$ lying on $\Gamma_1$ (resp. on $\Gamma_2$)
satisfying the following properties:
\begin{itemize}
\item $F^{-1}(z)$ lies on $\Gamma_2$ (resp. on $\Gamma_1$), or
$F^{-1}(z)$ lies on $\Gamma_3$ but $F^{-2}(z)$ lies on $\Gamma_2$  (resp. on $\Gamma_1$);
\item $F(z)$ lies on $\Gamma_2$ (resp. on $\Gamma_1$), or
$F(z)$ lies on $\Gamma_3$ but $F^{2}(z)$ lies on $\Gamma_2$  (resp. on $\Gamma_1$).
\end{itemize}
Therefore, the angle variables on all the curves $\{F^k(W)\}_{k\ge 0}$ have absolute value bounded by $\varphi_0$.
Hence there is $C_0>0$ such that for any $k\ge 0$,
$$
|F^k(W)|/C_0 \le |F^k(W)|_p\le C_0 |F^k(W)|.
$$
Set $C=C_0^2 \diam(M)$.
Note that $x_n(0)$ and $x_{n+m}(0)$ are the endpoints of $W$, then
\begin{align*}
\|x_{n+m}(0)-x_n(0)\|=
|W|\le C_0 |W|_p  \le C_0 \Lambda^{-2n} |F^{2n+1}(W)|_p 
& \le C_0^2 \Lambda^{-2n} |F^{2n+1}(W)| \\
& \le C\Lambda^{-2n}.
\end{align*}
The remaining estimates in this lemma can be shown in a similar fashion,
by noticing the following facts: 
$x_n(k)$ and $x_{n+m}(k)$ are endpoints of $F^{2k}(W)$ for $1\le k\le n+1$;
$y_n(k)$ and $y_{n+m}(k)$ are endpoints of $F^{2k+1}(W)$ for $0\le k \le n $. 
\end{proof}

By Lemma~\ref{lem: rough esta}, we define
\beq\label{def ga infty}
\begin{split}
x_\infty(k)&=\lim_{n\to\infty} x_n(k), \ k= 0, 1, 2, \dots; \\
y_\infty(k)&=\lim_{n\to\infty} y_n(k), \ k= 0^-, 0,  1, 2, \dots.
\end{split}
\eeq
Then we obtain a semi-orbit
\beqn
\gamma_\infty:=(y_\infty(0^-) \ x_\infty(0) \ y_\infty(0) \ x_\infty(1) \ y_\infty(1) \dots \ x_\infty(k) \ y_\infty(k) \dots),
\eeqn
which corresponds to the symbolic code
$(323 121212 \cdots)$.
The following lemma shows that $\gamma_\infty$ is a homoclinic semi-orbit of
the period two orbit $\gamma^*=\overline{xy}$.

\begin{lemma}\label{lem: homoclinic}
$\lim\limits_{k\to\infty} x_\infty(k) = x$, and $\lim\limits_{k\to\infty} y_\infty(k) = y$.
\end{lemma}

\begin{proof} For any $k\ge 1$, by Lemma~\ref{lem: rough esta},
we take $n=2k-1$ and let $m\to\infty$, then
\beqn
\|x_\infty(k)-x_{2k-1}(k)\|\le C\Lambda^{-2k+2}.
\eeqn
On the other hand, the trajectory $\gamma_{2k-1}$ hits
perpendicularly at $\Gamma_1$ at the $k$-th step, i.e.,
$x_{2k-1}(k)=(s_{2k-1}(k), 0)$.
Also, we note that the period two point on $\Gamma_1$
is given by $x=(0, 0)$.
Let $W$ be  the wave front between
$s=s_{2k-1}(k)$ and $s=0$, associated with zero angles.
Then $W$ is an unstable curve, and so is $F^\ell(W)$ for
any $0\le \ell\le k-1$ since $F^\ell(W)$ does not hit the gluing points.
Therefore, $|F^{k-1}(W)|_p\ge \Lambda^{k-1} |W|_p$.
Similar to the proof of Lemma~\ref{lem: rough esta}, we obtain
\beqn
\|x_{2k-1}(k)-x\|=|s_{2k-1}(k)-0|\le C\Lambda^{-k+1},
\eeqn
and thus, $\|x_\infty(k)-x\|\le 2C\Lambda^{-k+1}$, which implies that
$\lim\limits_{k\to\infty} x_\infty(k) = x$. The other limiting result
$\lim\limits_{k\to\infty} y_\infty(k) = y$ can be proven in a similar way.
\end{proof}

We denote the coordinates of $\gamma_\infty$ as
\begin{align*}
x_\infty(k)&=(s_\infty(k), \varphi_\infty(k)), \ k=0, 1, 2, \dots; \\
y_\infty(k)&=(t_\infty(k), \psi_\infty(k)), \  k= 0, 1, 2, \dots,
\end{align*}
and $y_\infty(0^-)=(t_\infty(0), -\psi_\infty(0))$. Note that
$\varphi_\infty(0)=\lim\limits_{n\to\infty} \varphi_n(0)=0$.

\subsection{The Convergence of $\gamma_\infty$ to $\gamma^*$ and
the Shadowing of $\gamma_n$ along $\gamma_\infty$}
\label{sec: palin finer est}

Recall that $\lambda>1$ is the leading eigenvalue of $DF^2$ along the
period two orbit $\gamma^*=\overline{xy}$.
The next lemma provides finer estimates of
the asymptotic convergence of the homoclinic orbit
$\gamma_\infty$ to the period two orbit $\gamma^*$,
as well as the shadowing estimates of
$\gamma_n$ along $\gamma_\infty$.

\begin{lemma}\label{lem: finer est} 
$ $
\begin{itemize}
\item[(a)]
The following estimates hold for the homoclinic orbit $\gamma_\infty$:
{\allowdisplaybreaks
\beq\label{finer est ga n}
\begin{split}
x_\infty(k)=\lambda^{-k}(C_s, C_\varphi) +  \cO(\lambda^{-1.5k}),
& \ \ k=0, 1, 2, \dots, \\
y_\infty(k)=\lambda^{-k} (C_t, C_\psi) + \cO(\lambda^{-1.5k}),
& \ \ k=0, 1, 2, \dots,
\end{split}
\eeq
}where the constants $C_s, C_\varphi, C_t$ and $C_\psi$ satisfy 
the following relation:
\beq\label{C relation}
\dfrac{C_\varphi}{C_s}=\frac{\lambda^{-1}-\lambda}{4a_2\tau^*}, \ \
\dfrac{C_\psi}{C_t}=\frac{\lambda^{-1}-\lambda}{4a_1\tau^*},  \ \
\dfrac{C_t}{C_s}=-\frac{1+\lambda^{-1}}{2a_2 }= -\frac{2a_1}{1+\lambda}.
\eeq
\item[(b)] 
The following estimates hold for the palindromic orbit $\gamma_n$:
 \beq\label{finer est ga n infty0}
\begin{split}
x_n(k)-x_\infty(k)=\lambda^{k-n}(C_{s,k}, C_{\varphi, k}) + \cO(\lambda^{0.5k-n}) ,
& \ \ k=0, 1, \dots, \left\lfloor{\frac{n}{2}}\right\rfloor+1, \\
y_n(k)-y_\infty(k)=\lambda^{k-n}(C_{t,k}, C_{\psi, k}) + \cO(\lambda^{0.5k-n}),
& \ \ k= 0, 1, \dots, \left\lfloor{\frac{n}{2}}\right\rfloor.
\end{split}
\eeq
Here the constants $C_{s,k}, C_{\varphi, k}, C_{t,k}$ and $C_{\psi, k}$
are  given by
{\allowdisplaybreaks
\beq\label{C1 relation}
\begin{split} 
C_{s,k}=C_s \,\lambda^{-2} \,(\lambda^{-2k}+1), \ \ 
& \ \ C_{\varphi, k}= C_\varphi \,\lambda^{-2} \,(\lambda^{-2k}-1), \\
C_{t,k}=C_t\, \lambda^{-1} \,(\lambda^{-2k}+ 1), \  \ 
& \ \  C_{\psi, k}=C_\psi \, \lambda^{-1} \,(\lambda^{-2k}- 1).
\end{split}
\eeq
} 
\end{itemize}
\end{lemma}
\begin{proof}
Let $U_1$ and $U_2$ be the neighborhood of  $x$ and  $y$ respectively, which are
given by Lemma~\ref{lem: linearization}.
Choose an integer $k_0>0$ such that $x_n(k)\in U_1$ for all $k\in [k_0, n+2-k_0]$
and $y_n(k)$ for all $k\in [k_0,  n+1-k_0]$.
We apply the coordinate change given by \eqref{coordinate-change}, that is,
$x_n(k)$ and $y_n(k)$ are represented by $(\xi_n(k), \eta_n(k))$ and $(\zeta_n(k), \iota_n(k))$ respectively,
for $k\ge k_0$. Set
\beq\label{setting xis}
\xi_n=\lambda^{-k_0}\xi_n(k_0), \ \
\eta_n=\lambda^{k_0}\eta_n(k_0), \ \
\zeta_n=\lambda^{-k_0}\zeta_n(k_0), \ \
\iota_n=\lambda^{k_0}\iota_n(k_0).
\eeq
Then \eqref{linearization 1} implies that

\beq\label{iterate ga n}
\xi_n(k)=\lambda^{k} \xi_n, \ \ 
\eta_n(k)=\lambda^{-k} \eta_n, \ \
\zeta_n(k) = \lambda^{k} \zeta_n,  \ \
\iota_n(k)=\lambda^{-k} \iota_n.
\eeq
These formulas also hold for $x_n(k)$ with $k\in [0, k_0)\cup (n+2-k_0, n+1]$
and for $y_n(k)$ with $k\in [0, k_0)\cup (n+1-k_0, n]$,
by suitably extending $\Psi_1$ and $\Psi_2$
along a neighborhood of the separatrices of $\gamma^*$.
In particular, we denote 
\beqn
\xi_n(0)=\xi_n, \ \
\eta_n(0)=\eta_n, \ \
\zeta_n(0)=\zeta_n, \ \
\iota_n(0)=\iota_n.
\eeqn

\medskip

\noindent (a) We first show the estimates along the homoclinic orbit $\gamma_\infty$.
The coordinates of 
$x_\infty(k)$ are denoted by $(\xi_\infty(k), \eta_\infty(k))$ for $k=0, 1, 2, \dots$,
and 
the coordinates of 
$y_\infty(k)$ are denoted by $(\zeta_\infty(k), \iota_\infty(k))$ for 
$k=0^-, 0,   1, 2, \dots$. 
By \eqref{def ga infty} and \eqref{setting xis}, 
the following limits exist:
\beqn
\xi_\infty=\lim\limits_{n\to\infty} \xi_n, \ \
\eta_\infty=\lim\limits_{n\to\infty} \eta_n, \ \
\zeta_\infty=\lim\limits_{n\to\infty} \zeta_n, \ \
\iota_\infty=\lim\limits_{n\to\infty} \iota_n.
\eeqn
By passing $n\to\infty$ in \eqref{iterate ga n}, we have
for any $k\ge 0$, 
\beqn
\xi_\infty(k)=\lambda^k \xi_\infty, \ \
\eta_\infty(k)=\lambda^{-k}\eta_\infty, \ \
\zeta_\infty(k)=\lambda^k\zeta_\infty, \ \
\iota_\infty(k)=\lambda^{-k}\iota_\infty.
\eeqn
It is clear that $\lambda^{-k_0}x_\infty(k_0)=(\xi_\infty, \eta_\infty)$ and 
$\lambda^{-k_0}y_\infty(k_0)=(\zeta_\infty, \iota_\infty)$.

Note that
$\lim\limits_{k\to\infty} x_\infty(k)=x$ implies that
$x_\infty(k)$ lies on the local stable manifold $W^s(x)=\{(\xi, \eta):\ \xi= 0\}$,
and thus, $\xi_\infty(k)=\xi_\infty=0$ for all $k\ge 0$. Hence
{\allowdisplaybreaks
\begin{align*}
x_\infty(k)=\begin{pmatrix} s_\infty(k) \\ \varphi_\infty(k) \end{pmatrix}
=\Psi_1\circ \Theta_1 \begin{pmatrix} \xi_\infty(k) \\  \eta_\infty(k) \end{pmatrix}
&=\Psi_1\circ
\begin{pmatrix}
\ \ \cos\theta_1 & \cos\theta_1 \\
-\sin\theta_1 & \sin\theta_1
\end{pmatrix}
 \begin{pmatrix} 0 \\ \lambda^{-k} \eta_\infty \end{pmatrix}
\\
&=\Psi_1\left( \lambda^{-k} \eta_\infty
\begin{pmatrix}
\cos\theta_1 \\ \sin\theta_1
\end{pmatrix} \right) - \Psi_1(0, 0) \\
&= \lambda^{-k}
\begin{pmatrix}
C_s \\ C_\varphi
\end{pmatrix}  + \cO(\lambda^{-1.5k}),
\end{align*}
}where we set
$C_s:=\eta_\infty\cos\theta_1$ and
$C_\varphi:=\eta_\infty\sin\theta_1$.
Similarly, since $y_\infty(k)$ lies on the local stable manifold
$W^s(y)=\{(\zeta, \iota):\ \zeta= 0\}$,
we obtain that $\zeta(k)=\zeta_\infty=0$ for any $k\ge 0$,
and thus,
\beqn
y_\infty(k)=
\begin{pmatrix} t_\infty(k) \\ \psi_\infty(k) \end{pmatrix}
= \lambda^{-k}
\begin{pmatrix}
C_t \\ C_\psi
\end{pmatrix}  + \cO(\lambda^{-1.5k}),
\eeqn
where we set
$C_t:=\iota_\infty\cos\theta_2$ and
$C_\psi:=\iota_\infty\sin\theta_2$.

The first two relations in \eqref{C relation}, that is,
$C_\varphi / C_s=\tan\theta_1$ and
$C_\psi / C_t=\tan\theta_2$, directly follows from \eqref{theta1} and \eqref{theta2}.
Moreover, by \eqref{linearization 2},
we have $\iota_\infty=\lambda_{1, s}\eta_\infty$, and hence
the third relation in \eqref{C relation}, that is,
$C_t/C_s=\lambda_{1, s}\cos\theta_2/\cos\theta_1$,
directly follows from \eqref{lambdazw} and \eqref{lambda relation}. \\

\medskip

\noindent (b) We now show the estimates for the palindromic orbits $\gamma_n$.
Note that there are involutions given by $\cI_1: (s, \varphi)\mapsto (s, -\varphi)$
and $\cI_2: (t, \psi)\mapsto (t, -\psi)$.
Let $\cJ_1(\xi, \eta):=\Theta_1^{-1}\circ \Psi_1^{-1} \circ \cI_1 \circ \Psi_1 \circ \Theta_1(\xi, \eta)$,
then $\cJ_1(0, 0)=(0, 0)$ and
{\allowdisplaybreaks
\begin{align*}
\cJ_1\begin{pmatrix} \xi \\ \eta \end{pmatrix}
&=\begin{pmatrix}
\ \ \cos\theta_1 & \cos\theta_1 \\
-\sin\theta_1 & \sin\theta_1
\end{pmatrix}^{-1}
\begin{pmatrix} 1 & 0 \\ 0 & -1 \end{pmatrix}
\begin{pmatrix}
\ \ \cos\theta_1 & \cos\theta_1 \\
-\sin\theta_1 & \sin\theta_1
\end{pmatrix}
\begin{pmatrix} \xi \\ \eta \end{pmatrix} + \cO\left((|\xi|^2+ |\eta|^2)^{\frac32} \right)\\
&=\begin{pmatrix} \eta \\ \xi \end{pmatrix} + \cO\left(\max\{|\xi|, |\eta|\}^{\frac32} \right), \\
\text{and} & \ \
\cJ_1\begin{pmatrix} \xi \\ \eta \end{pmatrix} - \cJ_1\begin{pmatrix} \xi' \\ \eta' \end{pmatrix}=
\begin{pmatrix} \eta-\eta' \\ \xi-\xi' \end{pmatrix} + \cO\left(\max\{|\xi-\xi'|, |\eta-\eta'|\}^{\frac32} \right).
\end{align*}
}We have similar properties for
$\cJ_2(\zeta, \iota):=\Theta_2^{-1}\circ \Psi_2^{-1} \circ \cI_2 \circ \Psi_2 \circ \Theta_2(\zeta, \iota)$.\\

By time reversibility \eqref{time-rev s},
$(s_n(n+2-k), \varphi_n(n+2-k)=\cI_1(s_n(k), \varphi_n(k))$ for any $1\le k\le [\tfrac{n}{2}]+1$,
and hence correspondingly,
{\allowdisplaybreaks
\begin{align*}
(\lambda^{n+2-k}\xi_n, \lambda^{k-n-2}\eta_n)
&=(\xi_n(n+2-k), \eta_n(n+2-k)) \\
&=\cJ_1(\xi_n(k), \eta_n(k)) \\
&=\cJ_1(\lambda^k \xi_n, \lambda^{-k}\eta_n) \\
&=(\lambda^{-k}\eta_n, \lambda^k \xi_n)
+\cO\left(\max\{ \lambda^{1.5k} |\xi_n|^{1.5}, \ \lambda^{-1.5k}|\eta_n|^{1.5} \} \right).
\end{align*}
}It follows that $\xi_n=\cO(\lambda^{-n})$ by taking $k=1$,
and further, $\xi_n=\eta_n\lambda^{-n-2}+\cO\left( \lambda^{-1.25n}\right)$
by taking $k=\lfloor (n+2)/2 \rfloor$.

On the other hand, recall that $x_n(0)=(s_n(0), 0)$ and $x_\infty(0)=(s_\infty(0), 0)$,
which are unchanged under the involution $\cI_1$. Correspondingly,
$\cJ_1(\xi_n, \eta_n)=(\xi_n, \eta_n)$ and
$\cJ_1(0, \eta_\infty)=(0, \eta_\infty)$. Then
\beqn
(\xi_n, \eta_n-\eta_\infty)
=\cJ_1(\xi_n, \eta_n)-\cJ_1(0, \eta_\infty)
=(\eta_n-\eta_\infty, \xi_n)
+ \cO\left(\max\left\{|\xi_n|, |\eta_n-\eta_\infty|\right\}^{\frac32} \right),
\eeqn
which implies that
\beqn
\eta_n-\eta_\infty=\xi_n+\cO(\lambda^{-1.5n})
=\eta_n\lambda^{-n-2}  + \cO(\lambda^{-1.25n}).
\eeqn
Thus,
$\eta_n-\eta_\infty=\eta_\infty \lambda^{-n-2}+ \cO(\lambda^{-1.25n})$,
and $\xi_n=\eta_\infty \lambda^{-n-2}+ \cO(\lambda^{-1.25n})$. Hence
 for any $0\le k\le \lfloor \frac{n}{2}\rfloor +1$,
{\allowdisplaybreaks
\begin{align*}
x_n(k)-x_\infty(k)
&=\begin{pmatrix} s_n(k) \\ \varphi_n(k) \end{pmatrix} -
\begin{pmatrix} s_\infty(k) \\ \varphi_\infty(k) \end{pmatrix} \\
&=\Psi_1\circ \Theta_1 \begin{pmatrix} \xi_n(k) \\  \eta_n(k) \end{pmatrix}
- \Psi_1\circ \Theta_1 \begin{pmatrix} \xi_\infty(k) \\  \eta_\infty(k) \end{pmatrix}\\
&=\begin{pmatrix}
\ \ \cos\theta_1 & \cos\theta_1 \\
-\sin\theta_1 & \sin\theta_1
\end{pmatrix}
\begin{pmatrix} \lambda^k \xi_n \\ \lambda^{-k} (\eta_n-\eta_\infty) \end{pmatrix}
+ \cO\left(\max\left\{\lambda^{1.5(k-n)}, \lambda^{-1.5k-n} \right\}\right)
\\
&=\lambda^{k-n-2}
\begin{pmatrix}
\eta_\infty \cos\theta_1(\lambda^{-2k}+1) \\
\eta_\infty \sin\theta_1(\lambda^{-2k}-1)
\end{pmatrix}
+ \cO\left(\lambda^{0.5k-n} \right) \\
&= \lambda^{k-n}
\begin{pmatrix}
C_s\lambda^{-2} (\lambda^{-2k}+1) \\ C_\varphi \lambda^{-2} (\lambda^{-2k}- 1)
\end{pmatrix}
+ \cO\left( \lambda^{0.5k-n} \right).
\end{align*}
}The estimates of $y_n(k)-y_\infty(k)$ can be shown in a similar fashion.
The proof of this lemma is complete.
\end{proof}

\begin{remark}\label{rem: finer est ga n} 
We make some comments on the consequences of 
\eqref{finer est ga n infty0}.
\begin{itemize}
\item[(a)]
When 
$k\approx \frac{n}{2}$, i.e., 
$\frac{n}{2}-m \le k\le \frac{n}{2}$ for some fixed $m\in \IN$, 
we have 
\beq\label{finer est ga n infty'}
\begin{split}
x_n(k)=\left( C_s \left(\lambda^{-k} + \lambda^{k-n-2} \right), 
C_\varphi \left(\lambda^{-k} - \lambda^{k-n-2} \right) \right)+ \cO(\lambda^{-0.75n}), \\
y_n(k)=\left( C_t \left(\lambda^{-k} + \lambda^{k-n-1}\right), 
C_\psi \left(\lambda^{-k} - \lambda^{k-n-1}\right) \right) + \cO(\lambda^{-0.75n}).
\end{split}
\eeq
From the proof of \eqref{finer est ga n infty0}, it is not hard to see that the above estimates 
hold for $\frac{n}{2} \le k\le \frac{n}{2}+m$ as well, after possibly enlarging the constants 
in $\cO(\cdot)$. 

\item[(b)]
When
$\eps n \le k\le \frac{n}{2}$ for some $\eps\in (0, \frac12)$, 
we have 
\beq\label{xn-xinf0}
x_n(k)-x_\infty(k) 
=\lambda^{k-n} \left[ (C_{s, k}, C_{\varphi, k}) + o(1) \right], \ \ \text{as}\ n\to\infty.
\eeq
It immediately follows that $x_n(k)$ asymptotically lies on a straight line which goes through $x_\infty(k)$ 
and has direction vector $(C_{s, k}, C_{\varphi, k})$ when $n\to\infty$ and $k/n\ge \eps$. 
Similar formulas hold for $y_n(k)-y_\infty(k)$. 

However, when $k$ is a finite value, we can only conclude that
$x_n(k)-x_\infty(k)=\cO(\lambda^{k-n})$. In the following lemma, 
We shall see that \eqref{xn-xinf0} still holds for all $k\in [0, \lfloor n/2 \rfloor +1]$,
though the new coefficients are in lack of precise formulas. 
\end{itemize}
\end{remark}

\begin{lemma}\label{lem: finer est'} 
The following estimates hold for the palindromic orbit $\gamma_n$:  as $n\to\infty$, 
{\allowdisplaybreaks
\beq\label{finer est ga n infty''''}
\begin{split}
x_n(k)-x_\infty(k)=\lambda^{k-n} \left[v_\infty(2k) + o(1)\right] , \ \ \ \ \ \
& \ \ k=0, 1, \dots, \left\lfloor{\frac{n}{2}}\right\rfloor+1, \\
y_n(k)-y_\infty(k)=\lambda^{k-n} \left[ v_\infty(2k+1) + o(1)\right], \  
& \ \ k= 0, 1, \dots, \left\lfloor{\frac{n}{2}}\right\rfloor,
\end{split}
\eeq
}where the vectors $v_\infty(m)\in \IR^2$, $m=0^\pm, 0, 1, 2, \dots$, 
has uniformly bounded magnitudes. 
\end{lemma}

\begin{proof}
Recall that $x_n(0)=(s_n(0), 0)$ and $x_\infty(0)=(s_\infty(0), 0)$, which are all stay uniform distance away from the singular set (i.e. corner  points).
It suffices to show that there is $s_\infty\in \IR$ such that 
\beq\label{tn-tinfty}
s_n(0)-s_\infty(0)=\lambda^{-n} \left[ s_\infty  + o(1)\right].
\eeq
Indeed, if \eqref{tn-tinfty} holds, then by setting $v_\infty(0)=(s_\infty, 0)\in \IR^2$ we have 
\beqn
x_n(0) - x_\infty(0) = (s_n(0)- s_\infty(0), 0)=
\lambda^{-n} \left[ v_\infty(0) + o(1)\right]. 
\eeqn
Furthermore,
for any  $1\le k\le \left\lfloor{\frac{n}{2}}\right\rfloor+1$,
there exists $C_k'>0$ such that
{\allowdisplaybreaks
\begin{eqnarray*}
x_n(k) - x_\infty(k) 
&=& F^{2k}(x_n(0)) - F^{2k}(x_\infty(0)) \\
&=& D_{x_\infty(0)}F^{2k} ( x_n(0) - x_\infty(0) ) + C_k'\left(  x_n(0) - x_\infty(0) \right)^2  \\
&=& \lambda^{k-n} \cdot \lambda^{-k}  D_{x_\infty(0)}F^{2k} v_\infty(0) 
+ o(\lambda^{-n}) \\
&=& \lambda^{k-n} \left[ v_\infty(2k) + o(1)\right],
\end{eqnarray*}
}where in the last identity we set 
$v_\infty(2k):= \lambda^{-k}  D_{x_\infty(0)}F^{2k}v_\infty(0)$, 
and $D_{x_\infty(0)}F^{2k}$ is the differential of $F^{2k}$ evaluated at $x_\infty(0)$.
Similarly, for any $1\le k\le \left\lfloor{\frac{n}{2}}\right\rfloor$, we set
$v_\infty(2k+1):= \lambda^{-k}  D_{x_\infty(0)}F^{2k+1}v_\infty(0)$, then
\beqn
 y_n(k)-y_\infty(k)  = F^{2k+1}(x_n(0)) - F^{2k+1}(x_\infty(0))
 =\lambda^{k-n} \left[ v_\infty(2k+1) + o(1)\right].
 \eeqn

In the rest of the proof, we concentrate on how to obtain \eqref{tn-tinfty}. 
The proof is based on a geometric arguments by considering the growth of 
a special dispersing wave front. 
To be precise, 
let $W$ be the wave front between $s=s_n(0)$ and $s=s_\infty(0)$, 
associated with zero angles. 
Without loss of generality, we may assume $s_n(0)>s_\infty(0)$, then  
\beqn
W=\left\{x(s)=s(1, 0) \left| \ s_\infty(0)\le s\le s_n(0) \right. \right\}.
\eeqn
It is clear  that $W$ is an unstable curve connecting $x_n(0)$ and $x_\infty(0)$,
such that  $F^{2m}(W)$ is an unstable curve connecting $x_n(m)$ and $x_\infty(m)$ for any $1\le m\le n$. 
Moreover, 
\beqn
\left| F^{2m}(W)\right| =\int_{s_\infty(0)}^{s_n(0)} 
\left\| D_{x(s)} F^{2m} (1, 0)  \right\| ds. 
\eeqn

On the one hand, 
the semi-orbit $\gamma_\infty$ is homoclinic to the period two orbit $\gamma^*$
whose Lyapunov exponent along the unstable direction is equal to $\frac12 \log \lambda$, 
and the vector $(1, 0)$ is in the unstable cone $\CcC^u(x_\infty(0))$, 
there is a constant $C_\infty>0$ such that 
\beqn
\left\| D_{x_\infty(0)} F^{2m} (1, 0) \right\| = \lambda^m C_\infty  +o(1), \ \
\text{as} \ m\to\infty.
\eeqn
On the other hand, by Lemma~\ref{finer est ga n}, we have 
$x_n(m)-x_\infty(m)=\cO(\lambda^{m-n})$ for any $0\le m\le \lfloor \frac{n}{2} \rfloor+1 $.
Furthermore,  recall that $\cI_1$ is the involution given by $\cI_1(s, \varphi)=(s, -\varphi)$,
and the time reversibility~\eqref{time-rev s} implies that
$x_n(n+2-m)=\cI_1 x_n(m)$. Hence 
$ x_n(n+2-m) - \cI_1 x_\infty(m) =\cO(\lambda^{ m-n})$ 
for any $1\le m\le \lfloor \frac{n}{2}  \rfloor +1$.
Note that the reflected homoclinic semi-orbit $\cI \gamma_\infty$ 
have Lyapunov exponent $\frac12\log \lambda$ along the unstable direction as well. 
Since the curvature of unstable curves $\{F^{2m}(W)\}_{0\le m\le n}$ are uniformly bounded, 
we have for any $m=1, 2, \dots, \lfloor \frac{n}{2} \rfloor+1 $
and for any $s\in [s_\infty(0), s_n(0)]$, 
\beqn
F^{2m}(x(s)) - x_\infty(m) =\cO(\lambda^{m-n})  \ \ \text{and} \ \
F^{2n-2m}(x(s)) - \cI_1 x_\infty(m) =\cO(\lambda^{m-n}).
\eeqn
Therefore, we have 
{\allowdisplaybreaks
\begin{eqnarray*}
& & \left\| D_{x(s)} F^{2n} (1, 0) \right\|   \\
&=& \left\| \prod^{1}_{m=n+1-\lfloor \frac{n}{2}\rfloor } D_{F^{2(n-m)}(x(t))}F^2 
\prod_{m=0}^{\lfloor \frac{n}{2}\rfloor } D_{F^{2m} (x(t))}  F^2  \cdot (1, 0) \right\| \\
&=& \left\| \prod^{1}_{m=n+1-\lfloor n/2\rfloor } \left( D_{\cI_1 x_\infty(m)} F^2 + \cO(\lambda^{m-n})\right)
\prod_{m=0}^{\lfloor \frac{n}{2}\rfloor } \left(D_{x_\infty(m)}  F^2 + \cO(\lambda^{m-n})\right)
\cdot (1, 0) \right\| \\
&=& \lambda^n  (C_\infty +o(1)),
\end{eqnarray*}
}and hence 
\beqn
\left|F^{2n}W\right| = \int_{s_\infty(0)}^{s_n(0)} 
\left\| D_{x(t)} F^{2n} (1, 0)  \right\| ds 
= \lambda^n (C_\infty + o(1)) (s_n(0) - s_\infty(0)). 
\eeqn
On the other hand, since $\lim\limits_{n\to\infty} x_\infty(n)=x$ and 
$\lim\limits_{n\to\infty} x_n (n) =\cI_1 x_\infty(2)$, then 
$F^{2n}(W)$ converges to the unstable manifold $W'$ which connects  $x$ and $\cI_1 x_\infty(2)$,
and hence 
$\left|F^{2n}W\right|=|W'|+o(1)$ as $n\to \infty$. 
Therefore, 
$s_n(0) - s_\infty(0)=\lambda^{-n}\left[s_\infty + o(1)\right]$
if we set $s_\infty=|W'|/C_\infty$. 
The proof of this lemma is complete. 
\end{proof}

\bigskip

\section{The Linearized Isospectral Functionals}\label{sec:LIO}

To prove Squash Rigidity Theorem, we introduce the linearized isospectral functionals corresponding to
periodic orbits of the Bunimovich squash-type stadia.

\subsection{Normalization, Parametrization and Deformation}\label{sec: parametrization}

Let $\Omega$ be a Bunimovich squash-type stadium in $\cM^m_{ss}$, where $m\ge 3$.
By Lemma~\ref{lem: gamma uniq},
there is a unique maximal period two orbit
$\gamma^*=\overline{AB}$ for $\Omega$. Without loss
of generality, we may assume that $A$ is the origin of $\IR^2$,
and $B$ lies on the positive horizontal semiaxis of $\IR^2$.

Let $\left\{\hOmega_\mu\right\}_{|\mu|\le 1}$ be
a $C^1$ one-parameter family in $\cM^m_{ss}$  such that $\hOmega_0=\Omega$.
That is, each Bunimovich squash-type stadia $\hOmega_\mu$ 
satisfies Assumptions $\mathrm{(I_{ss})(II)(III_{ss})}$
and the convex arcs $\Gamma_1$ and $\Gamma_2$ are $C^m$ smooth.
We would like to normalize the position of this family as follows.
Lemma~\ref{lem: gamma uniq} shows that
$\hOmega_\mu$ has a unique maximal period two orbit
$$\hgamma^*(\mu)=\overline{\hA(\mu) \hB(\mu)}$$
for any $\mu\in [-1, 1]$.
Then there is a unique orientation-preserving planar isometry $\cT_\mu$
such that $\cT_\mu\left(\hA(\mu) \right)=A$,
$B(\mu):=\cT_\mu\left(\hB(\mu) \right)$
lies on the positive horizontal semiaxis of $\IR^2$, 
and $\cT_\mu\left(W(\mu) \right)$ lies on positive vertical semiaxis,
where  $W(\mu)$ be the vector perpendicular
to the vector $\overrightarrow{\hA(\mu)\hB(\mu)}$ in the counter-clockwise direction.
It is easy to see from the proof of Lemma~\ref{lem: gamma uniq} that
the dependence $\mu\mapsto \left(\hA(\mu), \hB(\mu) \right)$ is $C^1$ smooth,
and so is the mapping $\mu\mapsto \cT_\mu$ in the space of planar isometries, i.e.,
if we $\cT_\mu(v)=E(\mu) v + F(\mu)$ for any $v\in \IR^2$, where
$E(\mu)\in \mathrm{SO}(2, \IR)$ and $F(\mu)\in \IR^2$, then the mappings
$\mu\mapsto E(\mu)$ and $\mu\mapsto F(\mu)$ are both $C^1$ smooth.
We then denote
\beqn
\Omega_\mu=\cT_\mu\hOmega_\mu, \  \ \text{for} \ \ \mu\in [-1, 1],
\eeqn
and call $\left\{\Omega_\mu\right\}_{|\mu|\le 1}$ the normalized family of
the original family $\left\{\hOmega_\mu\right\}_{|\mu|\le 1}$.
Note that $\Omega_0=\Omega$ since $\cT_0=\mathrm{Id}_{\IR^2}$.
Also, the unique maximal period
two orbit of $\Omega_\mu$ is given by $\gamma^*(\mu)=\overline{A B(\mu)}$.
It is clear that the normalized
family $\left\{\Omega_\mu\right\}_{|\mu|\le 1}$ is $C^1$ smooth with respect to
the parameter $\mu$.

Let us now parametrize the $C^1$ one-parameter normalized family $\left\{\Omega_\mu\right\}_{|\mu|\le 1}$
of the Bunimovich squash-type stadia such that $\Omega_0=\Omega$.
We denote
\beqn
\p\Omega_\mu=\Gamma_1(\mu)\cup \Gamma_3(\mu)\cup \Gamma_2(\mu)\cup \Gamma_4(\mu),
\eeqn
where $\Gamma_1(\mu)$ and $\Gamma_2(\mu)$ are the two convex arcs 
and $\Gamma_3(\mu)$ and $\Gamma_4(\mu)$ are flat boundaries. 
We may choose a parametrization $\Phi: [-1, 1]\times J\to \IR^2$,
where the interval $J$ is a union of four consecutive sub-intervals, i.e.,
$J=J_1\cup J_3\cup J_2\cup J_4$, such that
\begin{itemize}
\item[(1)] For any $\mu\in [-1, 1]$ and $i=1, 2, 3, 4$, we have $\Phi(\mu, J_i)=\Gamma_i(\mu)$;
\item[(2)] The mapping $\mu\mapsto \Phi(\mu, \cdot)$ is $C^1$ smooth;
\item[(3)] Set $\Phi_i=\Phi|_{J_i}$ for $i=1, 2$.
For any fixed $\mu\in [-1, 1]$, the map $r\mapsto \Phi_i(\mu, r)$ is $C^m$ smooth.
\end{itemize}
We further define {\it the deformation function} $\bn: [-1, 1]\times J \to \IR$
of the normalized family $\left\{\Omega_\mu\right\}_{|\mu|\le 1}$ by
\beq\label{def deformation function}
\bn(\mu, r)=\bn_{\Phi}(\mu, r):=\langle \p_\mu\Phi(\mu, r), \ N(\mu, r) \rangle,
\eeq
where $\langle\cdot, \cdot \rangle$ is the standard scalar product in $\IR^2$
and $N(\mu, r)$ is the out-going unit normal vector to $\p\Omega_\mu$
at the point $\Phi(\mu, r)$.
It is obvious that $\bn(\mu, r)$ is continuous in $\mu$.
Moreover, for any $\mu\in [-1, 1]$, the function
$r\mapsto \bn(\mu, r)$ is $C^m$ smooth on $J_1\cup J_2$.

\begin{remark}\label{rem: bn}
There are certainly many different choices of the parametrization that obey 
the above rules (1)-(3). Different parametrization $\Phi$ would surely give 
different deformation function $\bn_\Phi$. Nevertheless, we shall only focus on 
the vanishing property of the deformation function, which does not depend on 
the parametrization at all. More precisely,  
let $\Phi$ and $\Phi'$ be two parametrizations with domain interval
$$
J=J_1\,\cup \,J_3\,\cup \,J_2\, \cup J_4\qquad \text{ and }
\qquad J'=J_1'\, \cup J_3' \,\cup J_2' \,\cup \, J_4',
$$
respectively, by Lemma 3.3 in \cite{dSKW17}, we have $n_{\Phi}|_{J_i}\equiv 0$ 
if and only if $n_{\Phi'}|_{J_i'}\equiv 0$ for any $i=1, 2, 3, 4$. 
\end{remark}

In the rest of the paper, we shall simply denote the deformation function by $\bn(\mu, r)$ 
if the parametrization $\Phi$ is clear. 
To prove the Squash Rigidity Theorem,  
the following lemma asserts that 
we only need to show the vanishing of 
$\bn$ on $J_1\cup J_2$.

\begin{lemma}\label{lem: main n vanish 00}
If $\bn(\mu, \cdot) \equiv 0$ on $J_1\cup J_2$ for any $\mu\in [-1, 1]$,
then the normalized family $\left\{\Omega_\mu\right\}_{|\mu|\le 1}$ is constant, 
and, thus, the original family $\left\{\hOmega_\mu\right\}_{|\mu|\le 1}$ is isometric.
\end{lemma}

\begin{proof}
By the definition of the deformation function, 
the condition $\bn(\mu, \cdot) \equiv 0$ on $J_i$, $i=1, 2$, implies that all the arcs 
$\Gamma_i(\mu)$ should lie on a common longest arc $\Gamma^0_i$. 
It is obvious that $\Gamma^0_i$ is a strictly convex curve. 
We show that in fact $\Gamma_i(\mu)=\Gamma^0_i$ for all $\mu\in [-1, 1]$ as follows.
Let $P_{14}^0$ be the left-top point of $\Gamma_1^0$, by compactness, 
there is $\mu_0\in [-1, 1]$ such that $P_{14}(\mu_0)=P_{14}^0$, where
$P_{14}(\mu)$ denotes the joint point of $\Gamma_1(\mu)$ and $\Gamma_4(\mu)$.
We claim that $P_{14}(\mu)=P_{14}^0$  for all $\mu\in [-1, 1]$. Otherwise, we suppose that 
there is $\mu_1\in [-1, 1]$ such that its joint point $P_{14}(\mu_1)$ cannot approach $P_{14}^0$. 
On the one hand, the strict convexity of $\Gamma^0_1$ and Assumption (II) imply that the slope of 
the flat boundary $\Gamma_4(\mu_1)$ should be bigger  than  that of $\Gamma_4(\mu_0)$.
On the other hand, both $\Gamma_2(\mu_0)$ and $\Gamma_2(\mu_1)$ lie on
the other longest arc $\Gamma^0_2$, 
which implies that $\Gamma^0_2$ should be under the flat boundary $\Gamma_4(\mu_0)$. 
Therefore, 
$\Gamma_4(\mu_1)$ and $\Gamma_2(\mu_1)$ cannot close up to form a 
Bunimovich squash-type stadium (see Fig.~\ref{fig:stadia10}). In a similar fashion, we can show 
$P_{ij}(\mu)=P_{ij}^0$  for all $\mu\in [-1, 1]$, $i=1, 2$ and $j=3, 4$.

\begin{figure}[h]
\begin{center}
\includegraphics[width=\textwidth]{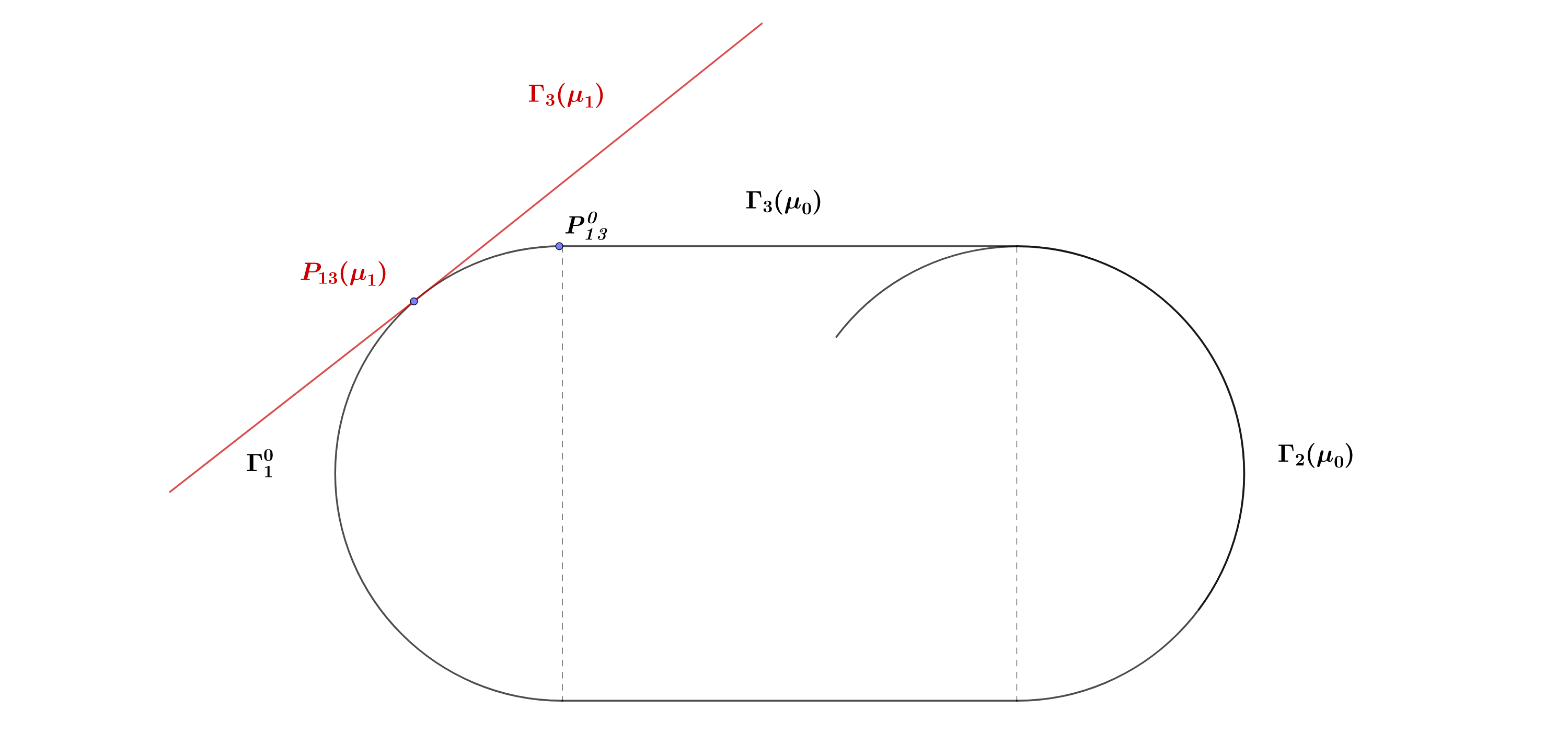}
\end{center}
\caption{Proof of $P_{14}(\mu)=P_{14}^0$}
\label{fig:stadia10}
\end{figure}
 
It directly follows that $\Gamma_i(\mu)=\Gamma_i(0)$, $i=1, 2, 3, 4$, 
for all $\mu\in [-1, 1]$. 
Hence the normalized family $\left\{\Omega_\mu\right\}_{|\mu|\le 1}$ is a constant family, that is,
$\Omega_\mu=\Omega_0$. Therefore,
$\hOmega_\mu=\cT_\mu^{-1}\Omega_\mu=\cT_\mu^{-1}\Omega_0$,
which implies that the original family $\left\{\hOmega_\mu\right\}_{|\mu|\le 1}$ is an isometric family.
\end{proof}

Let us recall the statement of Squash Rigidity Theorem:
for any homothety ratio $\chi>0$, 
if a family $\left\{\hOmega_\mu\right\}_{|\mu|\le 1}$ of
Bunimovich squash-type stadia in $\cM^\omega_{ss}(\chi)$ is
dynamically isospectral, then it is an isometric family. By
Lemma~\ref{lem: main n vanish 00}, it is sufficient to show that
the normalized family $\left\{\Omega_\mu\right\}_{|\mu|\le 1}$
is a constant family, or equivalently, $\bn\equiv 0$
on $J_1\cup J_2$. Note that if the original family
$\left\{\hOmega_\mu\right\}_{|\mu|\le 1}$ is dynamically
isospectral, so is the normalized family
$\left\{\Omega_\mu\right\}_{|\mu|\le 1}$
since isometries do not change length spectra.
Moreover, the normalized family
$\left\{\Omega_\mu\right\}_{|\mu|\le 1}$
still belongs to the class $\cM^\omega_{ss}(\chi)$.
For the rest of this section and Section~\ref{sec: proof of thm1},
we shall focus on the normalized family $\left\{\Omega_\mu\right\}_{|\mu|\le 1}$.

\subsection{Functionals related to Length Spectra}
\label{sec: functional spectra}

We first show the following basic fact
for the length spectrum of a Bunimovich squash-type stadium.

\begin{lemma}\label{lem: cL zero} For any Bunimovich squash-type
stadium $\Omega$, its length spectrum $\cL(\Omega)$ has zero Lebesgue measure.
\end{lemma}
\begin{proof} The proof is almost the same as that of Lemma 4.1 in \cite{dSKW17},
with the only difference that the  parametrization
$\Phi$ of the boundary $\p\Omega$ is not $C^2$ at the four gluing points.
Nevertheless, there are at most countably many periodic orbits through those points,
and thus $\cL(\Omega)$ has zero Lebesgue measure.
\end{proof}

By Lemma~\ref{lem: cL zero} and the intermediate value theorem, we immediately get

\begin{lemma}\label{lem: Darboux}
For any continuous function $\Delta: [-1, 1]\to \IR$,
if Range$(\Delta)\subset \cL(\Omega)$, then $\Delta\equiv$constant.
\end{lemma}

Now for any Bunimovich squash-type stadium $\Omega$ with boundary
$\p\Omega=\Gamma_1\cup \Gamma_3\cup \Gamma_2\cup \Gamma_4$,
we would like to construct functionals corresponding to symbolic codes.
For any $q\ge 2$, we introduce the length function associated to
a symbolic code $\bi=(i_1, i_2, \dots, i_{q})\in \{1, 2, 3, 4\}^q$, that is,
\beqn
L_{\bi}(\gamma)=\sum_{k=1}^{q} \tau(r_k, r_{k+1})
\eeqn
for $\gamma=(r_1, r_2, \dots, r_{q})$ such that $r_k\in \Gamma_{i_k}$ for $1\le k\le q$,
where we set $r_{q+1}=r_1$ since $q+1 \equiv 1 \pmod q$.
Similar to Lemma~\ref{lem: gamma uniq} and Lemma~\ref{lem: gamma_n uniq},
we can show that
the maximal points of $L_{\bi}$ shall either give the maximal periodic orbits corresponding to $\bi$,
or a singular set containing some gluing points.

Let us now consider a $C^1$ one-parameter normalized family
$\left\{\Omega_\mu\right\}_{|\mu|\le 1}$ of the Bunimovich squash-type stadia
in $\cM^m_{ss}$, where $m\ge 3$. 
We say that a symbolic code $\bi$ is {\it good} for
the family $\left\{\Omega_\mu\right\}_{|\mu|\le 1}$ if
\begin{itemize}
\item for every $\mu\in [-1, 1]$, 
there exists a unique maximal periodic orbit $\gamma_\bi(\mu)$ corresponding to $\bi$;
\item the dependence $\mu\mapsto \gamma_\bi(\mu)$ is $C^1$ smooth.
\end{itemize}
Then we can define a $C^1$ function by
\beqn
\Delta(\mu; \bi)=L^\mu_{\bi}(\gamma_{\bi}(\mu)), \ \ \text{for any}\ \mu\in [-1, 1].
\eeqn
Moreover, let us define the function
\beq\label{def G}
G(\mu; z)=G(\mu; r, \varphi)=\bn(\mu, r)\cos\varphi
\eeq
for a collision point $z=(r, \varphi)$,
where $\bn$ is the deformation function given by \eqref{def deformation function}.
By Proposition 4.6 in \cite{dSKW17}, we obtain
\beq\label{def ell bi}
\frac12\Delta'(\mu; \bi)=\frac12\p_\mu L^\mu_{\bi}(\gamma_{\bi}(\mu))
=\sum_{z\in \gamma_{\bi}(\mu)} G(\mu; z).
\eeq

If the family $\left\{\Omega_\mu\right\}_{|\mu|\le 1}$  is dynamically isospectral, that is,
$\cL(\Omega_\mu)=\cL(\Omega)$ for any $\mu\in [-1, 1]$,
then $\mu\mapsto \Delta(\mu; \bi)$ is constant by Lemma~\ref{lem: Darboux}.
Therefore, by \eqref{def ell bi},
\beq\label{sum G vanish}
\sum_{z\in \gamma_{\bi}(\mu)} G(\mu; z)= 0, \ \ \text{for any}\ \ \mu\in [-1, 1].
\eeq

We now investigate the unique maximal periodic orbits corresponding to
special good codes that we have studied in earlier sections:
\begin{itemize}[leftmargin=0.75cm]
\item[(1)]
Functionals related to the maximal period two orbit $\gamma^*(\mu)$:
note that the corresponding symbolic code is $\overline{12}$.
Since $\Delta(\mu; \overline{12})$ is constant in $\mu$, and
 the angles for $\gamma^*(\mu)$ are of zero degree, we have for any $\mu\in [-1, 1]$,
\beq\label{def ell12}
\frac12\Delta'(\mu; \overline{12})
=\sum_{z\in \gamma^*(\mu)} G(\mu; z)=\bn_1(\mu, 0)+\bn_2(\mu, 0).
\eeq
\item[(2)]
Functionals related to the palindromic periodic orbits $\gamma_{n}(\mu)$,
which are studied in Section~\ref{sec:palindromic} for the table $\Omega_\mu$.
The symbolic codes $\bi_n$ for $\gamma_{n}(\mu)$ are given by \eqref{def palin},
and the corresponding functional is given by
$$
\frac12\Delta'(\mu; \bi_{n})=\sum_{z\in \gamma_n(\mu)}  G(\mu; z).
$$
\end{itemize}

For notational simplicity, we shall omit $\mu$
and just write $\Gamma_i$, $\bn(r)$, $\gamma^*$, $\gamma_n$,  $G(z)$,
instead of $\Gamma_i(\mu)$, $\bn(\mu, r)$,  $\gamma^*(\mu)$, $\gamma_n(\mu)$,
$G(\mu; z)$ respectively. Also,
we briefly write the summation $\sum_{z\in \gamma} G(z)$ by $\sum_\gamma G$.

\subsection{Arclength Parameterization on $\Gamma_1$ and $\Gamma_2$}\label{sec: arclength p}

Recall that the deformation function $\bn=\bn_\Phi$ 
would have different formulas under different pararmetrizations $\Phi$.
As we have explained in Remark~\ref{rem: bn}, Lemma 3.3 in \cite{dSKW17} shows that
the vanishing property of $\bn$ does not depend on the parametrizations at all.
In order to apply the formulas that we previously obtained in Section~\ref{sec: period two} and~\ref{sec:palindromic}, 
we shall introduce the parametrizations separately on 
$\Gamma_1(\mu)$ and $\Gamma_2(\mu)$
which are of arclength parametrization  for some a priori fixed parameter $\mu_0$ as follows.

More precisely, 
we recall that the normalized family $\left\{\Omega_\mu\right\}_{|\mu|\le 1}$
belongs to the class $\cM^m_{ss}(\chi)$,
where $m\ge 3$ and $\chi$ is an a priori fixed homothety ratio.
Given a fixed $\mu_0\in [-1, 1]$, we first introduce a parametrization
$s\mapsto \Phi_1(\mu, s)$ with $s\in J_1$ on $\Gamma_1(\mu)$
such that 
$\Phi_1(\mu, 0)=A$ for all $\mu\in [-1, 1]$
and 
$\Phi_1(\mu_0, s)$ is of arclength parametrization on $\Gamma_1(\mu_0)$.
Recall that 
the unique maximal period two orbit of $\Omega_\mu$ is given by
$\gamma^*(\Omega_\mu)=\overline{A B(\mu)}$,
where 
$A$ is the origin of $\IR^2$ and
$B(\mu)$ lies on the positive horizontal semiaxis of $\IR^2$.
By Assumption (IV), there is an orientation preserving $\chi$-homothety $\cS_\mu$ transforms 
a sub-curve of $\Gamma_1(\mu)$ near $A$ onto a sub-curve of $\Gamma_2(\mu)$ near $B(\mu)$.
We note that the tangent (resp. normal) vector of both $\Gamma_1(\mu)$ at $A$  and $\Gamma_2(\mu)$ at $B(\mu)$
is vertical (resp. horizontal), and 
$\cS_\mu$ transforms the tangent/normal vector of $\Gamma_1(\mu)$ at $A$ 
into the tangent/normal vector of $\Gamma_2(\mu)$ at $B(\mu)$.
Therefore, if we set 
\beqn
\widetilde{\Phi}_2(\mu, t)=\cR_\mu\left[\chi \Phi_1\left(\mu, t/\chi\right)\right]
\ \ \text{for}\ t \ \text{close to} \ 0,
\eeqn
where $\cR_\mu$ is the counter-clockwise rotation 
with center at the middle point of $\overline{A B(\mu)}$
by 180 degree,
then 
the graph of $\widetilde{\Phi}_2(\mu, t)$ coincides 
with $\Gamma_2(\mu)$ near  $B(\mu)$. 
Therefore, we can extend $t\mapsto \widetilde{\Phi}_2(\mu, t)$ to 
a parametrization $t\mapsto \Phi_2(\mu, t)$ with $t\in J_2$ on $\Gamma_2(\mu)$ 
for all $\mu\in [-1, 1]$ such that 
\beq
\label{eq: homo}
\Phi_2(\mu, t)=\cR_\mu\left[\chi \Phi_1\left(\mu,  t/\chi \right)\right]
\ \ \text{for}\ t \ \text{close to} \ 0.
\eeq
In particular, we have $\Phi_2(\mu, 0)=B(\mu)$ for all $\mu\in [-1, 1]$
and $\Phi_2(\mu_0, t)$ is of arclength parametrization on $\Gamma_2(\mu_0)$.
Of course, the parametrizations 
$s\mapsto \Phi_1(\mu, s)$ and $t\mapsto \Phi_2(\mu, t)$
need not be of arclength for other parameters $\mu\ne \mu_0$.
Under the new parametrizations, we denote the deformation function
on $\Gamma_1(\mu)$ and $\Gamma_2(\mu)$ by $\bn_1(\mu, s)$ and $\bn_2(\mu, t)$ respectively.
Below is an immediate consequence of the fact that the family
$\{\Omega_\mu\}_{|\mu|\le 1}$ is in the normalized position.

\begin{lemma}\label{lem: arclength}
$\bn_1(\mu, 0)=\bn_1'(\mu, 0)=0$ for any $\mu\in [-1, 1]$.
\end{lemma}
\begin{proof}
$\Phi_1(\mu, 0)=A$ implies that $\p_\mu \Phi_1(\mu, 0)=0$, and hence
\beqn
\bn_1(\mu, 0)=\langle \p_\mu\Phi_1(\mu, 0), \ N_1(\mu, 0) \rangle=0.
\eeqn
Furthermore, for any $\mu\in [-1, 1]$, the tangent vector
of $\Gamma_1(\mu)$ at $A$, denoted by $\p_s\Phi_1(\mu, 0)$,
is a downward vertical vector at $A$ in $\IR^2$.
Here we take the downward direction because $\p\Omega_\mu$
is parametrized along the counter-clockwise direction.
Therefore, $\p_\mu \p_s \Phi_1(\mu, 0)$ is perpendicular to 
the out-going unit normal vector $N_1(\mu, 0)$ and hence
\beqn
\bn_1'(\mu, 0)=\langle \p_\mu \p_s \Phi_1(\mu, 0), \ N_1(\mu, 0) \rangle+
\langle \p_\mu\Phi_1(\mu, 0), \ \p_s N_1(\mu, 0) \rangle  =0.
\eeqn
The proof of this lemma is completed.
\end{proof}

By Lemma~\ref{lem: main n vanish 00},
the Squash Rigidity Theorem is reduced to showing 
that $\bn(\mu, \cdot)\equiv 0$ on $\Gamma_1\cup \Gamma_2$
for any $\mu\in [-1, 1]$. 
Since the vanishing property of $\bn$ does not depend on the parametrization
and the choice of $\mu_0$ is arbitrary, 
it suffices to show that 
$\bn_1(\mu_0, s)=0$ and $\bn_2(\mu_0, t)=0$.
For simplicity,
we shall omit $\mu_0$ and simply write $\bn_1(s)$ and $\bn_2(t)$.
Note that $\bn_1$ and $\bn_2$ are both $C^m$ smooth.

\bigskip

\section{Proof of Squash Rigidity Theorem}\label{sec: proof of thm1}

\subsection{Reduction of Squash Rigidity Theorem}

For any homothety ratio $\chi>0$, 
let $\left\{\hOmega_\mu\right\}_{|\mu|\le 1}$ be a $C^1$
one-parameter family in $\cM^\omega_{ss}(\chi)$, and
$\left\{\Omega_\mu\right\}_{|\mu|\le 1}$ be its normalized family.
As introduced in the previous section, the deformation function
is denoted by $\bn_1(s)$ on $\Gamma_1$, and by
$\bn_2(t)$ on $\Gamma_2$. Note that both $\bn_1(s)$ and
$\bn_2(t)$ are analytic. By Lemma~\ref{lem: main n vanish 00},
the Squash Rigidity Theorem
is reduced to the following.

\begin{proposition}\label{prop: main n vanish}
If the family $\left\{\Omega_\mu\right\}_{|\mu|\le 1}$ is dynamically
isospectral, then $\bn_1\equiv 0$ and $\bn_2\equiv 0$.
\end{proposition}

Note that the dynamically isospectral property implies
\eqref{sum G vanish}, that is, the sum of $G$ vanishes over
the unique maximal periodic orbit $\gamma_\bi$
which corresponds to a good symbolic code $\bi$.
In the analytic class $\cM^\omega_{ss}(\chi)$ of Bunimovich
squash-type stadia, it turns out that Condition
\eqref{sum G vanish} over the period two orbit $\gamma^*$ and
the palindromic orbits $\gamma_n$
are sufficient to establish the dynamical spectral rigidity. More precisely, we have

\begin{proposition}\label{lem van der}
If $\sum_{\gamma^*} G=\sum_{\gamma_n} G=0$ for any $n\ge 1$,
then
\beqn
\bn_1^{(d)}(0)=\bn_2^{(d)}(0)=0, \ \ \text{for any} \ d\ge 0.
\eeqn
\end{proposition}

We remark that Proposition~\ref{lem van der} also holds in the class $\cM^\infty_{ss}(\chi)$.
In the analytic class $\cM^\omega_{ss}(\chi)$,
it immediately follows that $\bn_1$ and $\bn_2$ both vanish since they are analytic,
which proves Proposition~\ref{prop: main n vanish} and thus
Squash Rigidity Theorem. In the rest of this section, we prove
Proposition~\ref{lem van der}.

\subsection{Cancellations by Interpolations}

\subsubsection{Sums of $G$ over $\gamma_n$}\label{sec: sum G}

Recall that $\gamma_n$ is the sequence of palindromic periodic orbits that we introduce
in Section~\ref{sec: palin}, with collision points
\begin{align*}
& y_{n}(0^-) \mapsto x_n(0) \mapsto y_n(0) \mapsto
x_{n}(1) \mapsto y_{n}(1)  \mapsto \dots
\mapsto x_{n}(n) \mapsto y_{n}(n)  \mapsto x_{n} (n+1).
\end{align*}

Note that we have
$y_n(0)=(t_n(0), \psi_n(0))$ and $y_{n}(0^-)=(t_n(0), -\psi_{n}(0))$.
Then \newline 
$G(y_n(0^-))=G(y_n(0))$,
due to the special formula of $G$, i.e., cosine function for the angle variable.
Moreover, the time-reversibility \eqref{time-rev s} implies that
{\allowdisplaybreaks
\beq\label{rev G}
\begin{split}
G(x_n(n+2-k))=G(x_n(k)), \ \ & \ \ k=1, \dots, n+1, \\
G(y_n(n+1-k))=G(y_n(k)), \ \ & \ \ k=1, \dots, n.
\end{split}
\eeq
}For a sufficiently large integer $\ell>0$ and for any integer $n\ge 2\ell$, 
we define the global sum and local sum of $G$ over $\gamma_n$ as follows.
\begin{itemize}
\item \textbf{Global sum:}
we define the \emph{global sum} of $G$ over the points of $\gamma_n$ (with minus signs)
away from the period two orbit as
{\allowdisplaybreaks
\beqn\label{def S_n global}
\begin{split}
S_n^{global}(\ell):= & -G(y_n(0^-))-G(x_n(0))-G(y_n(0))-  \sum_{k=1}^{\ell}\left[ G(x_n(k)) + G(y_n(k))\right] \\
& -  \sum_{k=1}^{\ell}\left[ G(x_n(n+2-k)) + G(y_n(n+1-k))\right] 
\end{split}
\eeqn
}This global sum contains $(4\ell+3)$ points. By \eqref{rev G} and the fact that $G(y_n(0^-))=G(y_n(0))$,
we get
\beqn
S_n^{global}(\ell)=-G(x_n(0))-2G(y_n(0)) -2\sum_{k=1}^{\ell} \left[ G(x_n(k)) + G(y_n(k)) \right].
\eeqn
\item \textbf{Local sum:}
we define the \emph{local sum} of $G$ over the points of $\gamma_n$ near  the period two orbit as
\beqn
S_n^{local}(\ell):=\sum_{k=\ell+1}^{n-\ell+1}  G(x_{n}(k)) + \sum_{k=\ell+1}^{n-\ell} G(y_{n}(k)).
\eeqn

Note that the local sum contains $(2n-4\ell+1)$ points. 
By convention, the above second sum is set to be zero if $n=2\ell$.
\end{itemize}
By the assumption of Proposition~\ref{lem van der}, i.e., 
$\sum_{\gamma_n} G=0$,
we get $-S_n^{global}(\ell) + S_n^{local}(\ell)=0$. 
For simplicity, we further denote 
\beqn
S_n(\ell) := S_n^{global}(\ell)  = S_n^{local}(\ell),
\eeqn
and thus $S_n(\ell)$ has two expressions:
{\allowdisplaybreaks
\begin{eqnarray}
S_n(\ell)  
&=& -G(x_n(0))-2G(y_n(0)) -2\sum_{k=1}^{\ell} \left[ G(x_n(k)) + G(y_n(k)) \right]  \label{def S_n}
\\
&=& \sum_{k=\ell+1}^{n-\ell+1}  G(x_{n}(k)) + \sum_{k=\ell+1}^{n-\ell} G(y_{n}(k)). \label{def S_n 1}
\end{eqnarray}
}In the next two subsections, we shall introduce the (weighted) Lagrange interpolation method 
and use it to prove the cancellations for the global sum representation of $S_n(\ell)$.
Then in Section~\ref{sec: proof lem van der}, we shall use such cancellations  to extract 
information about the derivatives $\bn_1^{(d)}(0)$ and $\bn_2^{(d)}(0)$ by applying 
the Taylor expansion for the local sum representation of $S_n(\ell)$.

\subsubsection{Lagrange Polynomial Interpolation and Weighted Interpolation}

In this subsection, we first recall the well known Lagrange polynomial interpolation for
functions of one variable (see e.g. \cite{SuMa03}, \S 6.2).
More precisely, for any integer $m\ge 2$ and real numbers $0\le u_m< \dots <u_1 \le 1$,  
the fundamental Lagrange polynomials over the data set $\{u_1, \dots, u_m\}$ are defined by
\beq\label{def Lagrange 0}
p_j(u)=p_j(u; u_1, \dots, u_m):=  \prod_{\substack{1\le i\le m\\ i\ne j}} \frac{u-u_i}{u_j-u_i}
\eeq
for $u\in \IR$ and $j=1, \dots, m$.
Note that these polynomials stasify
the Lagrange basis property, i.e., 
$p_j(u_j)=1$ and $p_j(u_i)=0$ if $i\ne j$.
The Lagrange polynomial interpolation provides an approximation 
of a $C^m$ smooth function in the space of degree $(m-1)$ polynomials,
with a higher order error.

\begin{lemma}\label{Lagrange interpolation 0}
For any function $g\in C^{m}[0, 1]$ and for any $u\in [0, 1]$,
there exists $\overline{u}$ in the smallest interval that contains $u_1, \dots, u_m$ and $u$ such that
\beqn\label{Lagrange02}
g(u)=\sum_{j=1}^m g(u_j) p_j(u)  +  \frac{g^{(m)}(\overline{u})}{m! } \prod_{j=1}^m (u-u_j).
\eeqn
\end{lemma}

We also need the weighted Lagrange interpolation, which can be regarded as 
a variant of the standard Lagrange polynomial interpolation (see e.g.~\cite{Car10}). More precisely, 
let $w\in C^m[0,1]$ be a positive function, and define the fundamental weighted Lagrange interpolants
over the data set $\{u_1, \dots, u_m\}$   by
\beq\label{def Lagrange}
q_j(u)=q_j(u; u_1, \dots, u_m; w):=\frac{w(u_j)}{w(u)} p_j(u; u_1, \dots, u_m),
\eeq
for $u\in \IR$ and $j=1, \dots, m$, where $p_j(u; u_1, \dots, u_m)$
is the standard  fundamental Lagrange polynomials given by \eqref{def Lagrange}.
Note that we still have the basis property for $\{q_j(u)\}_{1\le j\le m}$.
In particular case when $w\equiv 1$, 
we have that  $q_j(u)=p_j(u)$.

Appying Lemma~\ref{Lagrange interpolation 0}
for the function $w(u)g(u)$ and then dividing $w(u)$ on both sides, we obtain 
the weighted Lagrange interpolation as follows.

\begin{lemma}\label{Lagrange interpolation}
For any function $g\in C^{m}[0, 1]$ and for any $u\in [0, 1]$,
there exists $\overline{u}$ in the smallest interval that contains $u_1, \dots, u_m$ and $u$ such that
\beqn\label{Lagrange01}
g(u)=\sum_{j=1}^m g(u_j) q_j(u)  +  \frac{(wg)^{(m)}(\overline{u})}{m!w(u)} \prod_{j=1}^m (u-u_j).
\eeqn
\end{lemma}

\begin{remark} 
The Lagrange interpolation, either polynomial or weighted, heavily depends on the distributions of data points 
$u_1, \dots, u_m$. For our purpose, we shall consider geometric data points, i.e., $u_j=\lambda^{-j}u_0$
for some $u_0\in (0, 1]$. Putting $u=u_0$ in \eqref{Lagrange01}, we notice that $q_j(u_0)$, $j=1, \dots, m$,
are in fact  constants (for our special choices of $w_m(u)$ in \eqref{defnw}), 
and hence we shall obtain a linear combination of the function values 
$g(u_0), g(u_1), \dots, g(u_m)$ up to a higher order error. 
\end{remark}

\subsubsection{Cancellations by Lagrange Interpolations}

We now apply the weighted Lagrange interpolation to show the following cancellations
for a linear combination of $\{S_{2\ell+j}(\ell)\}_{0\le j\le m}$.
The key observation is from Lemma~\ref{lem: finer est'},
which suggests that for each $1\le k\le \ell$,
the points $\{x_{2\ell+j}(k)\}_{0\le j\le m}$ (resp. the points $\{y_{2\ell+j}(k)\}_{0\le j\le m}$
for each $k=0^\pm, 0 ,1, \dots$)
are asymptotically lying on a line. In the following applications, we choose the weighted function to be 
\beq\label{defnw}
w_m(u)=1 \text{ if $m$ is odd}; \,\,\,\,\,
w_m(u)= u \text{ if $m$ is even}.
\eeq

\begin{lemma}\label{lem: cancel}
For any integer $m\ge 2$ and sufficiently large $\ell\ge 1$,
we have 
\beq\label{cancellation}
\sum_{j=0}^{m}  A_{m, j}w_m( \lambda^{-j}) S_{2\ell+j}(\ell)
= \cO\left( \lambda^{-m\ell}\right),
\eeq
where  
\beq\label{def Adj}
A_{m, 0}=-1, \ \text{and} \ 
A_{m, j} = \prod_{\substack{1\le i\le m\\ i\ne j}} \frac{\lambda^{i}-1}{\lambda^{i-j}-1},
\ \ \text{for} \  j=1, \dots, m.
\eeq
\end{lemma}

\begin{proof}
For any $\ell\ge 1$ and $1\le k\le \ell$,
we set
$
u_j=u_j(k,\ell):=\lambda^{k-2\ell -j}
$
for $j=0, 1, \dots, m$.
By Lemma~\ref{lem: finer est'}, we have
\beqn
x_{2\ell+j}(k)-x_\infty(k) = u_j
\left[ v_\infty(2k) + o(1) \right].
\eeqn
This expression means that the points $\{x_{2\ell+j}(k)\}_{0\le j\le m}$
asymptotically lie on the straight line which goes through 
$x_\infty(k)$ and has direction vector $v_\infty(2k)$. 
Recall that $M$ is the phase space defined in \eqref{phase space}.
We then define a piecewise linear curve $\sigma_0: [0, u_0]\to M$ by setting
\beqn
\sigma_0(0)=x_\infty(k),  \ \
 \text{and} \  \ \sigma_0(u_j)=x_{2\ell+j}(k), \ \text{for} \ j=0, 1, \dots, m,
\eeqn
and connect two consecutive points by line segments. Then
\beqn
|\sigma_0'(u)|\le 2\left\|v_\infty(2k)\right\|, \ \ \text{and} \ \ |\sigma_0^{(i)}(u)|=0 \ \text{for} \ i\ge 2,
\eeqn
for all $u\in [0, u_0]$ except at possibly corner points $u_j$ for $j=1, \dots, m$.
Moreover, for sufficiently large $\ell\ge 1$,
the angle at each corner point can be made very obtuse and close to $180$ degree.

For any $\eps>0$, we can obtain a $C^{\infty}$ smooth curve $\sigma_\eps: [0, u_0]\to M$
by smoothening the curve $\sigma_0$ near these corner points, such that
\begin{itemize}[leftmargin=0.75cm]
\item[(1)] $\sigma_\eps$ uniformly converges to $\sigma_0$ in the $C^0$ topology. In particular,
$\sigma_\eps(0)=x_\infty(k)$, $\sigma_\eps(u_0)=x_{2\ell}(k)$,
and
$\sigma_\eps(u_j)\to x_{2\ell+j}(k)$ as $\eps\to 0$ for $1\le j\le m$.
\item[(2)] $\sigma_\eps$ is $C^\infty$ flat at the point $u=0$, i.e., 
$\sigma_\eps^{(n)}(0)$ vanishes for any $n\ge 1$.
\item[(3)] there is a constant $D_{m}>0$, which only depends on $m$, such that
the derivatives of $\sigma_\eps$ are uniformly bounded by $D_{m}$ up to order $m+1$.
\end{itemize}

Let $\{q_j(u)\}_{1\le j\le m}$ be the fundamental weighted Lagrange interpolants 
on the data set $\{u_1, \dots, u_{m}\}$ given by \eqref{def Lagrange}, 
where the polynomials $\{p_j(u)\}_{1\le j\le m}$
are given by \eqref{def Lagrange 0} and 
the weight function $w_m$ is defined as in (\ref{defnw}). 
It is straightforward to verify that $p_j(u_0)=A_{m, j}$, which are given by \eqref{def Adj}.
Also, 
$\dfrac{w_m(u_j)}{w_m(u_0)}=w_m(\lambda^{-j})$ and thus
$q_j(u_0)=A_{m, j}w_m(\lambda^{-j})  $.

We now consider the smooth function
$g_\eps: [0, u_0]\to \IR$ given by
$
g_\eps(u)=G\circ \sigma_\eps(u).
$
Applying  Lemma~\ref{Lagrange interpolation} to this function
with evaluation at $u=u_0$, there exists $ \overline{u}\in [0, u_0]$
such that
{\allowdisplaybreaks
\begin{eqnarray}
g_\eps(u_0)
&=&\sum_{j=1}^{m} A_{m, j}w_m(\lambda^{-j})   g_\eps(u_j)
+ \frac{(w_m g_\eps)^{(m)}(\overline{u})}{m !w_m(u_0)}  \prod_{j=1}^{m} (u_0-u_j) \nonumber \\
&=&\sum_{j=1}^{m} A_{m, j} w_m(\lambda^{-j})  g_\eps(u_j)
+\cO\left(\lambda^{m(k-2\ell)}\right). \label{error g eps}
\end{eqnarray}
}The above error estimate for the last term is due to the following facts:
\begin{itemize}
\item $|u_0-u_j|\le u_0= \lambda^{k-2\ell}$ and thus
$ \left|\prod\limits_{j=1}^{m} (u_0-u_j) \right| \le m\lambda^{k-2\ell}$;
\item if $m$ is odd, then $w_m(u)\equiv 1$ and thus
\beqn
\left|\frac{(w_m g_\eps)^{(m)}(\overline{u})}{m !w_m(u_0)} \right|
\le \|(g_\eps)^{(m)} \|_\infty \le 
\|G\|_{C^{m}}\cdot m^m \max\{1, \|\sigma_\eps\|_{C^{m}}\}^m 
\le m^m D_m^m  \|G\|_{C^{m}};
\eeqn
if $m$ is even, then $w_m(u)\equiv u$. We notice that $\overline{u}\le u_0$, and  
\newline 
$g_\eps^{(m-1)}(0)=(G\circ \sigma_\eps)^{(m-1)}(0)=0$. Since $\sigma_\eps$ is $C^\infty$ flat at $u=0$, then we have
{\allowdisplaybreaks 
\begin{eqnarray*}
\left|\frac{(w_m g_\eps)^{(m)}(\overline{u})}{m !w_m(u_0)} \right| &=&
\left|\frac{ \overline{u} g_\eps^{(m)}(\overline{u}) + m  g_\eps^{(m-1)}(\overline{u})}{m ! u_0} \right| \le \left|g_\eps^{(m)}(\overline{u})  \right| +  
 \left|\frac{ g_\eps^{(m-1)}(\overline{u}) }{\overline{u}} \right| \\
 &\le & 2 \|(g_\eps)^{(m)} \|_\infty\le 2m^m D_m^m  \|G\|_{C^{m}}. 
\end{eqnarray*}
}
\end{itemize}
Letting $\eps\to 0$ in \eqref{error g eps}, we get
\beqn 
-G(x_{2\ell+j}(k)) + \sum_{j=1}^{m} A_{m, j} w_m(\lambda^{-j})   G(x_{2\ell+j}(k))
=\cO\left(\lambda^{m(k-2\ell)}\right).
\eeqn
In a similar fashion, we also have for any $k=0^\pm, 0, 1, \dots, \ell$, 
\beqn
- G(y_{2\ell+j}(k)) +
\sum_{j=1}^{m} A_{m, j}w_m(\lambda^{-j})   G(y_{2\ell+j}(k))=\cO\left(\lambda^{m(k-2\ell)}\right).
\eeqn
Therefore, by the definition of $S_{2\ell+j}(\ell)$ given in \eqref{def S_n},
we obtain \eqref{cancellation} by noticing that
the error term is
$2\sum_{k=0}^\ell \cO\left(\lambda^{m(k-2\ell)}\right)=\cO\left( \lambda^{-m\ell}\right)$.
\end{proof}

We shall need the following property for the coefficients $A_{m, j}$ given in \eqref{def Adj}.

\begin{lemma}\label{lem coefficients}
The coefficients $\{A_{m, j}\}_{0\le j\le m}$ in \eqref{def Adj}
satisfy the following properties:
\begin{enumerate}
\item[(1)] For any $0\le k\le m-1$,
\beq\label{property Amj}
\sum_{j=0}^{m} A_{m, j} \lambda^{-kj}=0.
\eeq
Furthermore,
\beq\label{property Amj 1}
\sum_{j=0}^{m} A_{m, j} \lambda^{-mj}\ne 0.
\eeq
\item[(2)] For any $ 0\le k<m$,  
\beq\label{property Amj 2}
\sum_{j=0}^{m}\  j A_{m, j} \lambda^{-kj}\neq 0. 
\eeq
\end{enumerate}
\end{lemma}

\begin{proof}
Let $\{e_j(u)\}_{1\le j\le m}$
be the fundamental Lagrange polynomials over the data set $\{\lambda^{-1}, \dots, \lambda^{-m}\}$,
and  notice that $A_{m, j}=e_j(1)$ for $j=1, \dots, m$. Recall that $A_{m, 0}=-1$. 
Applying Lemma~\ref{Lagrange interpolation 0} for functions $g(u)=u^k$ for any $k\ge 0$, 
and evaluating at $u=1$, 
there is $\overline{u}\in [\lambda^{-1}, 1]$ such that
{\allowdisplaybreaks 
\begin{eqnarray*}
\sum_{j=0}^{m} A_{m, j} \lambda^{-kj} 
&=& - \frac{(u^k)^{(m)}|_{u=\overline{u}}}{m! } \prod_{j=1}^m (1-\lambda^{-j}) \\
&=&
\begin{cases}
0, \ & \ \mathrm{if} \ 1\le k\le m-1, \\
\prod\limits_{j=1}^m (1-\lambda^{-j})  \ne 0, \ & \ \mathrm{if}\ k=m. 
\end{cases}
\end{eqnarray*}
}Therefore, \eqref{property Amj} and \eqref{property Amj 1} hold.

Similarly, applying Lemma~\ref{Lagrange interpolation 0} for the set of functions $g(u)=u^k\log u$ for any $k\ge 0$, and evaluating at $u=1$, 
there is $\overline{u}\in  [\lambda^{-1}, 1]$ such that
{\allowdisplaybreaks 
\begin{eqnarray*}
\sum_{j=0}^{m} A_{m, j}\lambda^{-kj}\log \lambda^{-j}
&= &
- \frac{(u^k\log u)^{(m)}|_{u=\overline{u}}}{m! } \prod_{j=1}^m (1-\lambda^{-j}) \\
&=& 
-\frac{(-1)^{m-k-1} k! (m-k-1)! \overline{u}^{k-m}}{m!}\prod_{j=1}^m (1-\lambda^{-j})\ne 0.
\end{eqnarray*}
}Dividing both sides by $-\log \lambda$, we get \eqref{property Amj 2}.
\end{proof}

\subsection{Proof of Proposition~\ref{lem van der}}\label{sec: proof lem van der}

Recall that in the assumptions of Proposition~\ref{lem van der},
we are dealing with a normalized family
$\left\{\Omega_\mu\right\}_{|\mu|\le 1}$ 
of Bunimovich squash-type stadia in $\cM_{ss}^\omega(\chi)$,
which satisfies that 
$\sum_{\gamma^*} G=\sum_{\gamma_n} G=0$ for any $n\ge 1$.
We have the following lemma.
 
\begin{lemma}\label{lem: n'}
$\bn_1(0)=\bn_2(0)=\bn_1'(0)=\bn_2'(0)=0.$
\end{lemma}

\begin{proof}
We already get that $\bn_1(0)=\bn_1'(0)=0$ by Lemma~\ref{lem: arclength}.
The assumption that $\sum_{\gamma^*} G=0$ yields that $\bn_1(0)+\bn_2(0)=0$,
and thus $\bn_2(0)=0$. It also implies that the length function
$\Delta(\mu; \overline{12})$ given by \eqref{def ell12} is constant,
which means that $B(\mu)=B$
for any $\mu\in [-1, 1]$. Then applying the same arguments in the proof
of Lemma~\ref{lem: arclength}, we get $\bn_2'( 0)=0$.
\end{proof}

A by-product of the proof of Lemma~\ref{lem: n'} is that 
$B(\mu)=B$ for any $\mu\in [-1, 1]$. Combining with Assumption (IV) with homothety ratio $\chi$,
we have the following relation between the derivatives of $\bn_1$ and $\bn_2$ at zero. 

\begin{lemma}
\label{lem: n1n2}
$\bn_2^{(d)}(0)=\chi^{1-d} \bn_1^{(d)}(0)$ for any $d\ge 0$.
\end{lemma} 

\begin{proof}
In Subsection~\ref{sec: arclength p}, 
we introduce the parametrizations $s\mapsto \Phi_1(\mu, s)$ on $\Gamma_1(\mu)$
and $t\mapsto \Phi_2(\mu, t)$ on $\Gamma_2(\mu)$.
Since now $B(\mu)=B$ for any $\mu\in [-1, 1]$, by \eqref{eq: homo}, we have 
\beqn
\Phi_2(\mu, t)=\cR\left[\chi \Phi_1\left(\mu,  t/\chi \right)\right]
\ \ \text{for}\ t \ \text{close to} \ 0,
\eeqn
where $\cR$ is the counter-clockwise rotation with center at the middle point of $\overline{AB}$ by 180 degree. 
By the definition of the deformation function given in \eqref{def deformation function},
for any $t$ close to $0$, we have 
\beqn
\bn_2(\mu, t)=\langle \p_\mu\Phi_2(\mu, t), \ N_2(\mu, t) \rangle
=\langle \cR\left[\chi \p_\mu \Phi_1\left(\mu,  t/\chi \right)\right],
\cR N_1(\mu, t/\chi)  \rangle
=\chi \bn_1\left(\mu, t/\chi \right),
\eeqn
and hence 
$\bn_2^{(d)}(0)=\chi^{1-d} \bn_1^{(d)}(0)$
by taking the $d$-th order derivative at $t=0$. 
\end{proof}

\bigskip

We are now ready to prove Proposition~\ref{lem van der} by induction.
Note that
the base of the induction has already been proven in Lemma~\ref{lem: n'}, that is,
\beqn
\bn_1(0)=\bn_2(0)=\bn_1'(0)=\bn_2'(0)=0.
\eeqn
 
Suppose now $d\ge 2$ is an integer such that
\beqn
\bn_1^{(k)}(0)=0\ \text{ and } \ \bn_2^{(k)}(0)=0, \
\ \text{for any} \ k=0, 1, \dots, d-1.
\eeqn
We shall use the assumption that $\sum_{\gamma_n} G=0$ 
to show that $\bn_1^{(d)}(0)=0$ and $\bn_2^{(d)}(0)=0$. 

\medskip

Recall that the maximal period two orbit is denoted as $\gamma^*=\overline{xy}$.
Also, $x=(0, 0)$ in the $(s, \varphi)$-coordinate and $y=(0, 0)$ in the $(t, \psi)$-coordinate.
For $z=(s, \varphi)$ near $x=(0, 0)$, we write the Taylor expansion of $\bn_1(s)$
up to the $d$-th order, as well as that of $\cos\varphi$ up to the 1st order as
\beq\label{Taylor exp 1}
\bn_1(s) = \frac{1}{d!}\, \bn_1^{(d)}(0) \, s^d + \cO\left( s^{d+1} \right), \ \
\cos\varphi=1 +\cO(\varphi^2).
\eeq
Similarly, for $z=(t, \psi)$ near $y=(0, 0)$,
\beq\label{Taylor exp 2}
\bn_2(t) = \frac{1}{d!} \, \bn_2^{(d)}(0) \, t^d +  \cO\left( t^{d+1} \right), \ \
\cos\psi=1 + \cO(\psi^2).
\eeq
Given an fixed $j\ge 0$,
by Lemma~\ref{lem: finer est}, we have that for any large $\ell$ and any $i=1, \dots, j+1$,
{\allowdisplaybreaks
\begin{align*}
x_{2\ell+j}(\ell+i)
&=\lambda^{-\ell-i} (C_s, C_\varphi) + \lambda^{-\ell+i-j} (C_{s, \ell+i}, C_{\varphi, \ell+i})   \\
&+\cO\left(\max\left\{\lambda^{-1.5(\ell+i)}, \lambda^{1.25(-\ell+i-j)}\right\} \right) \\
&=\lambda^{-\ell} \left(C_s\left(\lambda^{-i} + \lambda^{i-j-2} \right),
C_\varphi(\lambda^{-i} - \lambda^{i-j-2}) \right) +\cO(\lambda^{-1.25\ell}).
\end{align*}
}Note that the first term is of leading order $\lambda^{-\ell}$ (as $\ell\to \infty$).
Then by \eqref{Taylor exp 1},
{\allowdisplaybreaks
\beq\label{def Gx}
\begin{split}
G(x_{2\ell+j}(\ell+i))
&=\bn_1(s_{2\ell+j}(\ell+i))\cos\left(\varphi_{2\ell+j}(\ell+i)\right) \\
&=\frac{1}{d!} \bn_1^{(d)}(0)  \lambda^{-d\ell}  C_s^d \left(\lambda^{-i} + \lambda^{i-j-2} \right)^d
+\cO\left(\lambda^{-\ell(d+0.25)} \right)\end{split}
\eeq
}Similarly, for any $j\geq 1$, $i=1, \dots, j$,
{\allowdisplaybreaks
\beq\label{def Gy}
G(y_{2\ell+j}(\ell+i))=
\frac{1}{d!} \bn_2^{(d)}(0) \lambda^{-d\ell} C_t^d \left(\lambda^{-i} + \lambda^{i-j-1}  \right)^d
+\cO\left(\lambda^{-\ell(d+0.25)} \right).
\eeq
}Applying  Lemma~\ref{lem: cancel}
for $m=d+1$ and sufficiently large $\ell\ge 1$,
we obtain
\beq\label{compare d even}
\sum_{j=0}^{d+1} A_{d+1, j} w_{d+1}(\lambda^{-j}) S_{2\ell+j}(\ell) =\cO(\lambda^{-\ell(d+1)}),
\eeq
where $A_{d+1, j}$ are given by \eqref{def Adj}
and $w_{d+1}(\cdot )$ is given by \eqref{defnw}.
Under the assumption that $\sum_{\gamma_n} G=0$, we have an alternative
formula for $S_{2\ell+j}(\ell)$ given by \eqref{def S_n 1}, that is,
\beqn
S_{2\ell+j}(\ell)
= \sum_{i=1}^{j+1} G(x_{2\ell+j}(\ell+i))
+ \sum_{i=1}^{j} G(y_{2\ell+j}(\ell+i)).
\eeqn
By \eqref{def Gx} and \eqref{def Gy}, we rewrite the LHS of \eqref{compare d even} as
\beqn
\sum_{j=0}^{d+1} A_{d+1, j} S_{2\ell+j}(\ell)
=\frac{\lambda^{-d\ell}}{d!}  \left[ \cA_1 C_s^d \bn_1^{(d)}(0) + 
\cA_2 C_t^d \bn_2^{(d)}(0) \right] + \cO\left(\lambda^{-\ell(d+0.25)} \right),
\eeqn
where the coefficients $\cA_\sigma=\cA_\sigma(d, \lambda)$, $\sigma=1, 2$, 
are given by
\beq\label{def A sigma}
\cA_\sigma=\sum_{j=0}^{d+1} A_{d+1, j} w_{d+1}(\lambda^{-j}) \sum_{i=1}^{j+2-\sigma} \left(\lambda^{-i} + \lambda^{i-j-3+\sigma} \right)^d.
\eeq
Therefore, multiplying both sides of \eqref{compare d even} by $\lambda^{d\ell}$ 
and letting $\ell\to \infty$, we get
\beq\label{compare d even 1}
 \cA_1 C_s^d \bn_1^{(d)}(0) + \cA_2 C_t^d \bn_2^{(d)}(0) = 0.
\eeq

\medskip

The relation between $\cA_1$ and $\cA_2$ are given by the following lemma.

\begin{lemma}
\label{lem: A1A2}
The coefficients $\cA_1$ and $\cA_2$ in \eqref{compare d even 1} satisfy the following relation:
\beqn
\cA_2=\lambda^{d/2}\cA_1\ne 0, \ \text{when} \ d \ \text{is even}; 
\ \ \
\cA_2=\lambda^{d}\cA_1\ne 0, \ \text{when} \ d \ \text{is odd}.
\eeqn
\end{lemma}

\begin{proof}
When $d$ is even, we note that $w_{d+1}\equiv 1$. By
\eqref{property Amj}, \eqref{property Amj 2} and 
\eqref{def A sigma}, we get
{\allowdisplaybreaks
\begin{align*}
\cA_\sigma &= \sum_{j=0}^{d+1} A_{d+1, j}  \sum_{i=1}^{j+2-\sigma}
\sum_{\nu=0}^d {d \choose \nu} \lambda^{-i(d-\nu)+\nu(i-j-3+\sigma)} \\
&=\sum_{\nu=0}^d {d \choose \nu} \sum_{j=0}^{d+1} A_{d+1, j}
\sum_{i=1}^{j+2-\sigma} \lambda^{-i(d-2\nu)-\nu(j+3-\sigma)} \\
&=\sum_{\nu=0}^d {d \choose \nu} \sum_{j=0}^{d+1} A_{d+1, j} \cdot
\begin{cases}
\dfrac{\lambda^{-(d-2\nu)-\nu(j+3-\sigma)} - \lambda^{-(d-\nu)(j+3-\sigma)}}{1-\lambda^{-(d-2\nu)}},
& \ \mathrm{if} \ \nu\ne d/2, \\ 
\\
(j+2-\sigma) \lambda^{-\nu(j+3-\sigma)},  
& \ \mathrm{if} \ \nu= d/2, 
\end{cases}
\\
&=\sum_{\nu=0}^d {d \choose \nu} \cdot
\begin{cases}
0, & \ \mathrm{if} \ \nu\ne d/2, \\ 
\\
\sum_{j=0}^{d+1} j A_{d+1, j}  \lambda^{-\nu(j+3-\sigma)},
& \ \mathrm{if} \ \nu= d/2, 
\end{cases} 
\\
&=\lambda^{-d(3-\sigma)/2} {d \choose d/2} \sum_{j=0}^{d+1} j A_{d+1, j} \lambda^{-dj/2} \ne 0,
\end{align*}
where the last term came from the only non-trivial term with $\nu=d/2$.
}Therefore, $\cA_2=\lambda^{d/2}\cA_1\ne 0$.

When $d$ is odd, we note that $w_{d+1}(u)=u$ and thus
$w_{d+1}(\lambda^{-j})=\lambda^{-j}$.
By \eqref{property Amj}, \eqref{property Amj 1} and 
\eqref{def A sigma}, 
we have 
{\allowdisplaybreaks
\begin{eqnarray*}
\cA_\sigma &=& \sum_{j=0}^{d+1} A_{d+1, j} \lambda^{-j} \sum_{i=1}^{j+2-\sigma}
\sum_{\nu=0}^d {d \choose \nu} \lambda^{-i(d-\nu)+\nu(i-j-3+\sigma)} \\
&=&\sum_{\nu=0}^d {d \choose \nu} \sum_{j=0}^{d+1} A_{d+1, j}  \lambda^{-j}
\sum_{i=1}^{j+2-\sigma} \lambda^{-i(d-2\nu)-\nu(j+3-\sigma)} \\
&=&\sum_{\nu=0}^d {d \choose \nu} \sum_{j=0}^{d+1} A_{d+1, j} \lambda^{-j}  \cdot
\dfrac{\lambda^{-(d-2\nu)-\nu(j+3-\sigma)} - \lambda^{-(d-\nu)(j+3-\sigma)}}{1-\lambda^{-(d-2\nu)}} \\
&=&\sum_{\nu=0}^d {d \choose \nu} \dfrac{\cA_\sigma^1(\nu) - \cA_\sigma^2(\nu)} {1-\lambda^{-(d-2\nu)}},
\end{eqnarray*}
}where
{\allowdisplaybreaks
\begin{eqnarray*}
\cA_\sigma^1(\nu)
&=& \lambda^{-(d-2\nu)-\nu(3-\sigma)} \sum_{j=0}^{d+1} A_{d+1, j}  \lambda^{-(\nu+1) j}, \\
&=& \lambda^{-(d-2\nu)-\nu(3-\sigma)} \cdot 
\begin{cases}
0, \ & \ \mathrm{if} \ 0\le \nu < d, \\
\\
\sum\limits_{j=0}^{d+1} A_{d+1, j}  \lambda^{-(d+1) j}, \ & \ \mathrm{if} \  \nu = d,
\end{cases} 
\end{eqnarray*}
}and 
{\allowdisplaybreaks
\begin{eqnarray*}
\cA_\sigma^2(\nu) 
&=& \lambda^{-(d-\nu)(3-\sigma)} \sum_{j=0}^{d+1} A_{d+1, j}  \lambda^{-(d+1-\nu) j} \\
&=& \lambda^{-(d-\nu)(3-\sigma)} \cdot 
\begin{cases}
0, \ & \ \mathrm{if} \ 1\le \nu \le d, \\
\\
\sum\limits_{j=0}^{d+1} A_{d+1, j}  \lambda^{-(d+1) j}, \ & \ \mathrm{if} \  \nu = 0.
\end{cases} 
\end{eqnarray*}
}Therefore, 
{\allowdisplaybreaks
\begin{eqnarray*}
\cA_\sigma =
 \dfrac{\cA^1_\sigma(d)}{1-\lambda^d} - \dfrac{\cA^2_\sigma(0)}{1-\lambda^{-d}} 
&=& \left( \dfrac{ \lambda^{d(\sigma-2)}}{ 1-\lambda^d} - \dfrac{ \lambda^{d(\sigma-3)}}{1-\lambda^{-d}} \right) \sum\limits_{j=0}^{d+1} A_{d+1, j}  \lambda^{-(d+1) j} \\
&=& \dfrac{2\lambda^{d(\sigma-2)} }{1-\lambda^d} \sum\limits_{j=0}^{d+1} A_{d+1, j}  \lambda^{-(d+1) j},
\end{eqnarray*}
}which implies that 
$\cA_2=\lambda^{d}\cA_1\neq 0$.
\end{proof}

\medskip

By \eqref{compare d even 1} and Lemma~\ref{lem: A1A2}, 
we get that for any $d\ge 2$,  
\beq\label{compare d even 3}
 C_s^d\, \bn_1^{(d)}(0) + \lambda^{d\beta(d)} \,C_t^d \, \bn_2^{(d)}(0) = 0,
\eeq
where we set $\beta(d)=\frac12$ when $d$ is even and 
$\beta(d)=1$ when $d$ is odd. 
By Lemma~\ref{lem: n1n2} which asserts that
$\bn_2^{(d)}(0)=\chi^{1-d} \bn_1^{(d)}(0)$, 
we further get 
\beq\label{compare d even 2}
C_s^d \left[ 1+ \lambda^{d\beta(d)} \left(C_t/C_s\right)^d \chi^{1-d} \right]\cdot \bn_1^{(d)}(0) =0.
\eeq
Recall that $C_s$ and $C_t$ are computed in \eqref{C relation} of
Lemma \ref{lem: finer est}, then we have $C_s\ne 0$ and 
$
C_t/C_s=-\frac{2a_1}{1+\lambda}>0
$
since $a_1<-1$ and $\lambda>1$.
Hence 
\beqn
C_s^d \left[ 1+ \lambda^{d\beta(d)} \left(C_t/C_s\right)^d \chi^{1-d} \right]\ne 0,
\eeqn
which implies that $\bn_1^{(d)}(0) =0$ and thus $\bn_2^{(d)}(0)=\chi^{1-d} \bn_1^{(d)}(0)=0$ as well.

\bigskip

Now that we have shown that no matter $d$ is even or odd, 
we always have that 
$\bn_1^{(d)}(0)=\bn_2^{(d)}(0)=0$, and so
the proof of Proposition~\ref{lem van der} is complete.
As we are in the setting of the analytic class $\cM^\omega_{ss}(\chi)$,
Proposition~\ref{lem van der}
implies that $\bn_1$ and $\bn_2$ both vanish,
which proves Proposition~\ref{prop: main n vanish} and thus
Squash Rigidity Theorem.

 \bigskip

\section{Proof of the Stadium Rigidity Theorem}\label{sec: proof of thm2}

Let $\{\Omega_\mu\}_{|\mu|\le 1}$ be a $C^1$ family of Bunimovich stadia in $\cM^\omega_{s, b}$,
such that the flat boundaries $\Gamma_3$ and $\Gamma_4$ are two opposite sides of a fixed rectangle. 
We also denote the arcs by $\Gamma_i=\Gamma_i(\mu)$, $i=1, 2$.
As the flat boundaries are a priori fixed, the four gluing points do not depend on the parameter $\mu$, 
and thus we could denote the gluing points by
\beqn
P_{ij}:=\Gamma_i(\mu)\cap \Gamma_j, \ \ i=1, 2, \ j=3, 4.
\eeqn
The absolute curvature of $\Gamma_i=\Gamma_i(\mu)$ at the gluing point $P_{ij}$ 
is denoted by $\cK_{ij}=\cK_{ij}(\mu)$.
Since $\Omega_\mu\in \cM^\omega_{s, b}$, in particular, 
$\Omega_\mu$ satisfies Assumption (II), i.e., $\p \Omega_\mu$ is not $C^2$ smooth 
at these gluing points, then we must have $\cK_{ij}>0$.

Note that $\overline{P_{14} P_{13} P_{23} P_{24}}$ is a rectangle, and we denote
\beq\label{quotient rect}
Q:=\frac{|\overline{P_{13} P_{23}}|}{|\overline{P_{14} P_{13}}|}.
\eeq

The Stadium Rigidity Theorem 
concerns the dynamical spectral rigidity in the class $\cM^\omega_{s, b}$.  
With careful arrangement, we choose
a parametrization $\Phi: [-1, 1]\times J\to \IR^2$ for
the family of stadia $\left\{\Omega_\mu\right\}_{|\mu|\le 1}$
along the counter-clockwise direction, where
\beqn
J=[0, |\p \Omega_0|]:=J_1\cup J_3 \cup J_2 \cup J_4,
\eeqn
such that
\begin{itemize}
\item[(1)] For any $\mu\in [-1, 1]$ and $i=1, 2, 3, 4$, we have $\Phi(\mu, J_i)=\Gamma_i(\mu)$. In particular,
$\Gamma_j(\mu)=\Gamma_j$ for $j=3, 4$. Also, $\Phi(\mu, 0)=P_{14}$ and $\Phi(\mu, |J_1|+|J_3|)=P_{23}$;
\item[(2)] The mapping $\mu\mapsto \Phi(\mu, \cdot)$ is $C^1$ smooth;
\item[(3)]
Set $\Phi_i=\Phi|_{J_i}$ for $i=1, 2$.
For any fixed $\mu\in [-1, 1]$, the map $r\mapsto \Phi_i(\mu, r)$ is analytic. 
\item[(4)]
As the flat boundaries $\Gamma_3$ and $\Gamma_4$ are fixed,
we could also assume that
\beq\label{flat fixed}
\Phi(\mu, r)=\Phi(0, r), \ \text{for any} \ \mu\in [-1, 1] \ \text{and} \ r\in J_3\cup J_4.
\eeq
\end{itemize}

We then define the deformation function
$\bn: [-1, 1]\times J\to \IR$ as in \eqref{def deformation function}.
It is obvious that $\bn\equiv 0$ on $[-1, 1]\times \left(J_3\cup J_4 \right)$.
Moreover, for any $\mu\in [-1, 1]$, the function $r\to \bn(\mu, r)$ is analytic on $J_1\cup J_2$.

For brevity, when the parameter $\mu$ is clear,
we shall just write $\bn(r)$ instead of $\bn(\mu, r)$.
From Section~\ref{sec: functional spectra},
we know that if the family of domains $\left\{\Omega_\mu\right\}_{|\mu|\le 1}$
is dynamically isospectral, then
\beq\label{def G sum zero}
\sum_{z\in \gamma} G(z)= \sum_{(r, \varphi)\in \gamma} \bn(r) \cos\varphi = 0,
\eeq
for any periodic billiard orbit $\gamma$.
Since $\bn\equiv 0$ on $J_3\cup J_4$, the equation \eqref{def G sum zero}
holds for any periodic trajectories of
the induced billiard map on $\Gamma_1\cup \Gamma_2$.

Similar to what we have done in Section~\ref{sec: arclength p}, 
for any fixed $\mu_0\in [-1, 1]$, we may introduce a parametrization
$s\mapsto \Phi_1(\mu, s)$ with $s\in J_1$ on $\Gamma_1(\mu)$
and $t\mapsto \Phi_2(\mu, t)$ with $t\in J_2$ on $\Gamma_2(\mu)$,
such that $\Phi_1(\mu_0, s)$ and $\Phi_2(\mu_0, t)$ are of arclength parametrization
on $\Gamma_1(\mu_0)$ and $\Gamma_2(\mu_0)$ respectively. 
Different from Section~\ref{sec: arclength p}, 
we intend to study the local behavior near the gluing points $P_{14}$ and $P_{23}$,
therefore, 
we further assume that $\Phi_1(\mu, 0)=P_{14}$ and $\Phi_2(\mu, 0)=P_{23}$.

As mentioned earlier in Remark~\ref{rem: bn},
the vanishing property of $\bn$ would not change under those reparametrizations.
In this way, we rewrite $\bn$ on $\Gamma_1$ and $\Gamma_2$ by $\bn_1(s)$
and $\bn_2(t)$ respectively, in which we omit the parameter $\mu$. 

\subsection{Induced Periodic Trajectories}

Unfolding a stadium $\Omega\in \cM^\omega_{s, b}$, we obtain a channel of consecutive cells isometric to $\Omega$.
We denote the $k$-th cell by $\Omega^k$,
whose below curve is  $\Gamma_1^k$ with left endpoint $P_{14}^k$,
and above curve is $\Gamma_2^k$ with right endpoint $P_{23}^k$.
Using this unfolding technique,
we first construct periodic billiard trajectories of induced period two.

\begin{lemma}\label{lem: per 2}
For any sufficiently large $n\ge 1$, there exists a period two palindromic trajectory (see Fig.~\ref{fig:unfolding2})
$$
\ogamma_n=\overline{A_nB_n}=\overline{(\os_n, 0) (\ot_n, 0)}
$$
for the induced billiard map on $\Gamma_1\cup \Gamma_2$,
such that $A_n\in \Gamma_1^0$ and $B_n\in \Gamma_2^n$.
Moreover, as $n\to \infty$,
\beq\label{per 2 est}
\os_n=\frac{Q}{n\cK_{14}} + \cO\left(\frac{1}{n^2}\right), \ \text{and} \ \
\ot_n=\frac{Q}{n\cK_{23}} + \cO\left(\frac{1}{n^2}\right),
\eeq
where $Q$ is the quotient given by \eqref{quotient rect},
$\cK_{14}$ is the absolute curvature of $\Gamma_1$ at $P_{14}$,
and
$\cK_{23}$ is the absolute curvature of $\Gamma_2$ at $P_{23}$.
\end{lemma}

\begin{figure}[h]
\begin{center}
\includegraphics[width=.8\textwidth]{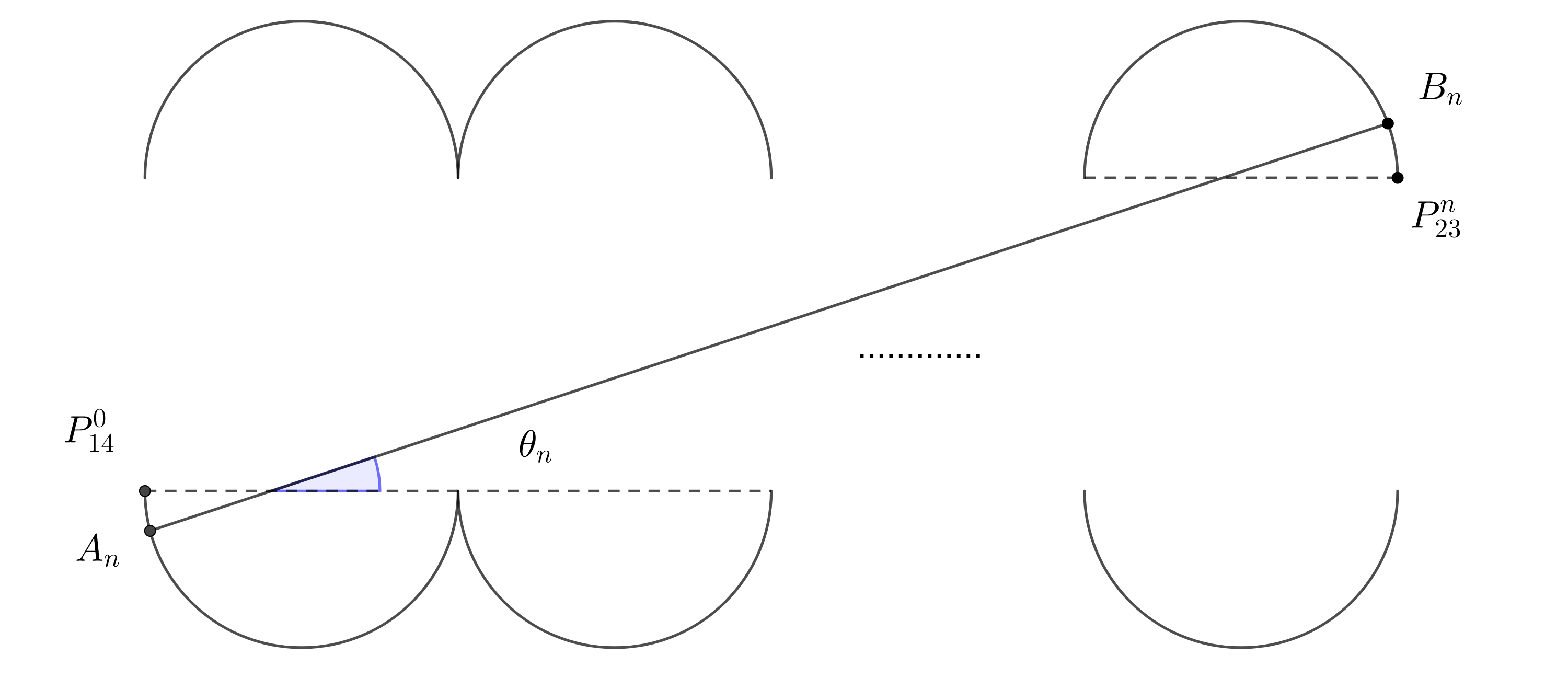}
\end{center}
\caption{The period two orbit $\ogamma_n$ for the induced billiard map}
\label{fig:unfolding2}
\end{figure}

\begin{proof}
By strict concavity of $\Gamma_1^0$ and $\Gamma_2^n$,
there is a unique pair of points $A_n\in \Gamma_1^0$ and $B_n\in \Gamma_2^n$,
which achieve the maximum of
\beqn
\left\{\left\|\overline{A B} \right\|: \ A\in  \Gamma_1^0, \ B\in \Gamma_2^n \right\}.
\eeqn
As a consequence, $\overline{A_n B_n}$ is perpendicular to $\Gamma_1^0$ at $A_n$,
as well as to $\Gamma_2^n$ at $B_n$.
Denote the coordinate of $A_n$ in $\Gamma_1^0$ by $\os_n$, and
the coordinate of $B_n$ in $\Gamma_2^n$ by $\ot_n$, then
$\ogamma_n=\overline{A_nB_n}=\overline{(\os_n, 0) (\ot_n, 0)}$ is a period two orbit
for the induced billiard map on $\Gamma_1\cup \Gamma_2$.

Recall that $Q$ is the quotient given by \eqref{quotient rect}.
Let $\theta_n$ be the angle formed by the horizontal axis and $\overline{A_n B_n}$,
then as $n\to \infty$,
\beq\label{per 2 est 1}
\theta_n= \tan\theta_n + \cO\left(\frac{1}{n^2}\right) = \frac{Q}{n} + \cO\left(\frac{1}{n^2}\right).
\eeq
On the other hand, as $n\to \infty$ and thus $\theta_n\to 0$,
the point $A_n$ is more and more close to $P_{14}^0$, and
the point $B_n$ is more and more close to $P_{23}^n$.
We approximate $\Gamma_1^0$ by the osculating circle at $P_{14}^0$ with absolute curvature $\cK_{14}$,
and approximate $\Gamma_2^n$ by the osculating circle at $P_{23}^n$ with absolute curvature $\cK_{23}$.
Then there exist constants $c_1, c_2\in \IR$
such that as $n\to \infty$,
\beq\label{per 2 est 2}
\theta_n= \cK_{14} \os_n+c_1 \os_n^2
=\cK_{23} \ot_n +c_2 \ot_n^2.
\eeq
Then the estimates in \eqref{per 2 est} directly follows from \eqref{per 2 est 1} and \eqref{per 2 est 2}.
\end{proof}

We also consider periodic billiard trajectories of induced period 4.

\begin{lemma}\label{lem: per 4}
For any fixed $\rho>1$ and
any sufficiently large number $n\ge 1$, there exists a period four palindromic trajectory (see Fig.~\ref{fig:unfolding4})
$$
\ogamma_{n, \rho}=
\ \overline{B_{n, \rho}' A_{n, \rho} B_{n, \rho}'' A_{n, \rho}}=
\ \overline{(\ot_{n, \rho}', 0) (\os_{n, \rho}, \ovarphi_{n, \rho})
(\ot_{n, \rho}'', 0) (\os_{n, \rho}, -\ovarphi_{n, \rho})},
$$
for the induced billiard map on $\Gamma_1\cup \Gamma_2$, such that
$A_{n, \rho}\in \Gamma_1^0$, $B_{n, \rho}'\in \Gamma_2^n$ and $B_{n, \rho}''\in \Gamma_2^{\lfloor n \rho \rfloor}$.
Moreover, as $n\to \infty$,
\beq\label{per 4 est}
\begin{split}
\ot_{n, \rho}'=\frac{Q}{n\cK_{23}} + \cO\left(\frac{1}{n^2}\right), \ \ \  \ \  \ \  \ \ \
& \ \ \ \ot_{n, \rho}''=\frac{Q\rho^{-1}}{n\cK_{23}} + \cO\left(\frac{1}{n^2}\right),   \\
\os_{n, \rho}=\frac{Q(1+\rho^{-1})}{2n\cK_{14}} + \cO\left(\frac{1}{n^2}\right),  \ \ \
& \ \ \ \ovarphi_{n, \rho}=\frac{Q(1-\rho^{-1})}{2n} + \cO\left(\frac{1}{n^2}\right).
\end{split}
\eeq
\end{lemma}

\begin{figure}[h]
\begin{center}
\includegraphics[width=.8\textwidth]{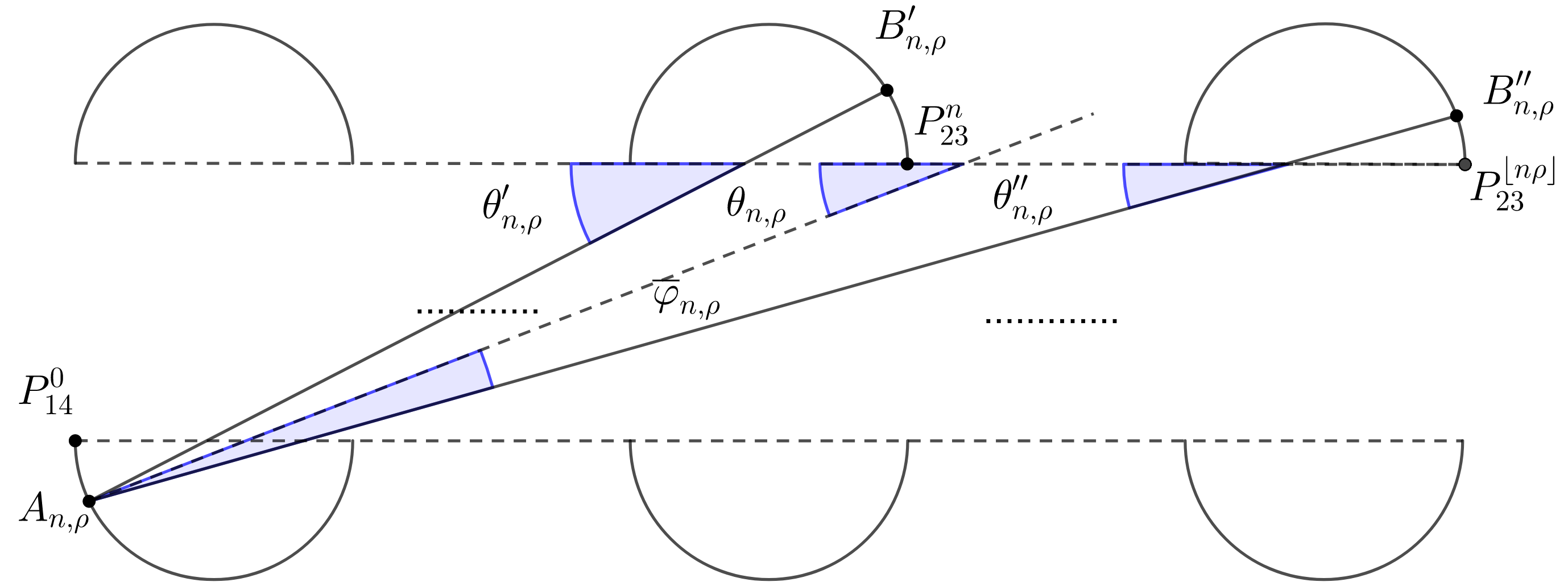}
\end{center}
\caption{The period four palindromic orbit $\ogamma_{n, \rho}$ for the induced billiard map}
\label{fig:unfolding4}
\end{figure}

\begin{proof} Similar to the proof of Lemma~\ref{lem: gamma_n uniq},
there exist $A_{n, \rho}\in \Gamma_1^0$, $B_{n, \rho}'\in \Gamma_2^n$
and $B_{n, \rho}''\in \Gamma_2^{\lfloor n \rho \rfloor}$,
which achieve the maximum of
\beqn
\left\{\left\|\overline{A B'} \right\| + \left\|\overline{A B''} \right\| : \ A\in  \Gamma_1^0, \
B'\in \Gamma_2^n, \ B''\in \Gamma_2^{\lfloor n \rho \rfloor} \right\}.
\eeqn
Thus, $\overline{A_{n, \rho} B_{n, \rho}'}$ is perpendicular to $\Gamma_2^n$ at $B_{n, \rho}'$,
and
$\overline{A_{n, \rho} B_{n, \rho}''}$ is perpendicular to $\Gamma_2^{\lfloor n \rho \rfloor}$ at $B_{n, \rho}''$.
Moreover, the angle $\angle(\overline{A_{n, \rho} B_{n, \rho}'}, \cN)$
is equal to the angle $\angle(\overline{A_{n, \rho} B_{n, \rho}''},, \cN)$,
where $\cN$ is the inner normal direction of $\Gamma_1^0$ at $A_{n, \rho}$.
We denote this angle by $\ovarphi_{n, \rho}$.
We further denote the coordinate of $A_{n, \rho}$ in $\Gamma_1^0$ by $\os_{n, \rho}$,
the coordinate of $B_{n, \rho}'$ in $\Gamma_2^n$ by $\ot_{n, \rho}'$,
and the coordinate of $B_{n, \rho}''$ in $\Gamma_2^{\lfloor n \rho \rfloor}$ by $\ot_{n, \rho}$.
Then it is clear that
$$
\ogamma_{n, \rho}=
\ \overline{B_{n, \rho}' A_{n, \rho} B_{n, \rho}'' A_{n, \rho}}=
\ \overline{(\ot_{n, \rho}', 0) (\os_{n, \rho}, \ovarphi_{n, \rho})
(\ot_{n, \rho}'', 0) (\os_{n, \rho}, -\ovarphi_{n, \rho})},
$$
is a period four orbit for the induced billiard map on $\Gamma_1\cup \Gamma_2$.

Let $\theta_{n, \rho}'$,  $\theta_{n, \rho}''$ and $\theta_{n, \rho}$ be the angles
formed by horizontal axis with $\overline{A_{n, \rho} B_{n, \rho}'}$,
$\overline{A_{n, \rho} B_{n, \rho}'}$ and  $\cN$ respectively.
Then we have
\beqn
\ovarphi_{n, \rho}=\theta_{n, \rho}'-\theta_{n, \rho}=\theta_{n, \rho}-\theta_{n, \rho}''.
\eeqn
As $n\to \infty$,
{\allowdisplaybreaks
\beq\label{per 4 est 1}
\begin{split}
\theta_{n, \rho}' &= \tan\theta_{n, \rho}' +
\cO\left(\frac{1}{n^2}\right) = \frac{Q}{n} + \cO\left(\frac{1}{n^2}\right),
\\
\theta_{n, \rho}'' &= \tan\theta_{n, \rho}'' +
\cO\left(\frac{1}{n^2}\right) = \frac{Q\rho^{-1}}{n} + \cO\left(\frac{1}{n^2}\right),
\end{split}
\eeq
}which implies that
{\allowdisplaybreaks
\beq\label{per 4 est 2}
\begin{split}
\theta_{n, \rho} &= \frac{\theta_{n, \rho}'+\theta_{n, \rho}''}{2}
=\frac{Q(1+\rho^{-1})}{2n}  + \cO\left(\frac{1}{n^2}\right),
\\
\ovarphi_{n, \rho} &=\frac{\theta_{n, \rho}' - \theta_{n, \rho}''}{2}
=\frac{Q(1-\rho^{-1})}{2n}  + \cO\left(\frac{1}{n^2}\right).
\end{split}
\eeq}We again approximate $\Gamma_1^0$ by the osculating circle at $P_{14}^0$,
approximate $\Gamma_2^n$ by the osculating circle at $P_{23}^n$,
and
approximate $\Gamma_2^{\lfloor n \rho \rfloor}$ by the osculating circle
at $P_{23}^{\lfloor n \rho \rfloor}$,
then there exist constants $c_1, c_2\in \IR$ such that
\beq\label{per 4 est 3}
\theta_{n, \rho}= \cK_{14} \os_{n, \rho}+c_1 \os_{n, \rho}^2, \ \
\text{and} \ \
\theta_{n, \rho}^*=\cK_{23} \ot_{n, \rho}^*+ c_2 \left( \ot_{n, \rho}^* \right)^2
\ \text{for}\ *=\prime, \ \prime\prime.
\eeq
Then the estimate \eqref{per 4 est} directly follows from \eqref{per 4 est 1}, \eqref{per 4 est 2} and \eqref{per 4 est 3}.
\end{proof}

\subsection{Flatness of $\bn$ at $P_{14}$ and $P_{23}$}

Recall that $\bn\equiv 0$ on $J_3\cup J_4$.
We introduce arclength parameter $s$ on $\Gamma_1$ such that $P_{14}$ corresponds to $s=0$,
and arclength parameter $t$ on $\Gamma_2$ such that $P_{23}$ corresponds to $t=0$.
Then we rewrite $\bn$ as $\bn_1(s)$ on $\Gamma_1$ and by $\bn_2(t)$ on $\Gamma_2$.
Note that $\bn_1(0)=\bn_2(0)=0$.

Under the assumption that \eqref{def G sum zero} holds for the special orbits $\ogamma_n$
and $\ogamma_{n, \rho}$, we shall prove the flatness of the deformation function $\bn$
at the two gluing points $P_{14}$ and $P_{23}$, that is,

\begin{proposition}\label{prop: vanishing bn stad}
If $\sum_{\ogamma_n} G=\sum_{\ogamma_{n, \rho}} G=0$ for any $n\ge 1$
and $\rho>1$, then
\beq\label{vanishing bn stad}
\bn_1^{(d)}(0)=\bn_2^{(d)}(0)=0, \ \ \text{for any}\ d\ge 1.
\eeq
\end{proposition}

 Stadium Rigidity Theorem 
immediately follows from Proposition~\ref{prop: vanishing bn stad},
since $\bn_1$ and $\bn_2$ are constantly zero as they are
both analytic. In the rest of this section, we prove
Proposition~\ref{prop: vanishing bn stad}.

\begin{proof}[Proof of Proposition~\ref{prop: vanishing bn stad}]
If \eqref{vanishing bn stad} does not hold, we can find a minimal
integer $d_i\ge 1$ such that $c_i:=\bn_i^{(d_i)}(0)\ne 0$, $i=1, 2$.
Recall that $\bn_1(0)=\bn_2(0)=0$.

We then further find $u>0$ such that the Taylor expansion is
valid for $\bn_1(s)$ on $s\in [0, u]$ up to order $d_1$, and for
$\bn_2(t)$ on $t\in [0, u]$ up to order $d_2$. In this way, for
any $s, t\in [0, u]$, we write
\beq\label{Taylor expansion}
\bn_1(s) = c_1 s^{d_1} + \cO(s^{d_1+1}), \  \ \text{and} \ \
\bn_2(t) = c_2 t^{d_2} + \cO(t^{d_2+1}).
\eeq
By Lemma~\ref{lem: per 2}, for sufficiently large $n\ge 1$, there is a period two trajectory
$\ogamma_n=\overline{(\os_n, 0) (\ot_n, 0)}$ for the induced billiard map,
such that $\os_n, \ot_n \in [0, u]$. Then by the assumption $\sum_{\ogamma_n} G=0$,
we immediately get
\beqn
\bn_1(\os_n) + \bn_2(\ot_n)=0.
\eeqn
By \eqref{per 2 est} and \eqref{Taylor expansion}, we obtain
\beq\label{cond c d}
c_1\left(\frac{Q}{n\cK_{14}}\right)^{d_1} + c_2 \left(\frac{Q}{n\cK_{23}}\right)^{d_2}
+ \cO\left(\frac{1}{n^{\min\{d_1, d_2\}+1}}\right)=0.
\eeq
We claim that $d_1=d_2$. Otherwise, to fix ideas let us assume
that $d_2>d_1=\min\{d_1, d_2\}$.
Multiplying $\left(\frac{Q}{n\cK_{14}}\right)^{-d_1}$ on both sides
of \eqref{cond c d}, and then letting $n\to \infty$,
we immediately get $c_1=0$, which contradicts  our choice of $d_1$.

Now we set $d_1=d_2=d$. We use the same trick again,
that is, multiplying $\left(\frac{Q}{n}\right)^{-d}$ on both sides
of \eqref{cond c d}, and then letting $n\to \infty$, we get
\beq\label{rel c1 c2 1}
c_1 \cK_{14}^{-d} + c_2 \cK_{23}^{-d}=0.
\eeq

By Lemma~\ref{lem: per 4}, for any fixed $\rho>1$ and any
sufficiently large number $n\ge 1$, there is a period four trajectory
$$
\ogamma_{n, \rho}=\overline{(\ot_{n, \rho}', 0) (\os_{n, \rho}, \ovarphi_{n, \rho})
(\ot_{n, \rho}'', 0) (\os_{n, \rho}, -\ovarphi_{n, \rho})},
$$
for the induced billiard map,
such that $\os_{n, \rho}, \ot_{n, \rho}',  \ot_{n, \rho}''\in [0, u]$
and the angle $\ovarphi_{n, \rho}$ is very close to zero.
By the assumption $\sum_{\ogamma_{n, \rho}} G=0$, we have
\beqn
2\bn_1(\os_{n, \rho})\cos\ovarphi_{n, \rho} + \bn_2(\ot_{n, \rho}') + \bn_2(\ot_{n, \rho}'')=0.
\eeqn
By \eqref{per 4 est} and \eqref{Taylor expansion},
as well as the fact that
$\cos\ovarphi_{n, \rho}=1+\cO(\ovarphi_{n, \rho}^2)=1+\cO(n^{-2})$ as $n\to \infty$,
we obtain
\beqn\label{cond c d 1}
2c_1\left(\frac{Q(1+\rho^{-1})}{2n\cK_{14}}\right)^{d}
+c_2 \left(\frac{Q}{n\cK_{23}}\right)^{d} + c_2\left(\frac{Q\rho^{-1}}{n\cK_{23}}\right)^{d}
+ \cO\left(\frac{1}{n^{d+1}}\right)=0.
\eeqn
Multiplying $\left(\frac{Q}{n}\right)^{-d}$ on both sides of the above, and letting $n\to \infty$,
we get
\beq\label{rel c1 c2 2}
c_1 \cdot 2^{1-d}\cK_{14}^{-d} \left(1+\rho^{-1}\right)^{d}
+c_2 \cK_{23}^{-d} (1+\rho^{-d}) =0.
\eeq
Combining \eqref{rel c1 c2 1} and \eqref{rel c1 c2 2}, we obtain a homogeneous system
of linear equations with two variables $c_1$ and $c_2$, whose
coefficient matrix has determinant
\beqn
\det
\begin{pmatrix}
\cK_{14}^{-d} & \cK_{23}^{-d} \\
2^{1-d}\cK_{14}^{-d} \left(1+\rho^{-1}\right)^{d} & \cK_{23}^{-d} (1+\rho^{-d})
\end{pmatrix}
=\frac{1+\rho^{-d}-2^{1-d}(1+\rho^{-1})^d}{\cK_{14}^d \cK_{23}^d}.
\eeqn
For any $d\ge 1$, it is easy to pick $\rho>1$ such that the above determinant
is non-zero, and hence $c_1=c_2=0$, which is a contradiction.
Therefore, we must have $\bn_1^{(d)}(0)=\bn_2^{(d)}(0)=0$ for any $d\ge 1$.
\end{proof}

\bigskip

\section{Analysis of Periodic Orbits with Rotation Number $\pm\frac{n}{2n+1}$}\label{sec8}

\subsection{Periodic Orbits with Rotation Number  $\pm\frac{n}{2n+1}$}

Let $\Omega$ be a Bunimovich squash-type stadium  in $\cM^m_{ss}$ for $m\ge 3$.
We introduce all the possible periodic $(2n+1)$ orbits
whose rotation number is either $\frac{n}{2n+1}$ or $-\frac{n}{2n+1}$.
More precisely, for any integer $n\ge 1$, 
we consider the periodic orbit $\gamma_{n}^i$
associated with the symbolic code
\beq\label{symbol ga n i}
(i\underbrace{12\cdots 12}_{2n}),
\eeq
for $i=1, 2, 3, 4$. 
The existence and uniqueness of the orbit $\gamma_{n}^{i}$
can be proven in a similar fashion as in Lemma~\ref{lem: gamma_n uniq},
that is,
\begin{itemize}
\item the orbit $ \gamma^1_n$ is the unique global maximum point of the length function
\beqn
L(r_0, r_1, \dots, r_{2n})=\sum_{k=0}^{2n} \tau(r_k, r_{k+1}), \ \   \text{with} \ r_{2n+1}=r_0,
\eeqn
for $(r_0, r_1, \dots, r_{2n})\in  \Gamma_1\times  (\Gamma_1\times \Gamma_2)^n$
with the restriction that $r_0\ge 0 \ge r_1$.
Such restriction makes $r_0$ and $r_1$ fall in different sides of  $x$. 
Similarly, $ \gamma^2_n$ is the global maximum point of the same length function
but for $(r_0, r_1, \dots, r_{2n})\in  \Gamma_2\times  (\Gamma_1\times \Gamma_2)^n$
with the restriction that $r_0\ge 0 \ge r_{2n}$. 
\item 
the orbit $\gamma^3_n$ corresponds to unfolded orbit $\widetilde{\gamma}^3_n$,
which is the unique global maximum point of the length function 
on the double cover table $\widetilde{\Omega}$ (see Fig.~\ref{fig:stadia}, right, which has 
a symmetric reflection line through $\Gamma_3$), given by
\beqn
L(r_0, r_1, \dots, r_{2n})=\sum_{k=0}^{2n-1} \widetilde{\tau}(r_k, r_{k+1})
\eeqn
for $(r_0, r_1, \dots, r_{2n})\in  \widetilde{\Gamma}_2\times  (\Gamma_1\times \Gamma_2)^n$
with the restriction that $r_0^*=r_{2n}$, where $r_0^*\in \Gamma_2$ denotes the reflected point
of $r_0\in \widetilde{\Gamma}_2$ by the symmetry line through $\Gamma_3$. 
The orbit $\gamma^4_n$ can be obtained in a similar fashion by considering the double cover table
with symmetry line through $\Gamma_4$.
\end{itemize}

As the winding number of a periodic orbit is counted in the counter-clockwise direction,
the rotation number of $\gamma_n^i$ equals to $(-1)^i\frac{n}{2n+1}$ for $i=1, 2, 3, 4$. 
Switching the roles of $\Gamma_1$ and $\Gamma_2$,
we could also consider the periodic orbit $\hgamma_{n}^i$ associated with
\beqn
(i\underbrace{21\cdots 21}_{2n}).
\eeqn
It is easy to see that $\hgamma_n^i$ is the inverse orbit of $\gamma_n^i$. Therefore,
$\hgamma_n^i$ and $\gamma_n^i$ are of the same total length,
and the rotation number of $\hgamma_n^i$ equals to $(-1)^{i+1}\frac{n}{2n+1}$.

\medskip

Along any of the above orbits, i.e., $\gamma^i_n$ and $\hgamma^i_n$ with $i=1, 2, 3, 4$,  
a billiard ball moves from an initial position on $\Gamma_i$,
collides successively between $\Gamma_1$ and $\Gamma_2$
for $2n$ times, and then gets back to the initial position on $\Gamma_i$.
It is easy to see that the first $n$ collisions approach the period two orbit $\gamma^*$, while
the last $n$ collisions become away from $\gamma^*$. 
There are essentially two types: 
\begin{itemize}
\item $\gamma^i_{n}$ and $\hgamma^i_{n}$ are {\it symmetric} for $i=3, 4$, i.e., 
the number of times that the billiard ball lies on $\Gamma_1$
is the same as that on $\Gamma_2$ in one period; 
\item $\gamma^i_{n}$ and $\hgamma^i_{n}$ are {\it asymmetric} for $i=1, 2$, i.e., 
the number of times that the billiard ball lies on $\Gamma_1$
is different from that on $\Gamma_2$ in one period.
\end{itemize}
It turns out that the trace of the above periodic orbits in the same type
are similar. To illustrate the corresponding dynamics, we picture 
the following two sequence of periodic orbits $\gamma_{n}^2$ and $\gamma_{n}^3$.
Recall that the period two orbit is denoted by $\gamma^*=\overline{AB}$.
Along the asymmetric orbit $\gamma_{n}^{2}$ (see Fig.~\ref{fig:gamma23}, upper), 
a billiard ball lies on the closest position near $A$ (or $B$) at the $n$-th and 
near $B$ (or $A$) at the $(n+1)$-th collision if $n$ is odd (if $n$ is even); 
along the symmetric orbit
$\gamma_{n}^3$ (see Fig.~\ref{fig:gamma23}, lower), every
free path crosses $\overline{AB}$ except the middle path
from the $n$-th collision to the $(n+1)$-th collision.

\begin{figure}[h]
\begin{center}
\includegraphics[width=1\textwidth]{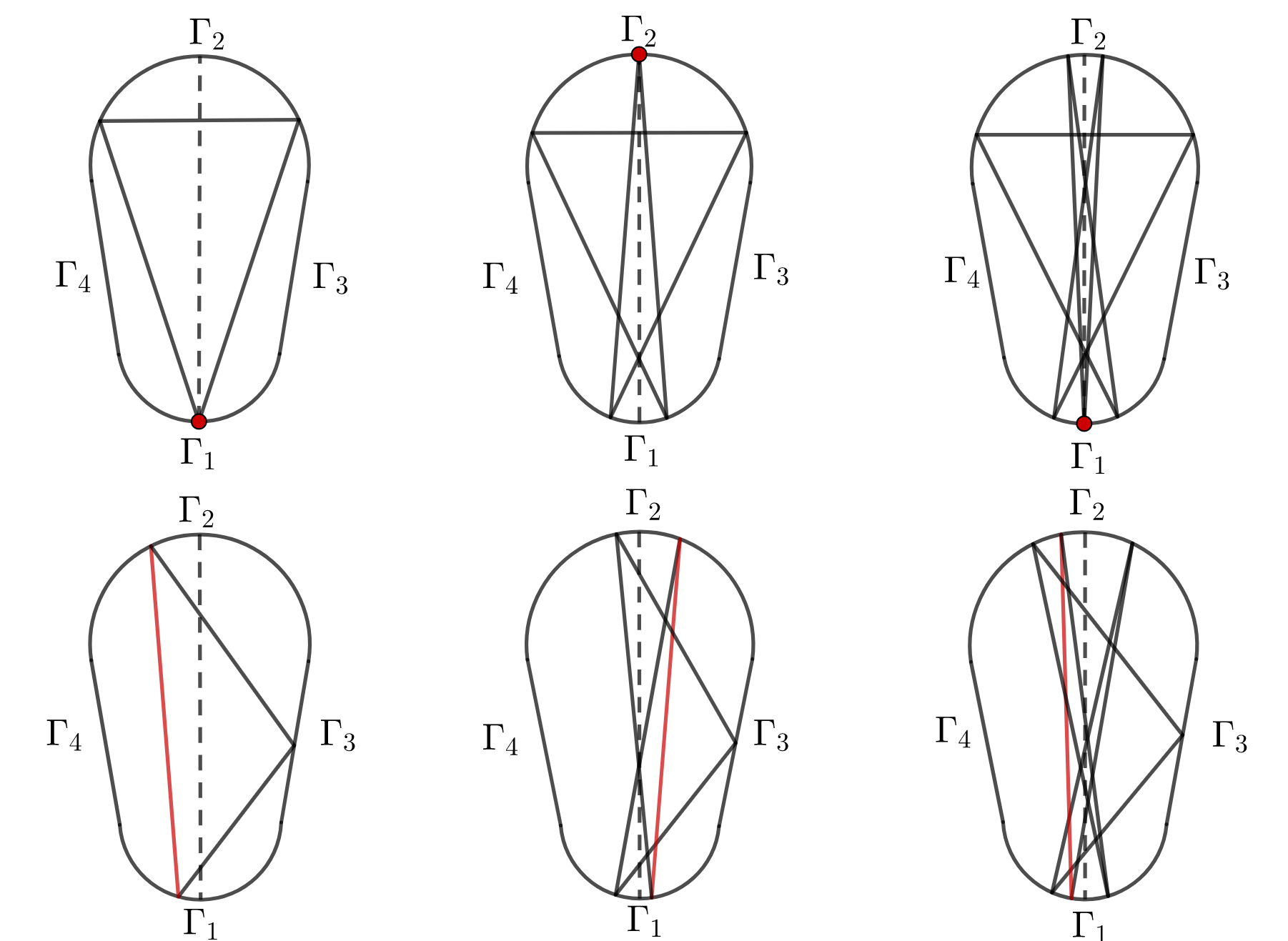}
\end{center}
\caption{Upper: $\gamma_{n}^{2}$ with $n=1, 2, 3$; \
Lower: $\gamma_{n}^{3}$ with $n=1, 2, 3$.}
\label{fig:gamma23}
\end{figure}

\subsection{The Homoclinic Semi-orbits and the Shadowing Estimates}\label{subsect5.3}

For $i=1, 2, 3$ or $4$,
we denote the collision points of $\gamma_{n}^i$ by
\begin{align*}
y^i_{n}(0) \mapsto x^i_n(1) \mapsto y^i_n(1) \mapsto \dots
\mapsto x^i_{n}(n) \mapsto y^i_{n}(n) ,
\end{align*}
where 
\begin{itemize}
\item
at the initial stage, we denote by $y^i_n(0)=(t^i_n(0), \psi^i_n(0))$ the collision point on $\Gamma_i$;
\item
at the stage of $2n$ successive collisions between $\Gamma_1$ and $\Gamma_2$,
we denote
\begin{align*}
\text{on}\ \Gamma_1: \ x^i_{n}(k) & =(s^i_{n}(k), \varphi^i_{n}(k)), \ k=1, 2, \dots, n; \\
\text{on}\ \Gamma_2: \ y^i_{n}(k) & =(t^i_{n}(k), \psi^i_{n}(k)), \ k=1, 2, \dots, n.
\end{align*}
\end{itemize}
Similarly, we denote the collision points of the inverse orbit $\hgamma_{n}^i$ by
\begin{align*}
\hx^i_{n}(0) \mapsto \hy^i_n(1) \mapsto \hx^i_n(1) \mapsto \dots
\mapsto \hy^i_{n}(n) \mapsto \hx^i_{n}(n) ,
\end{align*}
with coordinates $\hx^i_n(k)=(\widehat{s}^i_{n}(k), \widehat{\varphi}^i_{n}(k))$
for $k=0, 1, \dots, n$ and 
$\hy^i_n(k)=(\widehat{t}^i_{n}(k), \widehat{\psi}^i_{n}(k))$ for $k= 1, \dots, n$.
Unlike the palindromic orbits $\gamma_n$, the time reversibility does not hold 
for $\gamma^i_n$ or $\hgamma_n^i$ itself,
but $\gamma_n^i$ and $\hgamma_n^i$ form an involution pair, that is, 
\beq\label{involution s}
\begin{split}
\widehat{s}_n^i(n+1-k) = s_n^i(k), \ \widehat{\varphi}_n^i(n+1-k) = -\varphi_n^i(k), \ k=1, 2, \dots, n, \\
\widehat{t}_n^i(n+1-k) = t_n^i(k), \ \widehat{\psi}_n^i(n+1-k) = -\psi_n^i(k), \ k=1, 2, \dots, n.
\end{split}
\eeq
Also, $\widehat{s}_n^i(0)=t_n^i(0)$ and  
$\widehat{\varphi}_n^i(0)=-\psi_n^i(0)$.  \\

\medskip
 
Using similar arguments as in Lemma~\ref{lem: rough esta} and Lemma~\ref{lem: homoclinic}, 
we can define a homoclinic semi-orbit of the period two orbit $\gamma^*=\overline{xy}$  by
\beqn
\gamma^i_\infty:=( x^i_\infty(0)   \ x^i_\infty(1) \ y^i_\infty(1) \ x^i_\infty(2) \ y^i_\infty(2)  \dots ),
\eeqn
which corresponds to the symbolic code $(  i 1212 \cdots)$,
where 
\beqn
\begin{split}
x^i_\infty(k) &=\lim_{n\to\infty} x^i_n(k), \ k= 1,    2, \dots; \\
y^i_\infty(k) &=\lim_{n\to\infty} y^i_n(k), \ k= 0,  1,    2, \dots.
\end{split}
\eeqn
We write the coordinates $x^i_\infty(k)=(s^i_{\infty}(k),  \varphi^i_{\infty}(k))$
for $k=1, 2, \dots$, and 
$y^i_\infty(k)=(t^i_{\infty}(k),  \psi^i_{\infty}(k))$ for $k= 0, 1, 2, \dots$.

\medskip

Similar to Lemma~\ref{lem: finer est} and Lemma~\ref{lem: finer est'},
we can provide estimates for the convergence of $\gamma^i_\infty$ to $\gamma^*$, 
and for the shadowing of $\gamma^i_n$ along $\gamma^i_\infty$ as follows.
 
\begin{lemma}\label{lem: finer est 1}  Let $i=1, 2, 3$ or $4$.
\begin{itemize}
\item[(a)]
The following estimates hold for the homoclinic orbit $\gamma_\infty^i$:
{\allowdisplaybreaks
\beq
\label{finer est ga n i}
\begin{split}
x^i_\infty(k)=\lambda^{-k}(C_s^{i}, C_\varphi^{i}) +  \cO(\lambda^{-1.5k}),
& \ \ k= 1,    2, \dots, \\
y^i_\infty(k)=\lambda^{-k} (C_t^{i}, C_\psi^{i}) + \cO(\lambda^{-1.5k}),
& \ \ k= 0,  1,   2,\dots,
\end{split}
\eeq
}where the constants $C^i_s, C^i_\varphi, C^i_t$ and $C_\psi^{i}$ satisfy 
the following relation:
\beq\label{C relation 1}
\dfrac{C_\varphi^{i}}{C_s^{i}}=\frac{\lambda^{-1}-\lambda}{4a_2\tau^*}, \ \
\dfrac{C_\psi^{i}}{C_t^{i}}=\frac{\lambda^{-1}-\lambda}{4a_1\tau^*}, \ \ \text{and} \ \
\dfrac{C_t^{i}}{C_s^{i}}=-\frac{1+\lambda^{-1}}{2a_2}=
-\frac{2a_1}{1+\lambda}.
\eeq
\item[(b)]
The following estimates hold for the periodic orbit $\gamma_n^i$:
{\allowdisplaybreaks
\beq\label{finer est ga n i i}
\begin{split}
x^i_n(k)-x^i_\infty(k)=\lambda^{k-n}  (C^i_{s,k}, C^i_{\varphi, k}) + \cO(\lambda^{0.5k-n}) ,
& \ \ k=1, \dots, \left\lfloor \frac{n}{2} \right\rfloor +1, \\
y^i_n(k)-y^i_\infty(k)=\lambda^{k-n} (C^i_{t,k}, C^i_{\psi, k}) + \cO(\lambda^{0.5k-n}),
& \ \ k=0, 1, \dots, \left\lfloor \frac{n}{2} \right\rfloor +1.
\end{split}
\eeq
}where the constants $C^i_{s,k}, C^i_{\varphi, k}, C^i_{t,k}$ and $C^i_{\psi, k}$
are given by
{\allowdisplaybreaks
\beq\label{C1 relation 1}
\begin{split} 
C^i_{s,k}=C_s^i \lambda^{-1} (\lambda^{-2k}+\Theta_{i}), \ \ & \ \ C^i_{\varphi, k}= C^i_\varphi \lambda^{-1} (\lambda^{-2k}-\Theta_i), \\
C^i_{t,k}=C^i_t\lambda^{-1} (\lambda^{-2k}+\Theta_{i}), \  \ & \ \  C^i_{\psi, k}=C^i_\psi \lambda^{-1} (\lambda^{-2k}-\Theta_i).
\end{split}
\eeq
}Here $\Theta_{i}=-1$ if $i=1, 2$, and $\Theta_{3}=\dfrac{\tan\theta_1 + \cK_A }{\tan\theta_1 - \cK_A}$, $\Theta_{4}=\dfrac{\tan\theta_2 + \cK_B }{\tan\theta_2 - \cK_B}$.
\item[(c)]
Alternatively, the following estimates hold for the periodic orbit $\gamma_n^i$:
{\allowdisplaybreaks
\beq
\label{finer est ga n i infty''''}
\begin{split}
x^i_n(k)-x^i_\infty(k)=\lambda^{k-n} \left[v^i_\infty(2k) + o(1)\right] , \ \ \ \ \ \
& \ \ k=1, \dots, \left\lfloor{\frac{n}{2}}\right\rfloor+1, \\
y^i_n(k)-y^i_\infty(k)=\lambda^{k-n} \left[ v^i_\infty(2k+1) + o(1)\right], \  
& \ \ k= 0, 1, \dots, \left\lfloor{\frac{n}{2}}\right\rfloor+1,
\end{split}
\eeq
}where the vectors $v^i_\infty(m)\in \IR^2$, $m=0^\pm, 0, 1, 2, \dots$, 
has uniformly bounded magnitudes. 
\end{itemize}
\end{lemma}

The proof of Lemma~\ref{lem: finer est 1} is similar to 
that of Lemma~\ref{lem: finer est} and Lemma~\ref{lem: finer est'},
and thus  we leave it to the reader. 
The only difference between the proofs of 
Lemma~\ref{lem: finer est} and Lemma~\ref{lem: finer est 1}
is on how to obtain the constant $\Theta_i$. That is, in Lemma~\ref{lem: finer est} we 
obtain $\Theta=1$ using the self
time-reversibility for the palindromic orbit $\gamma_n$, while 
in Lemma~\ref{lem: finer est 1} we obtain the value of $\Theta_i$ by the fact that 
\begin{itemize}
\item[(i)]
the middle collision point  is close to $\gamma^*$ for $i=1, 2$; 
\item[(ii)]
the middle free path is almost parallel to $\gamma^*$ for $i=3, 4$. 
\end{itemize}
Similar to Lemma~\ref{lem: finer est 1}, corresponding estimates also hold for  
the inverse periodic orbit $\hgamma^i_n$ and the corresponding homoclinic semi-orbit 
$\hgamma^i_\infty$, where we simply put `hat'  on each term, i.e.,
$\widehat{x}^i_n$, $\widehat{x}^i_\infty$, $\widehat{C}^i_s$, $\widehat{\Theta}_i$, etc.
By the above facts (i) and (ii), we shall have 
\beq
\label{compare coeff}
\begin{split}
& (\hC^i_s, \hC^i_\varphi)=(-C^i_s, C^i_\varphi) 
\ \text{and} \ 
(\hC^i_t, \hC^i_\psi)=(-C^i_t, C^i_\psi) 
\ \ \text{if} \ i=1, 2; \\
& (\hC^i_s, \hC^i_\varphi)=(C^i_s, -C^i_\varphi) 
\ \text{and} \ 
(\hC^i_t, \hC^i_\psi)=(C^i_t, -C^i_\psi) 
\ \ \text{if} \ i=3, 4.
\end{split}
\eeq

 \bigskip

\section{Proof of Theorem~\ref{thm: main3}}\label{sec: proof of thm3}

\subsection{The Length Growth of the Homoclinic Semi-orbit $\gamma_{\infty}^{i}$}
\label{sec:gm-infty-length}

In this subsection, we show that for $i=1,2,3$ or $4$, the length growth
of the homoclinic semi-orbit $\gamma_\infty^{i}$ is exponentially asymptotic to that of the period two orbit
$\gamma^*$. More precisely, we denote
$l_\infty^i(0)=\tau(t_\infty^i(0), s_\infty^i(1))$ and 
\beq\label{tau k inf}
l_{\infty}^i(k)=\tau(s_\infty^i(k), t_{\infty}^i(k))+
\tau(t_{\infty}^i(k), s_\infty^i(k+1)) \ \ \text{for any} \ k\ge 1.
\eeq

\begin{lemma}\label{lem: sum 1}
There exists a constant $\cQ_\infty^i\in \IR$ such that  
\beq\label{sum 1}
l^i_{\infty}(k) -2 \tau^* = \cQ_\infty^i \lambda^{-2k}
+ \cO\left(\lambda^{-2.5 k}\right).
 \eeq
\end{lemma}

\begin{proof} 
Note that $\tau(s_\infty^i(k), t_{\infty}^i(k))- \tau^*$ is the difference of the free path
travelled by moving from $(s, t)=(0, 0)$ to 
$$(s, t)= (s_\infty^i(k), t_{\infty}^i(k))=\lambda^{-k}(C_s^{i}, C_t^{i}) +  \cO(\lambda^{-1.5k}),$$ 
where the estimates are due to Lemma~\ref{lem: finer est 1}. By \eqref{var path 0},
\begin{align*}
& \tau(s_\infty^i(k), t_{\infty}^i(k))- \tau^* \\
=& \frac{1}{2\tau^*}\left[ a_1 \left(s_\infty^i(k) \right)^2
+ 2s_\infty^i(k)  t^i_\infty(k) +a_2 \left(t^i_\infty(k)\right)^2 \right]
 + \cO\left(\left(\left(s^i_\infty(k)\right)^2+
 \left(t^i_\infty(k)\right)^2\right)^{\frac32}\right) \\
=&
\dfrac{a_1\left(C_s^i \right)^2 + 2C_s^iC_t^i + a_2\left(C_t^i \right)^2 }{2\tau^*} \cdot \lambda^{-2k} +
\cO\left(\lambda^{-2.5 k}\right)
=: \cQ(C_s^i, C_t^i)\cdot \lambda^{-2k} +
\cO\left(\lambda^{-2.5 k}\right),
\end{align*}
where $\cQ(X, Y)$ is a quadratic form defined by 
\beq
\label{def quadratic form}
\cQ(X, Y)=\dfrac{a_1 X^2 + 2XY + a_1 Y^2}{2\tau^*}.
\eeq
In a similar fashion, we also get
\beqn
\tau(t_{\infty}^i(k), s_\infty^i(k+1)) -\tau^* = 
\cQ(\lambda^{-1} C_s^i, C_t^i) \cdot \lambda^{-2k} +
\cO\left(\lambda^{-2.5 k}\right).
\eeqn
Hence \eqref{sum 1} is proven if we set
$
\cQ_\infty^i:=\cQ(C_s^i, C_t^i)+\cQ(\lambda^{-1} C_s^i, C_t^i).
$
\end{proof}
 
Similarly, we denote by
$
\hgamma_\infty^i=(\hx^i_n(0) \hy^i_n(1) \hx^i_n(1) \dots)
$
the homoclinic semi-orbit obtained by the limit of 
the inverse periodic orbits $\hgamma_n^i$.
Set 
$\hl_\infty^i(0)=\tau(\hs_\infty^i(0), \htt_\infty^i(1))$ and 
\beq\label{tau k inf inverse}
\hl_{\infty}^i(k)=\tau(\htt_\infty^i(k), \hs_{\infty}^i(k))+
\tau(\hs_{\infty}^i(k), \htt_\infty^i(k+1)) \ \ \text{for any} \ k\ge 1.
\eeq
Similar to Lemma~\ref{lem: sum 1}, there exists 
a constant $\hcQ_\infty^i\in \IR$ such that  
\beq\label{sum 1'}
\hl^{\, i}_{\infty}(k) -2 \tau^* = \hcQ_\infty^i \lambda^{-2k}
+ \cO\left(\lambda^{-2.5 k}\right).
 \eeq

Now we define $B^i=B^i_0+B^i_+ + B^i_-$, where 
\beq\label{def B i}
B^i_0=2\tau^*- l^i_\infty(0) - \hl^{\, i}_\infty(0) , 
\ \ 
B^i_+=\sum_{k=1}^\infty \left( 2\tau^* - l^i_{\infty}(k)  \right), 
\ \
B^i_-=\sum_{k=1}^\infty \left( 2\tau^* - \hl^{\, i}_{\infty}(k) \right).
\eeq

\subsection{The Length Difference between $\gamma_{n}^{i}$ and $\gamma_{\infty}^{i}$}

For the periodic orbit $\gamma_n^i$, 
we denote
$l_n^i(0)=\tau(t_n^i(0), s_n^i(1))$ and 
\beq\label{tau k n}
l_{n}^i(k)=\tau(s_n^i(k), t_{n}^i(k))+
\tau(t_{n}^i(k), s_n^i(k+1)) \ \ \text{for any} \ 1\le k\le n.
\eeq
Similarly, for the inverse periodic orbit $\hgamma_n^i$, we denote
$\hl_n^i(0)=\tau(\hs_n^i(0), \htt_n^i(1))$ and 
\beq\label{tau k n inverse}
\hl_{n}^i(k)=\tau(\htt_n^i(k), \hs_{n}^i(k))+
\tau(\hs_{n}^i(k), \htt_n^i(k+1)) \ \ \text{for any} \ 1\le  k\le n.
\eeq

We shall compare the length difference between 
$(\gamma_{n}^{i}$,  $\hgamma_{n}^{i})$ and $(\gamma_{\infty}^{i}$, $\hgamma_{\infty}^{i})$.
More precisely,
we set $\ell=\lfloor n/2\rfloor$, $\ell'=n-1-\ell$.
The total length of $\gamma_n^i$ (and $\hgamma_n^i$) is given by
\beq
\label{def L n i}
L\left(\gamma_n^i \right)=L\left(\hgamma_n^i \right)
=\sum_{k=0}^{\ell} l_n^i(k) 
+\sum_{k=0}^{\ell'} \hl_n^i(k) 
+\tau\left( s^i_n(\ell+1), \htt^i_n(\ell'+1) \right),
\eeq
in which we notice that $\htt^i_n(\ell'+1)=t_n^i(\ell+1)$. 
Correspondingly, we define
\beq
\label{def L infty i}
L_\infty^i(n)=
\sum_{k=0}^{\ell} l_\infty^i(k) 
+\sum_{k=0}^{\ell'} \hl_\infty^i(k) 
+\tau\left( s^i_\infty(\ell+1), \htt^i_\infty(\ell'+1) \right),
\eeq
where $l_\infty^i(k) $ and $\hl_\infty^i(k)$
are given by \eqref{tau k inf} and \eqref{tau k inf inverse} respectively.

\begin{lemma}
\label{lem: sum n inf}
There exists 
a constant $\cC_\infty^i\in \IR$ such that  
\beq\label{n inf sum 1}
L\left(\gamma_n^i \right) - L_\infty^i(n)
=\cC_\infty^i \lambda^{-n} + \cO\left(\lambda^{-1.5n}\right). 
 \eeq
\end{lemma}

\begin{proof}
By Lemma~\ref{lem: var path} and Lemma~\ref{lem: finer est 1},
for any $k=0, 1, \dots, \ell+1$, 
we can write
\beqn
\tau(s_n^i(k), t_n^i(k)) -\tau(s_{\infty}^i(k),t_{\infty}^i(k))=I_k^0 + J_k^0 +\cO\left(\lambda^{-1.5n}\right),
\eeqn
where \begin{eqnarray*}
I_k^0 &=& -\sin  \varphi^i_{\infty}(k) \ (s_n^i(k)-s^i_{\infty}(k))
+ \sin  \psi^i_{\infty}(k) \ (t_n^i(k)-t_\infty^i(k)),   \\
J_k^0 &=& \frac{1}{2\tau(s^i_\infty(k), t^i_\infty(k))}\left[ 
\alpha(s^i_{\infty}(k))(s_n^i(k)-s^i_{\infty}(k))^2  
+\alpha(t^i_{\infty}(k)) (t^i_n(k)-t^i_\infty(k))^2 \right.\\
& &  \ \ \ \ \ \ \ \ \ \ \ \ \ \ \ \  \ \ \left.
+2\cos  \varphi^i_{\infty}(k) \cos \psi^i_{\infty}(k)  (s_n^i(k)-s^i_{\infty}(k)) (t^i_n(k)-t^i_\infty(k)) \right].
\end{eqnarray*}
In a similar fashion, we can write
\beqn
\begin{split}
&\tau(t_n^i(k), s_n^i(k+1)) -\tau(t_{\infty}^i(k), s_{\infty}^i(k+1))=I_k^1 + J_k^1 +\cO\left(\lambda^{-1.5n}\right), \\
&\tau(\htt_n^{\, i}(k), \hs_n^{\, i}(k), ) 
-\tau(\htt_{\infty}^{\, i}(k), \hs_{\infty}^{\, i}(k))=I_k^2 + J_k^2 +\cO\left(\lambda^{-1.5n}\right),\\
&\tau(\hs_n^{\, i}(k), \htt_n^{\, i}(k+1)) -\tau(\hs_{\infty}^{\, i}(k), \htt_{\infty}^{\, i}(k+1))=I_k^3 + J_k^3 +\cO\left(\lambda^{-1.5n}\right),
\end{split}
\eeqn
where $I_k^j$ (resp. $J_k^j$), $j=1, 2, 3$, are of similar form as 
$I_k^0$ (resp. $J_k^0$). 
Also, 
\beqn
\tau\left( s^i_\infty(\ell+1), \htt^i_\infty(\ell'+1) \right) - 
\tau\left( s^i_\infty(\ell+1), \htt^i_\infty(\ell'+1) \right)
=I_* + J_* + \cO\left(\lambda^{-1.5n}\right).
\eeqn
Therefore, 
\beqn
L\left(\gamma_n^i \right) - L_\infty^i(n)
=I + J +\cO\left(\lambda^{-1.5n}\right),
\eeqn
where the linear order summation $I$ is 
an telescopic sum, i.e.,
\beqn
\begin{split}
I &:=\sum_{k=0}^{\ell } (I^0_k + I^1_k) + 
\sum_{k=0}^{\ell' } (I^2_k + I^3_k) + I_*
\\
&=
-\sin \varphi_\infty^i(\ell+1)  s^i_\infty(\ell+1)  
-\sin \widehat{\varphi}_\infty^{\, i}(\ell+1)  \hs^i_\infty(\ell+1)=\cO\left(\lambda^{-1.5n}\right),
\end{split}
\eeqn
where the error estimate is due to Lemma~\ref{lem: finer est 1}, Part (a) and \eqref{compare coeff}. 
Now we focus on the computation of the second order summation $J$,
and by Lemma~\ref{lem: finer est 1}, Part (c), we have 
\beqn
J:=\sum_{k=0}^{\ell } (J^0_k + J^1_k) + 
\sum_{k=0}^{\ell' } (J^2_k + I^3_k) + J_*
=
\sum_{k=\lfloor\ell/2\rfloor}^{\ell } (J^0_k + J^1_k) + 
\sum_{k=\lfloor\ell/2\rfloor}^{\ell' } (J^2_k + I^3_k) + J_* + \cO\left(\lambda^{-1.5n}\right).
\eeqn
For any $k\in [\ell/2, \ell]$, by Lemma~\ref{lem: finer est 1}, Part (a) again,
we have  
\beqn
\cos  \varphi^i_{\infty}(k) \approx 1, \ \ 
\cos \psi^i_{\infty}(k)\approx 1, \ \ 
\alpha(s^i_{\infty}(k)) \approx a_1, \ \
\alpha(s^i_{\infty}(k)) \approx a_2,  \ \
\tau(s^i_\infty(k), t^i_\infty(k)) \approx \tau^*, 
\eeqn
up to errors of order $\cO(\lambda^{-2k})$.
By Lemma~\ref{lem: finer est 1}, Part (b), we have 
\beqn
|s_{n}^i(k)-s_{\infty}^i(k)|\approx 
C_s^i\Theta_i\lambda^{k-n-1}, \ \ 
|t_n^i(k)-t_{\infty}^i(k)|\approx
C_t^i\Theta_i\lambda^{k-n-1}.
\eeqn
up to errors of order $\cO(\lambda^{0.5k-n-1})$. 
Therefore, we have 
\beqn
J_k^0=\Theta_i^2 \cQ(C_s^i, C_t^i) \cdot \lambda^{2(k-n-1)} + 
\cO\left(\lambda^{2(0.5 k-n-1)}\right),
\eeqn
where $\cQ(\cdot, \cdot)$ is the quadratic form $\cQ(\cdot, \cdot)$ defined in \eqref{def quadratic form}.
Hence
\beqn
\label{sum J 0}
\sum_{k=0}^{\ell } J^0_k = 
\sum_{k=\lfloor\ell/2\rfloor}^{\ell } J^0_k + \cO\left(\lambda^{-1.5n}\right)
=\cC_{\infty, 0}^i \lambda^{-n} + 
\cO\left(\lambda^{-1.5n}\right),
\eeqn
where we set
$
\cC^i_{\infty, 0}:= 
\begin{cases}
\Theta_i^2 \cQ(C_s^i, C_t^i) /(\lambda^2-1), \  &
\text{if}  \ n=2\ell, \\ 
\lambda^{-1} \Theta_i^2 \cQ(C_s^i, C_t^i) /(\lambda^2-1), \  &
\text{if}  \ n=2\ell+1.
\end{cases}
$
In a similar fashion, we can show that there are constants 
$\cC^i_{\infty, j}$, $j=1, 2, 3, *$, such that 
sums of $J^j_k$, $j=1, 2, 3$, and $J_*$ satisfy similar estimates as in \eqref{sum J 0}. 
Hence \eqref{n inf sum 1} holds if we take 
$\cC^i_{\infty}=\sum_{j=0}^3 \cC^i_{\infty, j} + \cC^i_{\infty, *}$.
 \end{proof}

\subsection{Proof of Theorem~\ref{thm: main3}}

Recall the definitions of $B_i$ and $L_\infty^i(n)$ 
in \eqref{def B i} and \eqref{def L infty i} respectively.
Note that
\beqn
\begin{split}
  & L_\infty^i(n) -  (2n+1) \tau^*   + B_i \\
=& \sum_{k>\ell}  \left(l^i_{\infty}(k)  - 2\tau^*\right) 
+\sum_{k>\ell'} 
\left(\hl^{\, i}_{\infty}(k)  - 2\tau^*\right)
+
\left[\tau\left( s^i_\infty(\ell+1), \htt^i_\infty(\ell'+1) \right)  -\tau^*\right].
\end{split}
\eeqn
Recall that $\ell=\lfloor n/2\rfloor$ and $\ell'=n-1-\ell$. 
By Lemma \eqref{def L infty i} and also apply its proof to the last term in the above, 
it is easy to show that there is a constant $\cE^i_\infty$ such that 
\beqn
 L_\infty^i(n) -  (2n+1) \tau^*   + B_i 
= \cE^i_\infty \lambda^{-n} + \cO\left(\lambda^{-1.5n}\right).
\eeqn
By Lemma~\ref{lem: sum n inf}, we set $\cD_i=\cC^i_\infty+\cE^i_\infty$, 
then 
\beqn
 L\left(\gamma^i_n\right) -  (2n+1) \tau^*   + B_i 
= \cD_i \lambda^{-n} + \cO\left(\lambda^{-1.5n}\right).
\eeqn
Now we choose $j\in \{1, 2, 3, 4 \}$ such that $B_j>B_i$ for all $i\ne j$.
If it happens that $B_j=B_i$ for distinct $j, i$, we pick $j$ such that $D_j>D_i$.
By such choice, we have 
\beqn
\ML^{\max}_{\Omega} \left( \frac{2n}{2n+1} \right) =  L\left(\gamma^j_n\right)
\eeqn
for sufficiently large $n$. 
Then the proof of Theorem~\ref{thm: main3} is complete if 
we set $B_{\frac12}=B_j$ and $D_{\frac12}=B_j$.

\bibliographystyle{alpha}

\begin{thebibliography}{99}

\bibitem{BdKL18}
Peter Balint, Jacopo de Simoi, Vadim Kaloshin, and Martin Leguil. 
Marked Length Spectrum, homoclinic orbits and the geometry of open dispersing billiards. Communications in Mathematical Physics volume 374, 1531--1575 (2020).

\bibitem{Car10}
Jesus Carnicer.
Weighted interpolation for equidistant nodes. 
Numer. Algorithms 55 (2010), no. 2-3, 223--232.

\bibitem{CWZ19}
Jianyu Chen, Fang Wang, Hong-Kun Zhang.
Markov partition and Thermodynamic Formalism for Hyperbolic Systems with Singularities. Preprint, 49pp, 2017.

\bibitem{CM06}
Nikolai Chernov and Roberto Markarian. Chaotic billiards, volume 127 of Mathematical Surveys and Monographs. American Mathematical Society, Providence, RI, 2006.

\bibitem{Colin84}
Colin de Verdi\`ere, Y.
Sur les longueurs des trajectoires p\'eriodiques d'un billard. 
South Rhone seminar on geometry, III (Lyon, 1983), 122--139,
Travaux en Cours, Hermann, Paris, 1984.


\bibitem{Cr90}
Christopher B. Croke. Rigidity for surfaces of nonpositive curvature. Comment. Math. Helv., 65(1):150--169, 1990.

\bibitem{CrSh98}
Christopher B. Croke and Vladimir A. Sharafutdinov.
Spectral rigidity of a compact negatively curved manifold. Topology 37 (1998), no. 6, 1265--1273.

\bibitem{dSKL19}
Jacopo de Simoi, Vadim Kaloshin, and Martin Leguil. 
Marked Length Spectral determination of analytic chaotic billiards with axial symmetries. Preprint, 57pp, 2019.


\bibitem{dSKW17}
Jacopo de Simoi, Vadim Kaloshin, and Qiaoling Wei. Dynamical spectral rigidity among Z2-symmetric strictly convex domains close to a circle. Ann. of Math. (2), 186(1):277--314, 2017. Appendix B coauthored with H. Hezari.



\bibitem{GWW92}
C. Gordon, D. L. Webb, and S. Wolpert. 
One cannot hear the shape of a drum,
Bull. Amer. Math. Soc. 27(1992), 134--138.

\bibitem{GuLe19}
Colin Guillarmou and Thibault  Lefeuvre.
The marked length spectrum of Anosov manifolds. 
Ann. of Math. (2) 190 (2019), no. 1, 321--344.


\bibitem{GuiKaz80} 
V. Guillemin and D. Kazhdan. Some inverse spectral results for negatively curved 2-manifolds. Topology, 19(3):301--312, 1980.

\bibitem{GuiMel81}
V. Guillemin, Victor and R. Melrose. A cohomological invariant of discrete dynamical systems,
E. B. Christoffel ({A}achen/{M}onschau, 1979) Birkh\"auser, Basel-Boston, Mass. 1981, 672--679.

\bibitem{Hezari17} 
Hamid Hezari. Robin spectral rigidity of nearly circular domains with a reflectional symmetry. Comm. Partial Differential Equations, 42(9):1343--1358, 2017.

\bibitem{HeZel10} 
Hamid Hezari and Steve Zelditch. Inverse spectral problem for analytic 
$(\mathbb{Z}/2\mathbb{Z})^n$ symmetric domains in $\mathbb{R}^n$, 
Geom. Funct. Anal. 20 (2010), no. 1, pp. 160--191.

\bibitem{HeZel12} 
Hamid Hezari and Steve Zelditch. $C^{\infty}$ spectral rigidity of the ellipse. Anal. PDE, 5(5):1105--1132, 2012.

\bibitem{HeZel19} 
Hamid Hezari and Steve Zelditch. One can hear the shape of ellipses of small eccentricity, preprint.


\bibitem{HKS18} 
Guan Huang, Vadim Kaloshin, and Alfonso Sorrentino. On the marked length spectrum of generic strictly convex billiard tables. Duke Math. J., 167(1):175--209, 2018.

\bibitem{Kac66}
Mark Kac. Can one hear the shape of a drum? Amer. Math. Monthly, 73(4, partII):1--23, 1966.

\bibitem{MaFo94}
John Mather  and Giovanni Forni. Action minimizing orbits in Hamiltonian systems ,Transition to chaos in classical and quantum mechanics ({M}ontecatini {T}erme, 1991), Springer, Berlin, 1994 (1589) 92--186.

\bibitem{Otal90} 
Jean-Pierre Otal. Le spectre marque des longueurs des surfaces  courbure negative. Ann. of Math. (2), 131(1):151--162, 1990.

\bibitem{PS92} 
Vesselin Petkov, Luchezar Stoyanov. Geometry of reflecting rays and inverse spectral problems, John Wiley \& Sons, Ltd., Chichester, 1992, and Geometry of the generalized geodesic flow and inverse spectral problems, 2nd ed., John Wiley \& Sons, Ltd., Chichester, 2017.

\bibitem{Sib04} 
Karl Friedrich Siburg. The principle of least action in geometry and dynamics, volume 1844 of Lecture Notes in Mathematics. Springer-Verlag, Berlin, 2004.

\bibitem{SuMa03} 
Endre Suli and David F. Mayers. An introduction to numerical analysis. Cambridge University Press, Cambridge, 2003.

\bibitem{Sun85}
T.  Sunada.  Riemannian  coverings  and  isospectral  manifolds, 
Ann.of  Math.121(1985), 169--186.


\bibitem{Stowe86} 
Dennis Stowe. 
Linearization in two dimensions. J. Differential Equations, 63(2):183--226, 1986.

\bibitem{St03} 
Luchezar Stoyanov. A sharp asymptotic for the lengths of certain scattering rays in the exterior of two convex domains, Asymptotic Analysis, 35(3, 4), 235--255.

\bibitem{Zel00}
Steve Zelditch. Spectral determination of analytic bi-axisymmetric plane domains, 
Geom. Funct. Anal. 10 (2000), no. 3, pp. 628--677.

\bibitem{Zel04} 
Steve Zelditch. Inverse spectral problem for analytic domains, I. Balian- Bloch trace formula, Comm. Math. Phys. 248 (2004), no. 2, 357--407.

\bibitem{Zel09} 
Steve Zelditch. Inverse spectral problem for analytic domains II: domains with one symmetry, Annals of Mathematics (2) 170 (2009), no. 1, 205--269.

\bibitem{ZZ14}  
Wenmeng Zhang and Weinian Zhang. Sharpness for $C^1$ linearization of planar hyperbolic diffeomorphisms. J. Differential Equations, 257(12):4470--4502, 2014.



\end{thebibliography}

\end{document}